\documentclass{article} 
\usepackage[utf8]{inputenc}
\usepackage[english]{babel} 
\usepackage[T1]{fontenc}
\usepackage{natbib}
\usepackage{xcolor}
\usepackage{amsfonts} 
\usepackage{amsthm}
\usepackage{amsmath}
\usepackage{amssymb} 
\usepackage{mathrsfs}
\usepackage{graphicx}  
\usepackage{pgf,tikz,pgfplots}
\usepackage[all]{xy}
\usepackage[top=4cm, bottom=4cm, left=4cm, right=4cm]{geometry}
\pgfplotsset{compat=1.18}

\newcommand{\bigtimes}{\mathop{\vcenter{\hbox{\Large$\times$}}}\displaylimits}

\usepackage{comment}
\newtheorem{lem}{Lemma}[section]
\newtheorem{prop}{Proposition}[section]
\newtheorem{theo}{Theorem}[section]
\newtheorem{coro}{Corollary}[section]
\newtheorem*{thm}{Theorem}
\newtheorem{conj}{Conjecture}

\theoremstyle{definition}
\newtheorem{definition}{Definition}[section]
\newtheorem{ex}{Example}[section]
\newtheorem{rem}{Remark}[section]

\newcommand{\Irr}[1]{\mathrm{Irr} (#1)}

\newcommand{\quest}[1]{\textcolor{red}{\underline{Question:} #1} }

\newcommand{\GL}[1]{\mathrm{GL}_#1}
\newcommand{\revdots}{\mathinner
{\mkern1mu\raise1pt\vbox{\kern7pt\hbox{.}}
\mkern2mu\raise4pt\hbox{.}
\mkern2mu\raise7pt\hbox{.}\mkern1mu}}

\title{A trace Paley-Wiener theorem for $\GL{n}(F)\backslash \GL{n}(E)$, with $E/F$ a quadratic extension of non-Archimedean local fields}
\author{Juliette Coutens}
\date{October 9, 2026}

\begin{document}
\maketitle

\tableofcontents

\pagebreak

\section{Introduction}

In general, the study of relative characters and the associated Paley–Wiener theorems arises naturally within the framework of the relative Langlands program. These objects appear not only on the spectral side of relative trace formulas, but also in the conjectures of Sakellaridis–Venkatesh, which connect automorphic periods with $L$-functions \cite{SV}. The overarching expectation is that the poles of relative characters are closely related to the poles of the local $L$-factors attached to the underlying spherical variety.

We focus here on a particular symmetric space. The main result of this paper is a generalization to a specific symmetric space of the trace Paley--Wiener theorem for reductive $p$-adic groups of Bernstein, Deligne, and Kazhdan \cite{BDK}. In the general relative setting, Delorme, Harinck, and Sakellaridis established a matrix Paley--Wiener theorem for the Schwartz space of spherical varieties \cite{DHS}. Let $E/F$ be a quadratic extension of local non-Archimedean fields. In this work, we prove a scalar Paley--Wiener theorem for relative Bessel distributions on the symmetric space $\mathrm{GL}_n(F)\backslash \mathrm{GL}_n(E)$. These distributions are relative characters of the form
\[
\pi \mapsto I_\pi(f), \qquad f\in C_c^\infty(\mathrm{GL}_n(E)),
\]
as $\pi$ ranges over the $\mathrm{GL}_n(F)$-distinguished, irreducible tempered representations of $\mathrm{GL}_n(E)$. They are constructed from a $\mathrm{GL}_n(F)$-invariant functional together with a Whittaker functional. As in the group case, the proof relies fundamentally on the theory of the Bernstein center. Another key ingredient is the Plancherel decomposition of $\mathrm{GL}_n(F)\backslash \mathrm{GL}_n(E)$ \cite{RBPPlanch}, together with the local Langlands correspondence and the base-change map arising from the quasi-split unitary group of rank $n$,
\[
U(n)=\{g\in \mathrm{GL}_n(E)\mid {}^tg^c w_n g = w_n\},
\]
where $w_n$ is the antidiagonal matrix with $1$'s on the antidiagonal and $0$'s elsewhere.

Let $H=\mathrm{GL}_n(F)$ and $G=\mathrm{GL}_n(E)$. We write $\mathrm{Temp}_H(G)$ for the set of isomorphism classes of tempered, irreducible representations of $G$ that are $H$-distinguished, i.e. those admitting a nonzero $H$-invariant linear form. Likewise, we denote by $\mathrm{Temp_{gen}}(U(n))$ the set of tempered, irreducible, generic representations of $U(n)$ for a fixed Whittaker datum. Let $\Omega_{\mathrm{gen}}(U(n))$ be the variety of cuspidal supports of irreducible generic representations of $U(n)$, and let $\mathfrak{Z}_{\mathrm{gen}}(U(n))$ denote its algebra of regular functions, supported on finitely many connected components. The base-change map
\[
BC_n : \mathrm{Temp}(U(n)) \longrightarrow \mathrm{Temp}(G),
\]
defined to be \emph{stable} when $n$ is odd and \emph{unstable} when $n$ is even, induces a bijection between $\mathrm{Temp_{gen}}(U(n))$ and $\mathrm{Temp}_H(G)$.

The main theorem of this paper (Theorem~\ref{mainT}) can now be stated as follows.

\begin{thm}
    The base-change map induces an isomorphism
    \[
    \bigl\{\pi \in \mathrm{Temp}_H(G) \mapsto I_\pi(f),\ f\in C_c^\infty(G)\bigr\}
    \;\cong\; 
    \mathfrak{Z}_{\mathrm{gen}}(U(n)).
    \]
    Equivalently,
    \[
    \bigl\{\rho \in \mathrm{Temp_{gen}}(U(n)) \mapsto I_{BC_n(\rho)}(f),\ f\in C_c^\infty(G)\bigr\}
    \;=\;
    \mathfrak{Z}_{\mathrm{gen}}(U(n)).
    \]
\end{thm}

In particular, this result provides a spectral local transfer for the relative Kuznetsov trace formula and may have applications to the Lapid--Mao conjecture on Fourier--Whittaker coefficients \cite{LapidMao}. In the spirit of the Ichino--Ikeda conjecture for Bessel periods, Lapid and Mao predict an explicit formula expressing the square of a global Whittaker period of a cuspidal automorphic representation in terms of a special value of an $L$-function, multiplied by a product of local factors attached to each place --- these local factors being expressed in terms of local zeta integrals, or equivalently in terms of \emph{local relative characters}. 

The relative trace formula allows one to compare a spectral expansion, built from local relative characters such as the $I_\pi$ appearing in Theorem~\ref{mainT}, with a geometric expansion in terms of orbital integrals on the symmetric space. Our trace Paley--Wiener theorem supplies the kind of harmonic-analytic input such a comparison requires: it identifies the space of relative characters $I_\pi(f)$, for $\pi$ ranging over the tempered spectrum of $\mathrm{GL}_n(F)\backslash \mathrm{GL}_n(E)$ and $f \in C_c^\infty(G)$, with the generic Bernstein center of $U(n)$ via base change.

We conclude with an outline of the paper. General notation is fixed in Section~2, and Section~3 recalls the notions and results from the literature required for the rest of the paper. Section~\ref{cuspBC} studies the behavior of cuspidal supports under base change: we establish a closed immersion
\[
\Omega_{\mathrm{gen}}(U(n)) \hookrightarrow \Omega(\mathrm{GL}_n(E))
\]
arising from the base-change (Proposition~\ref{closedim}), and we show that the base change of an irreducible, generic, supercuspidal representation of $U(n)$ is induced by irreducible and supercuspidal representations in $\mathrm{GL}_n(E)$ (Proposition~\ref{propcuspBC}).

Section~\ref{FamRep} develops the algebraic families of representations needed to extend several classical tools, through three main ingredients: a family version of Jacquet functionals (Part~\ref{famJacquet}), a family version of the Gelfand--Kazhdan theorem \cite{GelfKaz} (Part~\ref{GK}), and a meromorphic continuation of the functional $\beta$ together with a multiplicity-one result for induced representations (Propositions~\ref{BetaFam} and \ref{BetaProportional}, in Part~\ref{PartBetaFam}).

Finally, Section~\ref{PW} assembles these ingredients: we define the relative characters $I_\pi$, establish their meromorphic continuation, and prove the trace Paley--Wiener theorem (Theorem~\ref{mainT}).


\section{Notation}

\begin{itemize}
    \item For any non Archimedean local field $E$, we denote $k_E$ its residue field with cardinal $q_E$.
    \item For any integer $k$, we will denote $G_k$ for $\mathrm{GL}_k$.
    \item Let $F$ be a non Archimedean local field. We say that a group $G$ is a \emph{$p$-adic group}, if $G$ is the group of $F$-points of a connected reductive group defined over $F$.
    \item For $E/F$ a quadratic extension of non Archimedean local fields, we denote $c$ the non trivial element of $\mathrm{Gal}(E/F)$. Let $N_{E/F}: E \to F$ be the norm application (here $N_{E/F}(g)=gg^c$). And, let $\eta_{E/F}:F^\times \to \mathbb{C}^\times$ be the quadratic character, which is trivial on $N_{E/F}(E^\times)$ and equals to $-1$ elsewhere. Let $\eta':E^\times \to \mathbb{C}^\times$ be an extension of $\eta_{E/F}$ to $E^\times$.
    \item Let $\eta'_n$ be a character of $G_n(E)$, defined by
    
    $$\eta'_n(g) = \left\{\begin{array}{cc}
         \eta'(\det g), & \ n \ \mathrm{even},  \\
         1, & \ n \ \mathrm{odd}.
    \end{array} \right.$$
    \item We fix a non trivial additive character $\psi$ of $E$ with $\psi$ trivial on $F$. 
    \item  All groups $G$ are supposed to be defined over $F$ and if we note $G(E)$, we mean $G(E) = (\mathrm{Res}_{E/F}G_{E})(F)$. 
    \item Let $Q=MU$ be a parabolic subgroup of $G_n$, with $M$ a Levi subgroup and $U$ the unipotent radical of $Q$. Then there exists $n_1, \dots, n_k$ a partition of $n$, such that $M$ is of the form 
    $$M = G_{n_1}\times \cdots \times G_{n_k}.$$
    \item Let $P_n$ denote the mirabolic subgroup of $G_{n}$, that is the subgroup of matrices of the form
    $$\begin{pmatrix} * & \cdots & * & * \\
                  \vdots & \ddots & \vdots & \vdots \\
                  * & \cdots & * & * \\
                  0 & \cdots & 0 & 1 \end{pmatrix}$$
    \item Let $B_n = A_nN_n$ be the Borel subgroup of upper triangular matrices of $G_n$, with $A_n$ the torus of diagonal matrices and $N_n$ the unipotent radical. Let $\Delta_0$ be the set of simple roots of $A_n$ in $B_n$.
    \item Let $\pi$ be a representation of any group $G$. For $g\in G$, we denote $\pi^g$ the representation defined by $\pi^g(h)=\pi(ghg^{-1})$.
    \item For all non Archimedean local field $F$, let $W_F$ be the Weil group of $F$. We denote $W_F'=W_F\times \mathrm{SL}_2(\mathbb{C})$ the associated Weil-Deligne group.
    \item For $E/F$ a quadratic extension of non Archimedean local fields, and any representation $\pi$ of $G_n(E)$, we denote by $\pi^\vee$ the \emph{contragredient} (or smooth dual) representation of $\pi$, by $\pi^c$ the composition with the automorphism of $G_n(E)$ induced by the conjugation $c$, and we define $\pi^*=(\pi^\vee)^c$.
    \item Let $G$ be a $p$-adic reductive group, for any Levi subgroup $M$ of $G$ and $\sigma$ an irreducible and supercuspidal representation of $M$, denote by $[M, \sigma]$ the $G$-conjugacy class of the pair $(M, \sigma)$. Let $\Omega(G)$ be the set of such \emph{cuspidal pairs} $[M, \sigma]$. 
    \item For a $p$-adic reductive group $G$, denote by $\mathrm{Rep}(G)$ the category of smooth representations of $G$ with complex coefficients.
\end{itemize}

\section{Preliminaries}
\subsection{Groups of characters} 
Let $F$ be a non Archimedean local field, let $v_F$ be its valuation and $\vert \cdot\vert_F$ the corresponding absolute value. Let $G=\mathbf{G}(F)$ be the group of $F$-points of an algebraic connected reductive group $\mathbf{G}$, defined over $F$. Let $Z_G$ be the center of G and $A_G$ the maximal split torus in $Z_G$. Let $X^*(G)$ be the group of rational characters of $G$ defined over $F$, and set 

$$\mathfrak{a}_G^* = X^*(A_G)\otimes_{\mathbb{Z}}\mathbb{R}= X^*(G)\otimes_{\mathbb{Z}}\mathbb{R}, \ \mathfrak{a}_{G, \mathbb{C}}^*=\mathfrak{a}_G^*\otimes_{\mathbb{R}} \mathbb{C},$$
and 

$$\mathfrak{a}_G = \mathrm{Hom} (X^*(G), \mathbb{R}), \ \mathfrak{a}_{G, \mathbb{C}}= \mathfrak{a}_G \otimes_{\mathbb{R}}\mathbb{C}.$$
There is natural pairing between $\mathfrak{a}_G^*$ and $\mathfrak{a}_G$ denoted $\langle \cdot, \cdot \rangle$, from which we define a morphism $H_G:G\to \mathfrak{a}_G$ by $\langle \lambda, H_G(g)\rangle = v_F(\lambda(g))$ for $\lambda \in X^*(G)$ and $g\in G$. The kernel of $H_G$ is denoted by 
$$^0G:= \bigcap_{\lambda \in X^*(G)} \ker \vert \lambda \vert_F$$
and it contains all the compact subgroups of $G$. The group of characters from $G$ to $\mathbb{C^\times}$ that are trivial on $^0G$, denoted by $\mathcal{X}(G)$, is called the group of \emph{unramified characters} of $G$ and it has in fact a structure of a complex torus. To any $\lambda\in \mathfrak{a}_{G, \mathbb{C}}^*$ we associate a smooth unramified character
$$\begin{array}{cccc}
     & G/{}^0G & \longrightarrow & \mathbb{C}^\times \\
     & g & \mapsto & q_F^{-\langle \lambda, H_G(g)\rangle}
\end{array}.$$
We will usually say "let $Q=MU$ be a parabolic subgroup of $G$", meaning that we fix a Levi decomposition of the parabolic subgroup $Q$, with $M$ a Levi factor and $U$ the unipotent radical of $Q$. For every Levi subgroup $M$, we denote $W_G(M)=\mathrm{Norm}_G(M)/M$ its associated Weyl group. If $Q=MU$ is a parabolic subgroup of $G$, let $\delta_Q$ be its modular character, and let $H_Q:Q\to \mathfrak{a}_M$ be the composition of $H_M$ with the projection from $Q$ to $M$. If we fix a maximal and special compact subgroup $K_0$ of $G$, then we have the Iwasawa decomposition $G=QK_0$ and we can extend $H_Q$ to G by setting $H_Q(muk)=H_M(m)$, for all $(m,u,k)\in M\times U \times K_0$. 

\subsection{Subgroups of $\mathrm{GL}_n$}

For every integer $n\geq 1$, we denote $G_n$ for the algebraic group $\mathrm{GL}_n$ defined over $F$, and $Z_n$ its center. Let $B_n = A_nN_n$ be the Borel subgroup of upper triangular matrices of $G_n$, with $A_n$ the maximal torus of diagonal matrices and $N_n$ its unipotent radical. For $g\in G_n$, we denote ${}^tg$ its transpose and $g_{i,j}$, $1 \leq i,j\leq n$ its entries. We always consider $G_n$ as a subgroup of $G_{n+1}$, through the embedding $g\in G_n \mapsto \begin{pmatrix}
    g & \\ & 1
\end{pmatrix} \in G_{n+1}$.  Let $P_n$ denote the mirabolic subgroup of $G_{n}$, that is the subgroup of matrices of the form
    $$\begin{pmatrix} * & \cdots & * & * \\
                  \vdots & \ddots & \vdots & \vdots \\
                  * & \cdots & * & * \\
                  0 & \cdots & 0 & 1 \end{pmatrix}$$ 
and $U_n$ be the unipotent radical of $P_n$, namely the matrices of the form 
    $$\begin{pmatrix} 1 & \cdots & 0 & * \\
                  \vdots & \ddots & \vdots & \vdots \\
                  0 & \cdots & 1 & * \\
                  0 & \cdots & 0 & 1 \end{pmatrix}$$
such that $P_n =G_{n-1}U_n$. Let $Q=MU$ be a parabolic subgroup of $G_n$, with $M$ a Levi subgroup and $U=\mathrm{Rad}_{U}(Q)$. Then there exists $n_1, \dots, n_k$ a partition of $n$, such that $M$ is of the form 
$$M = G_{n_1}\times \cdots \times G_{n_k}.$$
The unramified characters of $G_n(F)$ are of the form $\vert \det \vert_F^\lambda$, with $\lambda\in \mathbb{C}$, hence for the Levi subgroup $M$ they are of the form
$$\vert \det\vert_F^{\lambda_1} \boxtimes \cdots \boxtimes \vert \det\vert_F^{\lambda_k}.$$
By fixing a uniformizer $\varpi_F$ of $F$, we get an isomorphims
$$\begin{array}{cccc}

     & \mathcal{X}(G_{n_i}) & \longrightarrow & \mathbb{C}^\times \\
     &  \vert \det \vert ^\lambda& \longmapsto & \vert \varpi_F\vert^\lambda = q_F^{-\lambda}
\end{array}$$
that does not depend on the choice of $\varpi_F$, so we may identify $\mathcal{X}(M)$ with $(\mathbb{C}^{\times})^k$.

For a standard parabolic subgroup $Q=MU$, we denote by $\Delta_0\backslash \Delta_0^M$ the set of simple roots of $A_n$ corresponding to the Levi $M$, and by $\Delta_0^Q$ the set of $\alpha\vert_{A_M}$, with $\alpha \in \Delta_0\backslash \Delta_0^M$ and $A_M$ being the central split component of $M$.

\subsection{Unitary groups}
Let $E/F$ be a quadratic extension of non Archimedean local fields. We call $V$ an \emph{Hermitian space} if it is a finite dimensional $E$-vector space equipped with a non-degenerate Hermitian form $h$ linear on the first variable. Then the associated unitary group, seen as an algebraic group defined of $F$ is defined by

$$U(V)=\{g \in \mathrm{GL}(V), \ h(gx,gy)=h(x,y) \ \forall x, y \in V\}.$$
If $e$ is a basis of $V$, we define the \emph{discriminant} of $h$, $\mathrm{disc}(h)=\det (\mathrm{Mat}_e(h))$ and we have 

$$\mathrm{disc}(h) \in F^\times/N_{E/F}(E^\times).$$
In our case, the discriminant and dimension determine entirely the Hermitian form $h$, up to isomorphism. By the class field theory we have $F^\times/N_{E/F}(E^\times)\cong \mathbb{Z}/2\mathbb{Z}$, which means that we have two different choices for $h$ for a fixed dimension. For all $n \geq 1$, we denote $U(n)$ the quasi-split unitary group of rank $n$ associated to the Hermitian form on $E^n$ associated to the following matrix

$$w_n = \begin{pmatrix} & & 1 \\
                        & \revdots & \\
                        1 & & \end{pmatrix}.$$
More precisely, for every $F$-algebra $R$

$$U(n)(R)=\{g\in G_n(R\otimes_F E), ^tg^cw_ng=w_n\}.$$
By abuse of notations, we will denote $U(n)$ instead of $U(n)(F)$ when there is no ambiguity. 

Every Levi subgroup $M$ of a unitary group $U(V)$ is isomorphic to 

$$\mathrm{Res}_{E/F}G_{n_1, E} \times \cdots \times \mathrm{Res}_{E/F}G_{n_k, E} \times U(W),$$
where $W$ is a non-degenerate subspace of $V$, such that $\dim(V)=\dim(W) +2n_1 + \cdots +2n_k$.

\subsection{Representations}

Let $G=\mathbf{G}(F)$ be the group of $F$-points of an algebraic connected reductive group $\mathbf{G}$, defined over $F$. We only consider smooth representations on complex vector spaces (here \emph{smooth} means that every vector has an open stabilizer). For every (smooth) representation $(\pi, V)$ (with $V$ a complex vector space), we will usually use either the notation $\pi$ or $V$ instead of $(\pi, V)$.

Let $\pi$ be a representation of $G$. We denote by $\pi^\vee$ the contragredient representation, and by $\langle\cdot, \cdot\rangle$ the natural pairing between $\pi$ and $\pi^\vee$. If $\chi: G \to \mathbb{C}^\times$ is a smooth character, we denote $\pi \otimes \chi$ the twist of $\pi$ by $\chi$. If $\lambda\in \mathfrak{a}_{G, \mathbb{C}}^*$, let $\pi_\lambda$ denote the twist of $\pi$ by the character $g \mapsto q_F^{-\langle \lambda, H_G(g)\rangle}$. When $G=G_n(F)$ we will denote $\pi_x = \pi \otimes \vert \det\vert^x$, for $x\in \mathbb{C}$.

If $G_1, G_2$ are two reductive groups over $F$ and $\pi_1, \pi_2$ are representations of $G_1(F)$ and $G_2(F)$ respectively, we denote by $\pi_1 \boxtimes \pi_2$ the tensor product representation of $G_1(F) \times G_2(F)$. 

For $\pi$ a representation of finite length of $G$, let $\mathrm{End}^\infty(\pi)$ be the space of endomorphisms $T: \pi \to \pi$ that are biinvariant by a compact-open subgroup. The natural map $\pi \otimes \pi^\vee \to \mathrm{End}^\infty(\pi)$ is an isomorphism of $G\times G$-representations. Hence, the canonical pairing $\pi \otimes \pi^\vee \to \mathbb{C}$ gives rise to a map from $\mathrm{End}^\infty(\pi)$ to $\mathbb{C}$, given by the trace $T \mapsto \mathrm{Trace}(T)$. For every $f \in C_c^\infty(G)$, we define an operator $\pi(f)\in \mathrm{End}^\infty(\pi)$ by fixing a Haar measure $dg$ on $G$
$$\pi(f)v=\int_G f(g)\pi(g)vdg.$$

Let $Q=MU$ be a parabolic subgroup of $G$ an $A_M$ be the central maximal split torus of $M$. For $(\pi, V)$ a representation of $G$, we denote $r_Q^G(\pi, V)$ (or $r_Q^G(\pi)$ or $r_Q^G(V)$) the normalized parabolic restriction. If $\chi$ is any smooth character of $A_M$ and $l\in \mathbb{N}$, set 
$$(r_Q^GV)_{\chi, l} = \{v \in (r_Q^GV): (r_Q^G(\pi)(a_1)-\chi(a_1))\dots(r_Q^G(\pi)(a_l)-\chi(a_l))v=0, \ \forall a_1, \dots, a_l\in A_M \}$$
and 
$$(r_Q^GV)_{\chi} =\bigcup_{l=0}^\infty (r_Q^GV)_{\chi, l}$$
for the generalized $\chi$-eigenspace. If $(r_Q^GV)_{\chi}$ is non-zero, then $\chi$ is called a \emph{normalized exponent} of $\pi$ associated to $Q$, and they form a set denoted $\mathrm{Exp}_Q(\pi)$. If we compose such an exponent $\chi$ with the complex module we get a character $\vert \chi \vert : A_M \to \mathbb{R}_{>0}$ from which we define $\Re(\chi)=\ln \vert \chi \vert \in \mathfrak{a}_M^*$.
For $\sigma$ a representation of $M$ we use the notation $i_Q^G(\sigma)$ for the normalized parabolic induction, it is the space of smooth functions $e:G\to \sigma$ such that for all $(m,u,g)\in M\times U \times G$, $e(mug)=\delta_Q^{1/2}(m)\sigma(m)e(g)$. When $G=G_n(F)$ then $M$ is of the form $M = G_{n_1}(F)\times \cdots \times G_{n_k}(F)$ and in that case we write 

$$\sigma_1 \times \cdots \times \sigma_k = i_Q^G(\sigma_1 \boxtimes \cdots \boxtimes \sigma_k),$$
with $\sigma_i$ a representation of $G_{n_i}(F)$, for all $1 \leq i \leq k$. We will use the same notation for the unitary group case $U(V)$, where $V$ is an Hermitian space with respect to the quadratic extention $E/F$. For a Levi subgroup 
$$M=G_{n_1}(E) \times \cdots \times G_{n_k}(E) \times U(W)(F),$$
where $W$ is a non-degenerate subspace of $V$, we denote 
$$\sigma_1\times \cdots \times \sigma_k \rtimes \sigma_W = i_Q^G(\sigma_1 \boxtimes \cdots \boxtimes \sigma_k\boxtimes \sigma_W),$$
with $\sigma_i$ a representation of $G_{n_i}(E)$, for all $1 \leq i \leq k$ and $\sigma_W$ a representation of $U(W)(F)$.

Let $\mathrm{Irr}(G)$ be the set of all isomorphism classes of irreducible representations of $G$. For all $\pi\in \mathrm{Irr}(G)$, we denote $\omega_\pi:Z_G \to \mathbb{C}^\times$ its central character. Let $\mathrm{Temp}(G) \subset \Irr{G}$ be the subset of irreducible tempered representations, and let $\Pi_2(G) \subset \mathrm{Temp}(G)$ be the subset of irreducible square-integrable representations (also called \emph{discrete series}). At some point we will need the following characterization of tempered representations. A representation $\pi$ is tempered if and only if for all parabolic subgroups $Q=MU$, and all normalized exponents $\chi\in \mathrm{Exp}_Q(\pi)$ we can write 
$$\Re(\chi) = \sum_{\alpha\in \Delta_Q} c_\alpha \alpha, \ c_\alpha \geq 0$$
where $\Delta_Q$ is the set of simple roots of $A_M$ in $Q$. 

\begin{lem}\label{expUnit}
      Let $G= G_d(F)$ and $\pi$ be a representation of $G$, parabolically induced from a unitary supercuspidal representation. Then, for every parabolic subgroup $P$ of $G$, the exponents of $r_P^{G}(\pi)$ are all unitary.
\end{lem}

\begin{proof}
    Every parabolic is supposed to be standard. Let $Q=MU_Q$ be a parabolic subgroup of $G_d(F)$ and $\sigma$ be a unitary supercuspidal representation of $M$, denote by $\omega_\sigma$ its central (unitary) character. We set $\pi=i_Q^G(\sigma)$. Recall that, due to the geometric lemma, for all parabolic subgroup $P=LU_P$ the representation $r_P^G(\pi)$ admits a filtration, with associated graded vector space, on which $L$ acts, is isomorphic to  
    $$\bigoplus_w i_{L\cap w^{-1}Q}^{L} (r^{M}_{M\cap wP} \sigma)^w$$
    where $w$ runs over representatives of the double cosets $Q\backslash G/P$. Since $\sigma$ is supercuspidal, the only $w$ that contribute to this direct sum are the ones that verify $M\subseteq wP $. In particular, we get that for all $w$, $Z_L\subseteq Z_{w^{-1}M} \cong Z_{M}$. Hence
    $$\bigoplus_w i_{L\cap w^{-1}Q}^{L} ( \sigma^w).$$
    the exponents of $r_P^G(\pi)$ are the central characters of the various $ i_{L\cap w^{-1}Q}^{L} ( \sigma^w)$, which are the restrictions to $Z_L$ of the central characters of $\sigma^w$, that are unitary by assumption.
\end{proof}

Let $H$ be a subgroup of $G$, a representation $\pi$ of $G$ is called $H$-\emph{distinguished} if there exists a non-zero $H$-invariant linear from $\pi$ to $\mathbb{C}$. We denote by $\mathrm{Temp}_H(G)$ the subset of tempered $H$-distinguished representations of $G$. 

Suppose $\mathbf{G}$ is quasi-split. Let $B=AN$ be a Borel subgroup of $G$, and $\psi_N:N\to \mathbb{C}^\times$ be a generic character of $N$ (i.e. its stabilizer in $A$ is the center of $G$). A representation $\pi$ is $(N, \psi_N)$-\emph{generic} (or just \emph{generic} when there is no ambiguity), if there exists a non-trivial $N$-equivariant map form $\pi$ to $\psi_N$. We denote $\mathrm{Temp_{gen}}(G)$ the set of tempered representations that are $(N, \psi_N)$-generic.

\subsection{Whittaker models}\label{WhittSection}
Let $F$ be a local non Archimedean field. In this section we write $B_n = A_n N_n$ instead of $B_n(F)= A_n(F)N_n(F)$, the Borel subgroup of upper triangular matrices of $G_n=G_n(F)$. Let $\psi: F \to \mathbb{C}^\times$ be a smooth character, and let $\psi_n$ be the generic character from $N_n$ to $\mathbb{C}^\times$, defined by 
$$\psi_n \begin{pmatrix}
        1 & x_1 &  * & \dots  & * \\
          & \ddots & \ddots& \ddots  & \vdots \\
          & & \ddots & \ddots & * \\
          &  &  &\ddots &   x_{n-1} \\
          & & & & 1
    \end{pmatrix} = \psi(\sum_{i=1}^{n-1}x_i).$$
Recall that a \emph{Whittaker function} is a function $f\in C^\infty(N_n\backslash G_n, \psi_n)$, which means that $f$ is a smooth function such that for all $u\in N_n$, and all $g\in G_n$
$$f(ug)=\psi_n(u)f(g).$$
We consider $G_n$ acting on the right on $C^\infty(N_n\backslash G_n, \psi_n)$. Recall that a representation $\pi$ of $G_n$ is \emph{generic} if it admits non trivial Whittaker functional in $\mathrm{Hom}_{N_n}(\pi, \psi_n)$. If $\pi$ is generic and irreducible, then $\pi$ is isomorphic to its \emph{Whittaker model} $\mathcal{W}(\pi, \psi_n)\subset C^\infty(N_n\backslash G_n, \psi_n)$. We recall the explicit formula given in \cite{LMAsymp}, Theorem 3.1. 

\begin{theo}\label{asympW}
    Let $\pi$ be a generic representation of finite length of $G_n$. Then there exists $l\in \mathbb{N}$ such that any $W \in \mathcal{W}(\pi, \psi_n)$ can be written as 
    $$W(uak) = \psi_n(u) \sum_{B_n\subseteq Q}\delta^{1/2}_Q(a)\sum_{\mu \in \mathrm{Exp}_Q(\pi)} \phi_{Q, \mu}(a, k), \ a \in A_n, \ u \in N_n, k\in K_0$$
    where for all $k\in K_0$, $\phi_{Q, \mu}(\cdot, k)$ is a finite sum of functions of the form $\xi.\varphi$ where $\xi \in \mathcal{F}(A_n)_{\chi,l}$, with
    $$\mathcal{F}(A_n)_{\chi,l}=\{f: A \to \mathbb{C},  \ (R(a_1)-\chi(a_1))\cdots (R(a_l)-\chi(a_l))f=0, \ \forall a_1, \dots,a_l\in A\},$$
    and $\varphi$ is a smooth function with compact support on $F^{\Delta_0\backslash \Delta_0^Q}\times (F^\times)^{\Delta_0^Q}$, and $\phi_{Q, \mu}$ is invariant under a compact open subgroup of $K_0$. Note that for the function $\varphi$ we have identified $A_n/Z_n$ with $(F^\times)^{\Delta_0}$ through the application $a\mapsto (\alpha(a))_{\alpha\in \Delta_0}$, and $\Delta_0^Q$ is the set of simple roots of $A_n$ in $M_Q$. 
\end{theo}

The following definition is due to Lapid-Mao, \cite{LapidMao} 2.1.
\begin{definition}[Regularized integral]\label{RegInt}
    We say that $f\in C^\infty(N_n)$ admits a \emph{regularized integral}, if there exists a compact-open subgroup $K_1$ of $N_n$, such that for all compact-open subgroup $K_2$ of $N_n$ containing $K_1$, we have 
    $$\int_{K_1}f(u)du = \int_{K_2}f(u)du.$$
    In that case, we denote by 
    $$\int^*_{N_n} f(u)du = \int_{K_1}f(u)du$$
    the \emph{regularized integral} of $f$.
\end{definition}

\begin{definition}\label{Wfpi}
    Let $Q=MU$ be a parabolic subgroup of $G=G_n$ and $\sigma \in \Irr{M}$ be square-integrable, set $\pi=i_Q^G(\sigma)$. For all $f\in C_c^\infty(G)$ we define
    \begin{enumerate}
        \item a function $f_\pi \in C^\infty(G)$, such that for all $g \in G$
        $$f_\pi(g):= \mathrm{Trace}(\pi(g)\pi(f^\vee)).$$
        \item a function $W_f\in C^\infty(N_n\backslash G\times N_n\backslash G, \psi_n^{-1}\otimes \psi_n)$, such that for all $(g_1, g_2)\in G\times G$
        $$W_f(g_1, g_2)= \int_N f(g_1^{-1}ug_2)\psi_n(u)^{-1} du.$$
    \end{enumerate}
\end{definition}
These definitions and Proposition 2.3, and 2.11 of \cite{LapidMao} allow us to define 
$$W_{f, \pi}(g_1, g_2):= W_{f_\pi}(g_1, g_2)= \int_N^* f_\pi(g_1^{-1}ug_2)\psi_n(u)^{-1} du,$$
for all $(g_1, g_2)\in G\times G$. At some point we will also use the function
$$W_{f, \pi}^1 := W_{f^\vee, \pi}(1, \cdot) $$
and since $\pi$ is a tempered representation of $G_n$, it is a generic representation, and then $W_{f, \pi}^1\in \mathcal{W}(\pi, \psi_n)$.

\begin{prop}\label{Wpi}
    Let $Q=MU$ be a parabolic subgroup of $G_n$, $\sigma\in \Irr{M}$ be supercuspidal, set $\pi = i_Q^{G_n}(\sigma)$. Let $\pi_+$ denote the unique $\psi_n$-generic irreducible subquotient of $\pi$. Then 
    $$W_{f,\pi}^1=W_{f, \pi_+}^1 \in \mathcal{W}(\pi_+, \psi_n).$$
\end{prop}
\begin{proof}
    It follows from Proposition 2.3 in \cite{LapidMao} and the fact that if $(\pi, V)$ is an admissible representation of $G_n$, admitting a filtration by subrepresentations
    $$0=V_0 \subseteq V_1 \subseteq \cdots \subseteq V_m = V,$$
    we denote $(\sigma_i, V_{i+1}/V_i)$ for all $i\in \{1,\dots, m-1 \}$, then for all $f\in C_c^\infty(G_n(E))$, 
    $$\mathrm{Trace}(\pi(f))=\sum_{i=0}^{m-1}\mathrm{Trace}(\sigma_i(f)).$$
\end{proof}

\subsection{Bernstein center}

Let $G$ be a $p$-adic group. It is known that for every irreducible representation $\pi$ of $G$ there exists $Q=MU$ and $\sigma\in \Irr{M}$ supercuspidal, such that $\pi$ is a subquotient of $i_Q^G\sigma$ (\cite{Ren}, VI.2.1). Moreover, the pair $(M,\sigma)$ is unique up to $G$-conjugacy. Let $\Omega(G)$ be the set of conjugacy classes of pairs $(M, \sigma)$ where $M$ is a Levi subgroup and $\sigma$ an irreducible supercuspidal representation of $M$. We denote by $[M, \sigma]_G$, or $[M, \sigma]$, the conjugacy class of the pair $(M, \sigma)$. This allows us to define a finite-to-one map
$$\begin{array}{cccc}
    \mathbf{Cs}: & \Irr{G} & \longrightarrow & \Omega (G)  \\
     & \pi & \longmapsto & [M, \sigma]_G
\end{array}$$
where "$\mathbf{Cs}$" stands for \emph{cuspidal support}. For all $[M, \sigma] \in \Omega(G)$, the image of the map 
$$\begin{array}{cccc}
     & \mathcal{X}(M) & \longrightarrow & \Omega (G)  \\
     & \chi & \longmapsto & [M, \sigma \otimes \chi]
\end{array}$$
will be denoted $\Omega_\sigma$ and called a \emph{connected component} of $\Omega(G)$. Actually, the theorem VI.7.1 in \cite{Ren} gives a description of these connected components. For all irreducible supercuspidal representation $\sigma$ of a Levi subgroup $M$ of $G$, we denote by $W(M,\sigma)$ the stabilizer of $\sigma$ in $W_G(M)\times \mathcal{X}(M)$. Then the above map induces a bijection 
$$\Omega_\sigma \cong \mathcal{X}(M)/W(M,\sigma).$$
Therefore, the connected components of $\Omega(G)$, and hence $\Omega(G)$ itself, inherit the structure of complex algebraic variety from the tori $\mathcal{X}(M)$.

The \emph{Bernstein center} is the algebra of polynomial (or regular) functions on $\Omega(G)$ denoted $\Hat{\mathfrak{Z}}(G)$, and $\mathfrak{Z}(G)\subset \Hat{\mathfrak{Z}}(G)$ be the subalgebra of functions with support in a finite number of connected components. The name "Bernstein center" refers to the fact that $\Hat{\mathfrak{Z}}(G)$ is the center of the category of smooth representations of $G$ (see \cite{Ren}, VI.10, or \cite{BDKV} 2.13).

\begin{definition}
    A function $I : \Irr{G}\to \mathbb{C}$ is \emph{regular} if there exists a regular function $I'$ on $\Omega(G)$, such that $I= I'\circ \mathbf{Cs}$.
    
\end{definition}


If $G$ is quasi-split, we can fix a Whittaker datum $(N, \psi)$ and consider $\Omega_{\mathrm{gen}}(G)$ the set of cuspidal pairs $[M, \sigma]$, with $M$ standard, such that
$$\mathrm{Hom}_N(i^G_Q(\sigma), \psi_N) \not = 0$$
or, equivalently, 
$$\mathrm{Hom}_{N\cap M}(\sigma, \psi_{N\cap M}) \not = 0$$
that is, for $\sigma$ $(N\cap M, \psi_{N\cap M})$-generic.

\subsection{Local Langlands correspondence}
Let $W_F$ be the Weil group of $F$, and $W'_F$ the Weil-Deligne group of $F$
$$W'_F = W_F \times \mathrm{SL}_2(\mathbb{C)}.$$
For $\mathbf{G}$ a connected reductive group over $F$, we associate its Langlands dual $\mathbf{G}^\vee$ a complex connected reductive group. We consider the Weil form of the L-group 
$$^L\mathbf{G}=\mathbf{G}^\vee \rtimes W_F.$$
In our context we have 
$$^LG_n(E)=(\GL n(\mathbb{C})\times \GL n(\mathbb{C}))\rtimes W_F$$
where the action of $W_F$ on $\GL n(\mathbb{C})\times \GL n(\mathbb{C})$ is given by 

$$\rho \cdot (g_1, g_2) = \left\{\begin{array}{cc}
     (g_2, g_1), & \mathrm{if} \ \rho \in W_F\backslash W_E, \\
     (g_1,g_2), & \mathrm{otherwise}
\end{array} \right.$$
If $V$ is a Hermitian space of dimension $n$ over $E$, then $^LU(V)=\GL n(\mathbb{C})\rtimes W_F$, where the action of $W_F$ on $\GL n(\mathbb{C})$ is given by 

$$\rho \cdot g = \left\{ \begin{array}{cc}
     J_n ^tg^{-1} J_n^{-1}, & \mathrm{if} \ \rho \in W_F\backslash W_E, \\
     g & \mathrm{otherwise}
\end{array} \right.$$
with $J_n = \begin{pmatrix} & & 1 \\
                        & \revdots & \\
                        (-1)^{n-1} & & \end{pmatrix}$.
\begin{definition}
    A \emph{Langlands parameter} of $\mathbf{G}$ (or of $\mathbf{G}(F)$) is the $\mathbf{G}^\vee$-conjugacy class of a continuous homomorphism $\varphi : W_F'\to \ ^L\mathbf{G}$ such that 
    \begin{enumerate}
        \item the restriction of $\varphi$ to $\mathrm{SL}_2(\mathbb{C})$ is algebraic, 
        \item $\varphi$ sends elements of $W_F$ to semi-simple elements,
        \item $\varphi$ commutes with the natural projections $W_F'\to W_F$ and $^L\mathbf{G}\to W_F$.
    \end{enumerate}
    Let $\Phi(\mathbf{G})$ be the set of Langlands parameters of $\mathbf{G}$.

    A Langlands parameter $\varphi$ is \emph{tempered}, if the projection of $\varphi(W_F)$ to $\mathbf{G}^\vee$ is bounded, and we denote by $\Phi_{\mathrm{temp}}(\mathbf{G})$ the set of such parameters. We say that $\varphi$ is \emph{discrete} if it is tempered and the centralizer of $\varphi(W'_F)$ in $\mathbf{G}^\vee$ is finite modulo $Z(\mathbf{G}^\vee)^{W_F}$ and the set of discrete Langlands parameter is denoted $\Phi_2(\mathbf{G})$.
\end{definition}
For simplicity, set $G=\mathbf{G}(F)$. The local Langlands correspondence is, in general, a conjecture of the existence of a finite-to-one map 
$$\begin{array}{cccc}
     & \Irr{G} & \longrightarrow &  \Phi(G)\\
     & \pi & \longmapsto & \varphi_\pi
\end{array}$$
satisfying certain properties. The fibers of this map are called \emph{L-packets} of $\mathbf{G}$ (or $G$). Recall that the local Langlands correspondence has been constructed for $G_n(F)$ by Harris-Taylor \cite{HT} , Henniart \cite{Hen}, and Scholze \cite{Sch}; and for $U(V)$ by Mok \cite{Mok}, and Kaletha-Minguez-Shin-White \cite{KMSW}. It also extends to groups obtained by extensions of scalars like $G_n(E)$. These correspondences should send $\mathrm{Temp}(G)$ to $\Phi_{\mathrm{temp}}(G)$, and $\Pi_2(G)$ to $\Phi_2(G)$. For $\pi \in \Irr{G}$ and $\lambda \in \mathfrak{a}_{G, \mathbb{C}}^*$, we have 
$$\varphi_{\pi_\lambda} = \varphi_\pi \vert \cdot \vert^\lambda$$
with $\vert \cdot \vert^\lambda$ being a character from $W_F'$ to $X^*(G)\otimes \mathbb{C}^\times \cong (Z(G^\vee)^{W_F})^0$. If $G = G_1 \times G_2$ and $\pi = \pi_1 \boxtimes \pi_2$, then we have $^LG={}^LG_1\times_{W_F} {}^LG_2$ and $\varphi_\pi=\varphi_{\pi_1} \times \varphi_{\pi_2} $. For $M$ a Levi subgroup of $G$, $Q=MU$ a parabolic subgroup of $G$ and $\sigma \in \mathrm{Temp}(M)$, the representation $i_Q^G(\sigma)$ only depends, up to isomorphism, on the Levi subgroup $M$, hence we write $i_M^G(\sigma)=i_Q^G(\sigma)$. Denote by $\pi$ an irreducible subquotient of $i_M^G(\sigma)$, and let $\iota_M : {}^LM \to {}^LG$ be the natural embedding between L-groups. Then we have 
$$\varphi_\pi = \iota_M \circ \varphi_\sigma.$$
And conversely, if $\varphi_\pi = \iota_M \circ \varphi_M$ with $\varphi_M \in \Phi_2(M)$ then $\varphi_M$ is the Langlands parameter of a representation $\sigma \in \Pi_2(M)$ such that $\pi$ is an irreducible subquotient of $i_M^G(\sigma)$. 

In the case of $G = G_n(F)$ or $G_n(E)$ we even have bijections
$$\Irr{G} \to \Phi(G), \ \Pi_2(G) \to \Phi_2(G).$$
Moreover, if $\pi\in \Pi_2(G)$ then $\varphi_\pi\in \Phi_2(G)$ is irreducible, hence there exists an irreducible representation $\psi_\pi:W_F \to \mathrm{GL}_l(\mathbb{C})$ (in the case of $G_n(E)$, $\psi_\pi:W_E \to \mathrm{GL}_l(\mathbb{C})$), such that
$$\varphi_\pi = \psi_\pi \boxtimes \mathrm{Sym}^{d-1}$$
with $n=dl$, and $\mathrm{Sym}^{d-1}$ denotes the algebraic irreducible representation of dimension $d$ of $\mathrm{SL}_2(\mathbb{C})$. An irreducible representation $\pi$ of $G$ is supercuspidal if and only if its Langlands parameter $\varphi_\pi$ is discrete and trivial on $\mathrm{SL}_2(\mathbb{C})$. With the previous notation, $\psi_\pi$ corresponds to the Langlands parameter of an irreducible supercuspidal representation of $G_l(F)$ (or $G_l(E)$). We can define the Langlands parameter of the cuspidal support of $\pi$ denoted $\varphi_\pi^{ss}$, through the \emph{semi-simplification} map that is the composition of $\varphi_\pi$ with the embedding 
$$\begin{array}{cccc}
     & W_F & \longrightarrow & W_F \times \mathrm{SL}_2(\mathbb{C})=W'_F  \\
     & w & \longmapsto & \left( w, \begin{pmatrix}
         \vert w \vert^{1/2} & \\ & \vert w \vert^{-1/2}
     \end{pmatrix} \right)
\end{array}.$$

\begin{definition}[Conjugate-dual representations]
    Let $R$ be a representation of $W'_E$ (we also denote by $R$ the vector space associated to this representation), and $R^\rho$ be the conjugate representation with the same action of $\mathrm{SL}_2(\mathbb{C})$, and for $\tau\in W_E$, $R^\rho(\tau) =R( \rho\tau \rho^{-1})$. 

    The representation $R$ is \emph{conjugate-orthogonal}, if there is a non-degenerate bilinear form $B:R\times R \to \mathbb{C}$, which satisfies, 
    $$\left\{\begin{array}{cc}
         B(R(\tau) r , R(\rho\tau \rho^{-1})t) = B(r, t)  \\
         B(t, r) = B(r, R(\rho^2)t)
    \end{array} \right.,$$
    for all $\tau \in W_E'$.

    The representation $R$ is \emph{conjugate-symplectic}, if there is a non-degenerate bilinear form $B$ which satisfies
    $$\left\{\begin{array}{cc}
         B(R(\tau) r , R(\rho\tau \rho^{-1})t) = B(r, t)  \\
         B(t, r) = - B(r, R(\rho^2)t)
    \end{array} \right.,$$
    for all $\tau \in W_E'$.

    In both case, the form $B$ gives an isomorphism of representations
    $$f: R^\rho \to R^\vee,$$
    whose conjugate-dual
    $$(f^\vee)^\rho : R^\rho \to ((R^\rho)^\vee)^\rho \overset{\sim}{\rightarrow} R^\vee$$
    satisfies
    $$(f^\vee)^\rho = b.f$$
    where $b$ is the sign of $B$, i.e. $b=1$ in the conjugate-orthogonal case, $b=-1$ in the conjugate-symplectic case.

    We say that a representation is \emph{$(\pm1)$-conjugate-dual} when it is conjugate-orthogonal ($(1)$-conjugate-dual) or conjugate-symplectic ($(-1)$-conjugate-dual).
\end{definition}

From now on, we denote $\varphi^* = (\varphi^\vee)^\rho$, for all Langlands parameters $\varphi$ (and for any irreducible representation $\pi$ of $G$, the local Langlands correspondence should send $\pi^*$ to the parameter $\varphi_{\pi^*}=(\varphi_\pi)^*$).

\begin{ex}
    Let $\varphi = \psi \boxtimes \mathrm{Sym}^{d-1}$ be an irreducible Langlands parameter for $\mathrm{GL}_n(E)$. The representation $\mathrm{Sym}^{d-1}$ is orthogonal when $d$ is odd, and symplectic when $d$ is even. Then, $\varphi$ is conjugate-orthogonal (resp. conjugate-symplectic) if $\psi$ is conjugate orthogonal and $d$ is odd (resp. even) or $\psi$ is conjugate-symplectic and $d$ is even (resp. odd).

    A conjugate-dual Langlands parameter $\varphi = \bigoplus_{i=1}^k \varphi_i$ can also have an irreducible component $\varphi_i$ that is neither conjugate orthogonal nor conjugate-symplectic but in that case, there exists $j \in \{1, \dots, k\}$ such that $\varphi_j \cong \varphi_i^*$.
\end{ex}
If $\varphi:W'_F \to {}^LU(V)$ is a Langlands parameter of $U(V)$, then $\varphi$ determines a representation 
$$W'_E\to \mathrm{GL}(R)$$
which, according to the Theorem 8.1 in \cite{GGPlocalroot}, is \emph{conjugate-orthogonal} (resp. \emph{conjugate-symplectic}) when $n$ is odd (resp. even). From now on, we identify the Langlands parameters of $U(V)$ with the corresponding representations of $W_E'$. For any Levi subgroup $M=G_{n_1} \times \cdots \times G_{n_k} \times U(W)$ of $U(V)$, and any $\sigma = \sigma_1 \boxtimes \cdots \sigma_k \boxtimes \sigma_W \in \Irr{M}$, we write
$$\varphi_\sigma = (\bigoplus_{i=1}^k \varphi_i) \oplus \varphi_W$$
for $\varphi_i = \varphi_{\sigma_i}$ the Langlands parameter for $G_{n_i}(E)$, for all $i$ and $\varphi_W=\varphi_{\sigma_W}$ the Langlands parameter for $U(W)$.

In the quasi-split case, we fix a Whittaker datum $(N, \psi_N)$ for $U(n)$ and we denote by $\mathrm{Temp_{gen}}(U(n))$ the set of isomorphism class of tempered representations that are $(N, \psi_N)$-generic. Any representation in $\mathrm{Temp}(U(n))$ is a subquotient of a $i_M^{U(n)}(\sigma)$ where $\sigma \in \Pi_2(M)$, and all the irreducible subquotients of $i_M^{U(n)}(\sigma)$ are in the same L-packet which we denote $[i_M^{U(n)}(\sigma)]$. Thus, if $\mathcal{L}$ is a set of representatives of the conjugacy classes of the Levi subgroups of $U(n)$, we get an application
$$\begin{array}{cccc}
     & \bigsqcup_{M\in \mathcal{L}} (\Pi_2(M)/\mathrm{stab})/W_G(M) & \longrightarrow & \mathrm{Temp}(U(n))/\mathrm{stab} \\
     & \sigma & \longmapsto & [i_M^{U(n)}\sigma]
\end{array}$$
which is even a bijection. Indeed, the local Langlands correspondence provides the following commuting diagram
$$\xymatrix{
     \bigsqcup_{M\in \mathcal{L}} (\Pi_2(M)/\mathrm{stab})/W_G(M)  \ar[r]^{\ \ \ \ \ \ \ \cong} \ar[d]_{}  & \bigsqcup_{M\in \mathcal{L}} \Phi_2(M)\ar[d]^{} \\
     \mathrm{Temp}(U(n))/\mathrm{stab}  \ar[r]_{\cong} & \Phi_{\mathrm{temp}}(U(n))
     }$$
Recall that every irreducible tempered representation is a subquotient of an induced discrete series (VII.2.6 in \cite{Ren}). In other words, for all Langlands parameters $\varphi : W_F' \to {}^LG$, there exists a unique Levi subgroup ${}^LM$ of ${}^LU(n)$ that $\varphi$ factorizes through, and minimal for that property. Therefore, we obtain that the map on the right hand side of the above diagram, given by 

$$\begin{array}{cccc}
     & \bigsqcup_{M\in \mathcal{L}} \Phi_2(M) & \longrightarrow &  \Phi_{\mathrm{temp}}(U(n))  \\
     & \varphi_M & \longmapsto & \iota_M \circ \varphi_M
\end{array},$$
is a bijection. 
Finally, let $M\in \mathcal{L}$ and $\sigma\in \Pi_2(M)$. There is a unique $(N, \psi_N)$-generic subquotient $\pi \in [i_M^{U(n)}\sigma]$ (Theorem 7 in \cite{Rodier}), and every discrete L-packet contains a unique generic representation \cite{Konno}. We therefore have 
$$\underset{{M\in \mathcal{L}}}{\bigsqcup} (\Pi_2(M)/\mathrm{stab})/W_G(M) \cong \mathrm{Temp}(U(n))/\mathrm{stab} \cong \mathrm{Temp_{gen}}(U(n)).$$

\subsection{Base-change}

Let $E/F$ be a quadratic extension of non Archimedean local fields. Let $V$ be a hermitian space of dimension $n$ over $E$ the \emph{base-change} map
$$\begin{array}{cccc}
    BC: & ^LU(V) & \longrightarrow & ^LG_n(E)  \\
     & (g, \rho) & \longmapsto & (g, J_n^tg^{-1}J_n^{-1}, \rho)
\end{array}$$
induces a map between Langlands parameters $\Phi_{\mathrm{temp}}(U(V)) \to \Phi_{\mathrm{temp}}(G_n(E))$, and thus by composing with the local Langlands correspondences we get a map from $\mathrm{Temp}(U(V))/\mathrm{stab}$ to $\mathrm{Temp}(G_n(E))$, which we also denote by $BC$. This last base-change map is injective and its image is formed by representations $\pi\in \mathrm{Temp}(G_n(E))$, satisfying $\varphi_\pi^*\cong \varphi_\pi$ and $\varphi_\pi$ is $(-1)^{n-1}$-conjugate-dual. For all $\tau \in \mathrm{Temp}(G_m(E))$ and $\sigma \in \mathrm{Temp}(U(W))$ we have 
$$BC(\tau \rtimes \sigma) = \tau \times \tau^* \times BC(\sigma).$$
Let $\eta'_n$ be a character of $G_n(E)$, defined by
    
    $$\eta'_n(g) = \left\{\begin{array}{cc}
         \eta'(\det g), & \ n \ \mathrm{even},  \\
         1, & \ n \ \mathrm{odd}.
    \end{array} \right.$$
We will need a twist of this base-change map, defined by $BC_n (\sigma) = BC(\sigma) \otimes \eta'_n$, for all $\sigma \in \mathrm{Temp}(U(V))/\mathrm{stab}$. For the quasi-split unitary group $U(n)$, we can describe the image of $BC_n$. A representation $\pi \in \mathrm{Temp}(G_n(E))$ is in the image of $BC_n: \mathrm{Temp}(U(n)) \to \mathrm{Temp}(G_n(E))$ if and only if it can be written as 
$$\pi = \left( \bigtimes_{i=1}^k \tau_i \times \tau_i^* \right) \times \bigtimes_{j=1}^l \mu_j$$
where for every $1\leq i \leq k$, $\tau_i \in \Pi_2(G_{n_i}(E))$ for some $n_i\geq 1$ and for every $1\leq j \leq l$, $\mu_j \in BC_{m_j}(\mathrm{Temp}(U(m_j))) \cap \Pi_2(G_{m_j}(E))$ for some $m_j \geq 1$. Moreover, we have seen that $\mathrm{Temp}(U(n))/\mathrm{stab}$ can be identified with $\mathrm{Temp_{gen}}(U(n))$. By using Theorem 4.2 of \cite{Mat}, we obtain that the image of $\mathrm{Temp_{gen}}(U(n))$ through the base-change map is the set of equivalence classes of $G_n(F)$-distinguished irreducible representations $\mathrm{Temp}_{G_n(F)}(G_n(E))$, that is

$$\mathrm{Temp_{gen}}(U(n))\cong \mathrm{Temp}(U(n))/\mathrm{stab}\underset{BC_n}{\cong} \mathrm{Temp}_{G_n(F)}(G_n(E)).$$

\subsection{Asai L-functions}

Let $E/F$ be a quadratic extension of non Archimedean local field, and $G=G_n(E)$. Let $R$ be a representation of ${}^LG$, for any Langlands parameter $\varphi$ attached to an irreducible representation $\pi$ of $G$, we get a Weil-Deligne representation $R \circ \varphi$ and write $L(s, \pi, \mathrm{R})=L(s, R\circ \varphi)$. For a standard parabolic subgroup $Q=MU$ of $G$, and $\sigma$ an irreducible representation of $M$ with Langlands parameter $\iota\circ \varphi_\sigma$, such that $\pi = i_Q^{G}\sigma$,  we set 
$$L(s, \pi, \mathrm{Ad}):= L(s, \mathrm{Ad}\circ \iota \circ\varphi_\sigma).$$
The following proposition is a reformulation of the Bernstein-Zelevinsky irreducibility criterion, Theorem 4.2 \cite{BZ1}.
\begin{prop}\label{LAd}
    Let $Q=MU$ be a standard parabolic subgroup of $G$, $\sigma$ be an irreducible supercuspidal representation of $M$, and $\pi = i_Q^{G}\sigma$. The representation $\pi$ is irreducible if and only if $L(s, \pi, \mathrm{Ad})$ has no pole at $s=1$.
\end{prop}

\begin{proof}
   We have $M=G_{n_1}(E)\times \cdots \times G_{n_k}(E)$ and $\pi = \sigma_1 \times \cdots \times \sigma_k$. Let $\mathrm{Frob}_E$ be a Frobenius of the Weil group $W_E$ of $E$, we have $W_E=\mathcal{I}_E \rtimes \langle \mathrm{Frob}_E\rangle$. If $\varphi_\pi$ is the Langlands parameter corresponding to $\pi$, we have
   $$\mathrm{Ad}\circ \varphi_\pi = \varphi_\pi \otimes \varphi_\pi^\vee,$$
   and 
   $$\varphi_\pi = \varphi_1 \oplus \cdots \oplus \varphi_k,$$
   where $\varphi_i$ is the Langlands parameter of $\sigma_i$, for all $i\in \{ 1, \dots, k \}$. Hence, 
   $$L(s, \pi, \mathrm{Ad}) = \prod_{1 \leq i, j \leq k }L(s, \varphi_i\otimes \varphi_j^\vee).$$
   The function $L(s, \pi, \mathrm{Ad})$ has a pole at $1$ if and only if there exists $i,j\in \{1, \dots , k\}$, $i\not = j$, such that 
   $$\mathrm{det}(1-q_E^{-1}\varphi_i\otimes \varphi_j^\vee (\mathrm{Frob}_E)\vert V^{\mathcal{I}_E})=0$$
   where $V$ denotes the vector space $V_i \otimes V_j^\vee$ of the representation $\varphi_i \otimes \varphi_j^\vee$. Recall that we have a $W'_E$-equivariant isomorphism
   $$\begin{array}{cccc}
        & V_i\otimes V_j^\vee & \longrightarrow & \mathrm{Hom}(V_j, V_i) \\
        & v_i\otimes v_j^\vee & \longmapsto & (w\in V_j \mapsto v_j^\vee(w)v_i)
   \end{array}$$
   As $q_E^{-1} = \vert \mathrm{Frob}_E \vert$ is an eigen value of $\varphi_i\otimes \varphi_j^\vee (\mathrm{Frob}_E)$ if and only if for some non-zero $v\in V^{\mathcal{I}_E}$ 
   $$\varphi_i \otimes \varphi_j^\vee (\mathrm{Frob}_E)v= q_E^{-1}v,$$
   we want to find the fixed points of the action of $W_E$ on $\mathrm{Hom}(V_j, \vert  \cdot  \vert_E V_i)$. These fixed points correspond to the intertwining operators between $\vert  \cdot  \vert_E \varphi_i$ and $\varphi_j$ (which are irreducible). Finally, $L(s, \pi, \mathrm{Ad})$ has a pole at $1$ if and only if there exist $i,j\in \{1, \dots , k\}$, $i\not = j$, such that  $\mathrm{Hom}_{W_E}(V_j, \vert  \cdot  \vert_E V_i) \not = 0$ or equivalently, such that $\sigma_j \cong \sigma_i \vert\det \vert_E$. Thanks to the irreducibility criterion of Bernstein-Zelevinsky \cite{BZ1} (Theorem 4.2), we conclude that $L(s, \pi, \mathrm{Ad})$ has no pole at $s=1$ if and only if $\pi$ is irreducible.
\end{proof}

We recall the definitions and results we will use, most of them come from \cite{GGPlocalroot}.

\begin{definition}[Asai L-functions]
    Let $V$ be a Hermitian space, and $U(V)$ be its associated unitary group. Let ${}^LU(V) = \mathrm{GL}(M)\rtimes W_F$. The L-group of $\mathrm{Res}_{E/F}(\mathrm{GL}(V))$ can be identified as
    $$H= (\mathrm{GL}(M) \times \mathrm{GL}(M)) \rtimes \mathrm{Gal}(E/F),$$
    where $\mathrm{Gal}(E/F)\cong \mathbb{Z}/2\mathbb{Z}$ acts by permuting the two factors of $\mathrm{GL}(M)$. Let $H_0 = \mathrm{GL}(M) \times \mathrm{GL}(M)$, we have 
    $$\mathrm{Ind}_{H^0}^H(M \boxtimes M) = \mathrm{As}^+(M) \oplus \mathrm{As}^-(M)$$
    where $\mathrm{As}^+(M)$ is the space $M \boxtimes M$ on which $\mathrm{Gal}(E/F)$ acts by permuting the two copies of $M$, and $\mathrm{As}^-(M)$ is obtained by twisting $\mathrm{As}^+(M)$ by the non trivial character of $H/H_0$.
\end{definition}
Let $\rho$ be an element of $W_F$ generating the quotient $W_F/W_E$. Given a representation 
$$\varphi : W'_E \to \mathrm{GL}(M),$$
we get a morphism
$$\Tilde{\varphi}: W'_F \to H,$$ 
defined by 
$$\Tilde{\varphi}(w) = (\varphi(w), \varphi(\rho w\rho^{-1})), $$
if $w\in W_E'$, and 
$$\Tilde{\varphi}(\rho) = ((\mathrm{Id}, \varphi(\rho^2)), \overline{1})$$
where $\overline{1}$ is the non-trivial element of $\mathbb{Z}/2 \mathbb{Z}$ permuting the two copies of $M$. Through this map we get an action of $W_F'$ on $\mathrm{As}^\pm(M)$, and 
$$\mathrm{As}^-(M) = \mathrm{As}^+(M)\otimes \eta_{E/F}.$$
Here we use mainly the results of the part 7 of \cite{GGPlocalroot}, and more precisely, thanks to Proposition 7.4 of op. cit. we get the following equality, for all $\pi \in BC_n(\mathrm{Temp_{gen}}(U(n)))$
$$L(s, \pi , \mathrm{Ad}) = L(s, \pi , \mathrm{As}^+)L(s, \pi , \mathrm{As}^-).$$

\subsection{Rankin-Selberg Zeta integrals}

Let $\pi \in \mathrm{Irr_{gen}}(G_n(E))$ and $e_n=(0, \dots, 0,1)\in F^n$. For all $W\in \mathcal{W}(\pi, \psi_n)$, $\phi \in C_c^\infty(F^n)$ and $s\in \mathbb{C}$, the Rankin-Selberg Zeta integral is defined by 
$$Z(s, W, \phi)=\int_{N_n(F)\backslash G_n(F)} W(h) \phi(e_nh)\vert\det h \vert_F^sdh$$
and is absolutely convergent whenever $\Re(s)$ is sufficiently large. Moreover, if $\pi$ is tempered, then $Z(s, W, \phi)$ is absolutely convergent whenever $\Re(s)> 0$ (\cite{Fli}, Section 4, Proposition (i)). Let $\hat{\phi}$ denote the Fourier transform of $\phi \in C_c^\infty(F^n)$ with respect to the additive character $\psi'$ and the selfdual measure. Explicitely, we have 
$$\hat{\phi}(x_1, \dots, x_n)= \int_{F^n}\phi(y_1, \dots, y_n)\psi'(x_1y_1+ \cdots+ x_ny_n)d(y_1, \dots, y_n),$$
such that $\hat{\hat{\phi}}(x)=\phi(-x)$, $x\in F^n$. Let $\pi \in \mathrm{Temp}(G_n(E))$, for $ W\in \mathcal{W}(\pi, \psi_n)$, we set
$$\beta (W) = \int_{N_n(F)\backslash P_n(F)} W(p)dp$$
and this integral is absolutely convergent according to Lemma 2.151 in \cite{RBPPlanch}.

\begin{prop}\label{ZBeta}
    Let $\pi \in \mathrm{Temp}_{G_n(F)}(G_n(E))$. For all $W\in \mathcal{W}(\pi, \psi_n)$ and $\phi\in C_c^\infty(F^n)$, 
    $$Z(1, W, \phi)=\hat{\phi}(0)\beta(W), $$
    up to a constant that depends on the choices of the measures on $N_n(F) \backslash G_n(F)$ and $N_n(F)\backslash P_n(F)$.
\end{prop}

\begin{proof}
    Let $\pi \in \mathrm{Temp}_{G_n(F)}(G_n(E))$. First of all, recall that the restriction of $\beta$ to $\mathcal{W}(\pi, \psi_n)$ is $G_n(F)$-invariant (see Section 4 of \cite{Offen}). Let $W\in \mathcal{W}(\pi, \psi_n)$ and $\phi\in C_c^\infty(F^n)$, we have 
    \begin{align*}
         Z(1,W, \phi) &= \int_{P_n(F)\backslash G_n(F)} \int_{N_n(F)\backslash P_n(F)} W(ph) \phi(e_nph) \vert \det ph \vert \vert \det p \vert^{-1} dp dh, \\
        & = \int_{N_n(F)\backslash P_n(F)} W(p) dp \int_{P_n(F)\backslash G_n(F)} \phi(e_nh) \vert \det h \vert dh.
    \end{align*}
    Setting $x=e_nh\in F^n$, we get 
    $$Z(1,W, \phi)= \beta(W) \int_{F^n}\phi(x)dx = \hat{\phi}(0) \beta(W).$$
\end{proof}

\section{Cuspidal support and base-change}\label{cuspBC}

Let $E/F$ be a quadratic extension of local non-Archimedean fields and $G=G_n(E)$. In this section we build a base-change map for cuspidal support, i.e. from $\Omega(U(n))$ to $\Omega(G)$, as follows. Recall that we have maps  
$$\Irr{U(n)} \longrightarrow \Omega (U(n)), \ \Irr{G_n(E)}  \longrightarrow  \Omega (G_n(E))$$
both denoted $\mathbf{Cs}$. For a cuspidal support $[M,\sigma] \in \Omega(U(n))$, we can write 
$$M=G_{n_1} \times \cdots \times G_{n_k} \times U(m), \ \mathrm{and} \ \sigma= \sigma_1 \boxtimes \cdots \boxtimes \sigma_k \boxtimes \sigma_u,$$
which gives rise to a Levi subgroup $M'$ of $G_n(E)$ of the form 
$$M'=G_{n_1} \times G_{n_1}\cdots \times G_{n_k} \times G_{n_k}\times G_{m},$$
and $\sigma$ to the following representation of $M'$
$$BC(\sigma) = \sigma_1 \boxtimes \sigma_1^* \cdots \boxtimes \sigma_k \boxtimes \sigma_k^* \boxtimes BC(\sigma_u).$$
We choose a representative of $\mathbf{Cs}(BC(\sigma_u))$, given by a Levi $L$ of $G_m$ and a supercuspidal representation $\tau$, such that we get a cuspidal support in $G_n(E)$ associated to $BC(\sigma)$,
$$[G_{n_1}^2\cdots \times G_{n_k}^2 \times L, \sigma_1 \boxtimes \sigma_1^* \cdots \boxtimes \sigma_k \boxtimes \sigma_k^* \boxtimes \tau]\in \Omega(G_n(E)).$$
Now, if we denote
$$\mathcal{M}=G_{n_1}^2\cdots \times G_{n_k}^2 \times L, \ \mathrm{and} \ CBC(\sigma) =  \sigma_1 \boxtimes \sigma_1^* \cdots \boxtimes \sigma_k \boxtimes \sigma_k^* \boxtimes \tau,$$
where $CBC$ stands for "cuspidal base-change", then we have defined a map from $\Omega(U(n))$ to $\Omega (G)$, by $[M, \sigma] \mapsto CBC([M, \sigma]) = [\mathcal{M}, CBC(\sigma)]$. We define similarly $CBC_n$, by replacing $BC$ by $BC_n$ in the previous definition. This $CBC$ map is called "extended cuspidal support" in \cite{Moeglin}, from which we use the results of Parts 5.3 and 5.4, in particular that the following diagram commutes
$$\xymatrix{
     \mathrm{Irr}(U(n))  \ar[r]^{BC} \ar[d]_{\mathbf{Cs}}  & \mathrm{Irr}(G_n(E)) \ar[d]^{\mathbf{Cs}} \\
     \Omega(U(n))  \ar[r]_{CBC} & \Omega(G_n(E))
     }$$
We remark that the base change map provides an injective map
$$\mathrm{Irr_{gen}}(U(n)) \hookrightarrow \mathrm{Irr}(G_n(E)),$$
because each  discrete L-packet of $U(n)$ contains at most one generic representation \cite{Konno}. Since $\mathbf{Cs}$ induces an ismorphism between $\mathrm{Irr_{gen}}(U(n))$ and $\Omega_\mathrm{{gen}}(U(n))$, the above diagram can also be given in more detail, 
$$\xymatrix{
     \mathrm{Irr_{gen}}(U(n)) \ar[r] \ar[d]_{\mathbf{Cs}} & \Phi(U(n))  \ar[r]^{BC_n} \ar[d]^{^{ss}} & \Phi(G_n(E))  \ar[r]^{LLC} \ar[d]^{^{ss}}  & \mathrm{Irr}(G_n(E)) \ar[d]^{\mathbf{Cs}} \\
     \Omega_{\mathrm{gen}}(U(n)) \ar[r] & \mathcal{B}(U(n))  \ar[r]_{BC_n} & \mathcal{B}(G_n(E)) \ar[r]_{LLC}& \Omega(G_n(E))
     }$$
where $\mathcal{B}(U(n))$ is the set of Langlands parameters from $W_F'$ to $^LU(n)$ that are trivial on $\mathrm{SL}_2(\mathbb{C})$, and $\mathcal{B}(G_n(E))$ is the set of semi-simple representations of $W_E$ of dimension $n$. These sets can also be seen as the image of $\Phi(U(n))$ (resp. $\Phi(G_n(E))$) through the semi-simplification map
$$\mathcal{B}(U(n))=\left\{\varphi^{ss}: w \mapsto \varphi\left(w, \begin{pmatrix}
    \vert w\vert^{1/2} & \\ & \vert w\vert^{-1/2}
\end{pmatrix}\right), \ \varphi \in \Phi(U(n))  \right\}.$$
Let $\Phi_{\mathrm{gen}}(U(n))$ be the set of generic parameters, i.e. $\varphi_\pi$ with $\pi \in \mathrm{Irr_{gen}}(U(n))$. Thanks to Proposition B.1 of \cite{GanIchino}, this set can be described as
$$\Phi_{\mathrm{gen}}(U(n)) = \{ \varphi \in \Phi(U(n)), \ L(1, \varphi, \mathrm{Ad})\not = \infty\} $$
where $\mathrm{Ad}$ is the adjoint representation of $^LU(n)$ on its Lie algebra. Recall that 
$$BC_n : \Phi(U(n)) \to \Phi(G_n(E)) \ \ (\mathrm{resp.} \ \ BC_n: \mathcal{B}(U(n))\to \mathcal{B}(G_n(E))),$$
is injective, and its image is the set of conjugate-orthogonal parameters in $\Phi(G_n(E))$ (resp. $\mathcal{B}(G_n(E)))$). Furthermore, for $\varphi \in \Phi(U(n))$ we have 
$$\mathrm{Ad}_{U(n)}\circ \varphi \cong \mathrm{As}^-\circ BC_n(\varphi).$$
We will now show that 
$$\mathbf{Cs} \circ BC_n : \mathrm{Irr_{gen}}(U(n)) \to \Omega(G_n(E))$$
is injective. Since $\mathbf{Cs}$ induces an ismorphism between $\mathrm{Irr_{gen}}(U(n))$ and $\Omega_\mathrm{{gen}}(U(n))$, it suffices to show that 
$$CBC : \Omega_\mathrm{{gen}}(U(n))\to \Omega(G_n(E))$$
is injective. Using the extended commuting diagram above, we want the composition
$$ \mathrm{Irr_{gen}}(U(n)) \longrightarrow \mathcal{B}(U(n))  \underset{BC_n}{\hookrightarrow} \mathcal{B}(G_n(E)) \overset{\cong}{\underset{LLC}{\longrightarrow}} \Omega(G_n(E))$$
to be injective. Hence, it remains to show that
$$\Omega_\mathrm{{gen}}(U(n)) \to \mathcal{B}(U(n))$$
is an injection. In fact, the following statement even shows that this is a bijection.


\begin{prop}\label{propcuspBC}
    Let $\psi:W_E \to \GL{n}(\mathbb{C})$ be a conjugate-orthogonal representation. There exists a unique conjugate-orthogonal Langlands parameter $\varphi \in \Phi(G_n(E))$ verifying the following conditions.
    \begin{enumerate}
        \item $\varphi^{ss}=\psi$, 
        \item $L(s, \varphi, \mathrm{As}^-)$ has no pole at $s=1$.
    \end{enumerate}
\end{prop}

\begin{proof}
    Let us recall that any irreducible representation of $W_E'$ is of the form $\Tilde{\psi}\boxtimes \mathrm{Sym}^{d-1}$, with $\Tilde{\psi}$ an irreducible representation of $W_E$ and $\mathrm{Sym}^{d-1}$ is the irreducible representation of dimension $d$ of $\mathrm{SL}_2(\mathbb{C})$. Hence, the semi-simplification is of the form
    $$\bigoplus_{i=0}^{d-1}\Tilde{\psi}\vert \cdot \vert^{\frac{d-(2i+1)}{2}}.$$
    
    Let $\psi: W_E \to \GL{n}(\mathbb{C})$ be a conjugate-orthogonal representation. Denote by $\mathrm{Irr}_u(W_E)$ the set of isomorphism classes of irreducible and unitary representations of $W_E$. Then $\psi$ can be decomposed in the form
    $$\psi=\bigoplus_{\psi'\in \mathrm{Irr}_u(W_E)}\bigoplus_{a\in E_{\psi'}} \psi'\vert \cdot \vert^ a,$$
    where each $E_{\psi'}\subset \mathbb{R}$ is a multiset. For $a, b\in \mathbb{R}$, such that $b-a\in \mathbb{N}$, we define the Weil-Deligne representation
    $$[a,b] := \vert \cdot\vert^{\frac{a+b}{2}}\boxtimes \mathrm{Sym}^{b-a},$$
    and from now on every segment $[a, b]= \{a, a+1, \dots, b\}$ is identified with this Weil-Deligne representation. We let $l([a,b])=b-a$ the \emph{length} of the segment $[a,b]$. When $a=-b$, the corresponding Weil-Deligne representation $ \mathrm{Sym}^{2b}$ is self-dual, of \emph{sign} $\mathrm{sgn}([-b,b])=(-1)^{2b}$.
    
    Let $\varphi$ be a semi-simple representation of $W_E'$. There exists a unique family $(S_{\Tilde{\psi}})_{\Tilde{\psi}\in \mathrm{Irr}_u(W_E)}$ of multisets of segments such that
    $$\varphi \cong \bigoplus_{\Tilde{\psi}\in \mathrm{Irr}_u(W_E)} \bigoplus_{I \in S_{\Tilde{\psi}}} \Tilde{\psi} \otimes I.$$
    Moreover $\varphi$ is conjugate-orthogonal (or $\mathrm{sgn}(\varphi)=1$), if and only if it satisfies the following conditions:  
    \begin{enumerate}
        \item[(i)] For every multiset $S_{\Tilde{\psi}}$, we have 
        $$S_{\Tilde{\psi}^*} = -S_{\Tilde{\psi}},$$
        where $-S_{\Tilde{\psi}}$ is the multiset $\{[-b, -a], \ [a,b]\in S_{\Tilde{\psi}}\}$.
        \item[(ii)] \begin{itemize}
            \item If ${\Tilde{\psi}}$ is conjugate-orthogonal, then any symmetric segment $I=-I$ such that $\mathrm{sgn}(I)=-1$ appears with even multiplicity in $S_{\Tilde{\psi}}$.
            \item If ${\Tilde{\psi}}$ is conjugate-symplectic, then any symmetric segment $I=-I$ such that $\mathrm{sgn}(I)=1$ appears with even multiplicity in $S_{\Tilde{\psi}}$.
        \end{itemize} 
    \end{enumerate}
    Moreover, we have the decomposition
    \begin{align*}
        L(s, \varphi, \mathrm{As}^-) &= \prod_{\Tilde{\psi}} \prod_{I \in S_{\Tilde{\psi}}} L(s,  \Tilde{\psi} \otimes I, \mathrm{As}^-) \prod_{\Tilde{\psi} \not \cong\Tilde{\psi}'}\prod_{I\in S_{\Tilde{\psi}}, \ J \in S_{\Tilde{\psi}'}} L(s, \Tilde{\psi}\otimes I \otimes (\Tilde{\psi}')^c\otimes J) \\
        & \times \prod_{\Tilde{\psi}}\prod_{I\in \underset{I \in S_{\Tilde{\psi}}, \ J \in S_{\Tilde{\psi}}\backslash \{I\}}{\{I,J\} \subset S_{\Tilde{\psi}}}} L(s, \Tilde{\psi}\otimes I_i \otimes \Tilde{\psi}^c\otimes I_j)
    \end{align*}
    
    From there, we see that $L(s, \varphi, \mathrm{As}^-)$ has a pole at $s=1$ if and only if there exist $\Tilde{\psi} \in \mathrm{Irr}_u(W_E)$ and $I\in S_{\Tilde{\psi}}$ such that $L(s, \Tilde{\psi}\otimes I, \mathrm{As}^-)$ has a pole at $s=1$, or there exist $\Tilde{\psi}, \Tilde{\psi}'\in \mathrm{Irr}_u(W_E)$ and $I\in S_{\Tilde{\psi}}, \ J \in S_{\Tilde{\psi}'}$ such that $L(s, \Tilde{\psi}\otimes I \otimes (\Tilde{\psi}')^c\otimes J)$ has a pole at $s=1$.
    
    Let $\Tilde{\psi}\in \mathrm{Irr}_u(W_E)$ and let $I=[a,b] \subset \mathbb{R}$. Denote $$\Tilde{\psi}_I= \Tilde{\psi} \otimes I = \Tilde{\psi} \vert  \cdot \vert ^{\frac{a+b}{2}}\boxtimes \mathrm{Sym}^{b-a},$$
    and $V_{\Tilde{\psi}}$, $V_{b-a}$, the vector space on which $W_E$ acts by $\Tilde{\psi} \vert\cdot \vert^{\frac{a+b}{2}}$ and $\mathrm{SL}_2(\mathbb{C})$ by $\mathrm{Sym}^{b-a}$ respectively.
    \begin{lem}
        For all $s \in \mathbb{C}$, all $\tilde{\psi}\in \mathrm{Irr}_u(W_E)$ and all segment $I=[a,b]\subset \mathbb{R}$, we have 
        $$L(s, \Tilde{\psi}\otimes I, \mathrm{As}^-) = \prod_{m=0}^{b-a} L(s+2b-m, \Tilde{\psi},\mathrm{As}^{(-1)^{m+1}}).$$
    \end{lem}
    
    \begin{proof}
    Recall that the representation $\mathrm{As}^-(\Tilde{\psi}_I)$ is realized on $V_{\Tilde{\psi}}\otimes V_{\Tilde{\psi}}\otimes V_{b-a}\otimes V_{b-a}$. For a certain $\sigma\in W_F\backslash W_E$,
    $$\mathrm{As}^-(\psi_I)(\sigma) = -((\mathrm{Id} \otimes \Tilde{\psi}(\sigma^2))\circ \iota_{V_{\Tilde{\psi}}}) \otimes \iota_{V_{b-a}}$$
    where $\iota_V(v_1\otimes v_2)=v_2\otimes v_1$, for $v_1,v_2\in V$. For all $\tau\in W_E$, $g\in \mathrm{SL}_2(\mathbb{C})$
    $$\mathrm{As}^-(\Tilde{\psi}_I)(\tau, g) = (\Tilde{\psi}(\tau)\vert\tau\vert^{\frac{a+b}{2}}\otimes \Tilde{\psi}(\sigma\tau\sigma^{-1})\vert\tau\vert^{\frac{a+b}{2}})\otimes(\mathrm{Sym}^{b-a}(g)\otimes \mathrm{Sym}^{b-a}(g)).$$
    We have $V_{b-a}\otimes V_{b-a}\cong V_{2b-2a}\oplus V_{2b-2a-2} \oplus \cdots \oplus V_0$ (see p.87 in \cite{BroDi}) and $\iota_{V_{b-a}}$ is an involutive self-intertwining operator of this $\mathrm{SL}_2(\mathbb{C)}$-representation, hence it acts by $\pm 1$ on each of the $V_{2i}$'s. More precisely, 
    \begin{lem}
        With the same notation, for all $i \in \{0, \dots, b-a\}$,
        $$\iota_{V_{b-a}}\vert_{V_{2i}} = (-1)^{b-a+i}\mathrm{Id}_{V_{2i}}.$$
    \end{lem}
    
    \begin{proof}
        $$\iota_{V_{b-a}} \in \mathrm{Hom}_{\mathrm{SL}_2(\mathbb{C})}\left(\bigoplus_{i=0}^{b-a}V_{2i}, \bigoplus_{i=0}^{b-a}V_{2i} \right).$$
    As $V_{2i}$ is irreducible for all $i\in \{0, \dots b-a\}$, we therefore have 
    $$\iota_{V_{b-a}} \in \bigoplus_{i=0}^{b-a}\mathrm{Hom}_{\mathrm{SL}_2(\mathbb{C})}(V_{2i}, V_{2i}),$$
    which means that $\iota_{V_{b-a}}\vert_{V_{2i}} = \lambda_i \mathrm{Id}_{V_{2i}}$ is a homothety such that $\lambda_i^2=1$ because 
    $$\iota_{V_{b-a}} \circ \iota_{V_{b-a}} = \mathrm{Id}_{V_{b-a}}.$$
    Recall that 
    $$V_{b-a} \otimes V_{b-a} = \mathrm{Sym}^2(V_{b-a}) \oplus \Lambda^2(V_{b-a}), $$
    where $\Lambda^2(V_{b-a})$ denotes the exterior square of $\Lambda^2(V_{b-a})$. By definition, $\iota_{V_{b-a}}$ acts by $+1$ on $\mathrm{Sym}^2(V_{b-a})$ and by $-1$ on $\Lambda^2(V_{b-a})$.
    Following \cite{FultonHarris}, 11.1, we show that for all integer $n$
    $$\mathrm{Sym}^2(V_{n}) = \bigoplus_{i=0}^{\lfloor n/2\rfloor} V_{2n-4i}.$$
    The weights of $V_{n}$ are $n, n-2, \dots, -n+2, -n$, hence the weights of $V_n \otimes V_n$ are of the form 
    $$\alpha +\beta , \ \text{with} \ \alpha, \beta \in \{n, n-2, \dots, -n+2, -n\},$$
    we can arrange them in a matrix as follows
    $$E := \begin{pmatrix}
        2n & 2n-2 & 2n-4 & \cdots & 0 \\
        2n-2 & 2n-4 & 2n-6 & \cdots & -2 \\
        2n-4 & 2n-6 & 2n-8 &  \cdots & -4 \\
        \vdots & \vdots &  & \ddots &  \vdots \\
        0 & -2 & -4 & \cdots & -2n
    \end{pmatrix}$$
    and remark that the weights of $\mathrm{Sym}^2(V_{b-a})$ are given by the upper triangular matrix
    $$\begin{pmatrix}
        2n & 2n-2 & 2n-4 & \cdots & 0 \\
        * & 2n-4 & 2n-6 & \cdots & -2 \\
        * & * & 2n-8 &  \cdots & -4 \\
        \vdots & \vdots &  & \ddots &  \vdots \\
        * & * & * & \cdots & -2n
    \end{pmatrix}.$$
    As the representations are entirely determined by their eigenvalues, we can identify the irreducible components of $\mathrm{Sym}^2(V_{b-a})$. Namely, the first line of $E$ together with the last column give 
    $$2n, 2n-2, \dots, -2n$$
    as eigenvalues, corresponding to $V_{2n}$. We then take the second line and the $(n-1)$-th column and obtain $V_{2n-4}$. We iterate this step until the $\lfloor n/2 \rfloor$-th line and column. This shows the formula 
    $$\mathrm{Sym}^2(V_{n}) \cong \bigoplus_{i=0}^{\lfloor n/2\rfloor} V_{2n-4i}.$$
    We deduce from that that for all $i \in \{0, \dots, b-a\}$,
    $$\iota_{V_{b-a}}\vert_{V_{2i}} = (-1)^{b-a+i}\mathrm{Id}_{V_{2i}}.$$
    
    \end{proof}

    Finally, we have 
    $$V_{\Tilde{\psi}}\otimes V_{\Tilde{\psi}}\otimes V_{b-a}\otimes V_{b-a} \cong \bigoplus_{m=0}^{b-a}(V_{\Tilde{\psi}}\otimes V_{\Tilde{\psi}})\otimes V_{2b-2a-2m}$$
    and 
    $$\mathrm{As}^-(\Tilde{\psi}_I) \cong \bigoplus_{m=0}^{b-a}\mathrm{As}^{(-1)^{m+1}}(\Tilde{\psi} \vert \cdot\vert^{\frac{a+b}{2}})\otimes \mathrm{Sym}^{2b-2a-2m}.$$
    In terms of segments, it becomes
    $$\mathrm{As}^-(\Tilde{\psi}_I) \cong \bigoplus_{m=0}^{b-a}\mathrm{As}^{(-1)^{m+1}}(\Tilde{\psi})\otimes [2a+m, 2b-m].$$
    Hence, 
    $$L(s, \Tilde{\psi}\otimes I, \mathrm{As}^-) = \prod_{m=0}^{b-a} L(s+2b-m, \Tilde{\psi},\mathrm{As}^{(-1)^{m+1}}).$$
    \end{proof}
    
    Using this formula, we obtain that $L(s, \Tilde{\psi}\otimes I, \mathrm{As}^-)$ has a pole at $s=1$ if and only if $-1-b \in \{a, a+1, \dots, b\}$ and $\Tilde{\psi}$ is $(-1)^{2b}$-conjugate-dual (with $(-1)$-conjugate-dual being conjugate-symplectic, and $1$ -conjugate-dual being conjugate-orthogonal). For the other L-factors, let $\Tilde{\psi}, \Tilde{\psi}'$ be irreducible and unitary Weil representations and $I=[a,b], \ J=[\alpha,\beta]$ two segments in $\mathbb{R}$. Then
    $$\Tilde{\psi} \otimes I \otimes (\Tilde{\psi'})^c \otimes J \cong \Tilde{\psi} \otimes (\Tilde{\psi'})^c \otimes \bigoplus_{m=0}^{\min(b-a, \beta-\alpha)}\left[a+\alpha+m, b+\beta-m\right].$$
    Therefore, 
    $$L(s,\Tilde{\psi} \otimes I \otimes (\Tilde{\psi'})^c \otimes J) = \prod_{m=0}^{\min(b-a, \beta-\alpha)} L(s+b+\beta-m,\Tilde{\psi} \otimes (\Tilde{\psi'})^c )$$
    has a pole at $s=1$ if and only if $ 1+b+\beta \in \{0, 1, \dots, \min( b-a, \beta-\alpha)\}$ and $\Tilde{\psi}^*\cong \Tilde{\psi'}$. Recall that two segments $[a, b]$ and $[\alpha, \beta]$ are said to be \emph{linked} if their union is still a segment and none of them is included in the other, when $a\leq \alpha$, i.e. $[a, b]$ is on the left of $[\alpha, \beta]$, we then write $[a, b]\prec[\alpha, \beta]$. With this notation, the condition 
    $$1+b+\beta \in \{0, 1, \dots, \min( b-a, \beta-\alpha)\}$$
    is equivalent to 
    $$I\prec -J, \ \ \mathrm{if} \ \  l(I)\leq l(J),$$
    and $J \prec -I$ otherwise.
    
    Let us reformulate our proposition in terms of segments. We know that
    $$\psi=\bigoplus_{\Tilde{\psi}\in \mathrm{Irr}_u({W_E})}\bigoplus_{a\in E_{\Tilde{\psi}}} \Tilde{\psi}\vert \cdot \vert^ a.$$
    We want to show that there exists a unique partition $S_{\Tilde{\psi}}$ of $E_{\Tilde{\psi}}$ into segments, for all $\Tilde{\psi}$ such that
    $$\varphi = \bigoplus_{\Tilde{\psi}\in \mathrm{Irr}_u({W_E})}\bigoplus_{I\in S_{\Tilde{\psi}}} \Tilde{\psi} \otimes I$$
    satisfies the following conditions.
    \begin{enumerate}
        \item $\varphi^{ss}=\psi$, meaning that for all $\Tilde{\psi}\in \mathrm{Irr}_u({W_E})$, $S_{\Tilde{\psi}}$ is a partition of $E_{\Tilde{\psi}}$, in other words $\underset{I\in S_{\Tilde{\psi}}}{\bigsqcup}I = E_{\Tilde{\psi}}.$
        \item $\varphi$ is conjugate-orthogonal, that is:
        \begin{enumerate}
            \item For all $\Tilde{\psi}\in \mathrm{Irr}_u({W_E})$, $S_{\Tilde{\psi}^*}=-S_{\Tilde{\psi}}$, where $-S_{\Tilde{\psi}}$ is the multiset $\{[-b, -a], \ [a,b]\in S_{\Tilde{\psi}}\}$.
            \item For all $\Tilde{\psi}\in \mathrm{Irr}_u({W_E})$ and all $I\in S_{\Tilde{\psi}}$, if $\Tilde{\psi}^*\cong\Tilde{\psi}$, $I=-I$ and $\mathrm{sgn}(\Tilde{\psi})\mathrm{sgn}(I)=-1$, then $I$ appears with even multiplicity in $S_{\Tilde{\psi}}$.
        \end{enumerate}
        \item For all $\Tilde{\psi}\in \mathrm{Irr}_u({W_E})$, all $I\in S_{\Tilde{\psi}}$ and all $J\in S_{\Tilde{\psi}^*}$, if $l(I)\leq l(J)$ and $I\prec -J$, then $\Tilde{\psi}\cong \Tilde{\psi}^*$, $I=J$ and $I$ appears with multiplicity one in $S_{\Tilde{\psi}}$.
        \item For all $\Tilde{\psi}\in \mathrm{Irr}_u({W_E})$ and all $I\in S_{\Tilde{\psi}}$, if $\Tilde{\psi} \cong \Tilde{\psi}^*$ and $I\prec -I$ then $\mathrm{sgn(\Tilde{\psi})}\mathrm{sgn}(I\cup-I)=-1$.
    \end{enumerate}
    Assuming the first conditions, conditions 3 and 4 are equivalent to : for all $\Tilde{\psi}$, and all $I,J\in S_{\Tilde{\psi}}$, if $I$ and $J$ are linked then  $\Tilde{\psi}^*\cong\Tilde{\psi}$, $J=-I$, $I$ appears with multiplicity one in $S_{\Tilde{\psi}}$ and $\mathrm{sgn(\Tilde{\psi})}\mathrm{sgn}(I\cup-I)=-1$.

    \begin{lem}[Decomposition in maximal segments]
        Let $E \subset \mathbb{R}$ be a finite multiset. Then, there exists a unique partition of $E$ into segments, 
        $$E = I_1\cup \cdots \cup I_d,$$
        such that none of the segments $I_i$ are linked.
            
    \end{lem}
    
    \begin{proof}
        Let $x=\max(E)$, and take $y$ such that $[y, x] \subset E $ is maximal for this property. Set $E'=E\backslash [y, x]$, by induction $E'$ admits a unique partition $E'=I'_1\cup \cdots \cup I'_d$ such that none of the segments are linked. As $E=E'\cup [y, x]$, we have to show that for all $i\in \{1, \dots, d\}$, $I'_i$ is not linked to $[y, x]$. Suppose that there exists $i\in \{1, \dots, d\}$, such that $I'_i$ is linked to $[y, x]$, therefore there exists $z\in I'_i$, such that $z<y$ and $[z, x]\subset E$. But then $[z, x]$ is longer than $[y, x]$, which is not possible by assumption on $y$. Hence, the partition $E=I'_1\cup \cdots \cup I'_d\cup [y, x]$ works. To show the uniqueness of this decomposition, suppose that $E=I_1\cup \cdots\cup I_d$ is such that none of the $I_i$'s are linked and for all $i\in\{1, \dots, d\}$, $I_i \not = [y, x]$. Hence, there exist $i\not = j$ such that $y\in I_i$ and $x\in I_j$. As $E$ contains all the integers between $y$ and $x$, there exists $n\in \mathbb{N}$ such that $y+n \in I_i$ and $y+n+1\not \in I_i$ that is $y+n+1\in I_k$ with $k\not = i$. Then the segments $I_i$ and $I_k$ are linked, which contradicts the assumption on the $I_i$'s.
    \end{proof} 

    If $E_{\Tilde{\psi}}$ is the underlying multiset of a non conjugate-dual component $\Tilde{\psi}$, we apply the lemma to $E_{\Tilde{\psi}}$ and denote $S_{\Tilde{\psi}}$ the multiset of segments $I_i$ such that 
    $$E_{\Tilde{\psi}}=I_1\cup \cdots\cup I_d,$$
    where any two segments $I_i$, $I_j$ of $E_{\Tilde{\psi}}$ are not linked. We construct
    $$\varphi_{\Tilde{\psi}} = \bigoplus_{I\in S_{\Tilde{\psi}}} \Tilde{\psi}\otimes I,$$
    a component of $\varphi$ corresponding to $\Tilde{\psi}$. By assumption, the direct sum 
    $$\bigoplus_{a\in -E_{\Tilde{\psi}}} \Tilde{\psi}^*\vert \cdot \vert^a$$
    must appear in the decomposition of $\psi$, hence
    $$\varphi_{\Tilde{\psi}}^* =\varphi_{\Tilde{\psi}^*}= \bigoplus_{I\in S_{\Tilde{\psi}}} \Tilde{\psi}^{*} \otimes (-I)$$
    also appears in $\varphi$. In that case, the only L-factor that can have a pole at $s=1$ is 
    $$L(s, \Tilde{\psi}\otimes I \otimes \Tilde{\psi}^\vee \otimes -J )$$
    with $I,J\in S_{\Tilde{\psi}}$. This L-factor has a pole if and only if $I$ and $J$ are linked, but $J\in S_{\Tilde{\psi}}$ and the decomposition in maximal segments implies that $I$ and $J$ are not linked.

    \begin{lem}
        Let $\Tilde{\psi}\in \mathrm{Irr}_u(W_E)$ be conjugate-dual and let $E=E_{\Tilde{\psi}} \subset \mathbb{R}$ be the corresponding multiset. Then, there exists a unique partition $S_{\Tilde{\psi}}$ of $E$ into segments such that:
        \begin{enumerate}
            \item $S_{\Tilde{\psi}}=-S_{\Tilde{\psi}}$, 
            \item For all $I\in S_{\Tilde{\psi}}$, if $I=-I$ and $\mathrm{sgn}(\Tilde{\psi})\mathrm{sgn}(I)=-1$, then $I$ appears with even multiplicity in $S_{\Tilde{\psi}}$.
            \item For all $I,J\in S_{\Tilde{\psi}}$, if $I$ and $J$ are linked then $J=-I$, $I$ appears with multiplicity one in $S_{\Tilde{\psi}}$ and $\mathrm{sgn(\Tilde{\psi})}\mathrm{sgn}(I\cup-I)=-1$.           
        \end{enumerate}
    
    \end{lem}

    \begin{proof}
        Let $\Tilde{\psi}\in \mathrm{Irr}_u(W_E)$ be conjugate-dual, that is $\Tilde{\psi}\cong \Tilde{\psi}^*$ and therefore $E=-E$. We prove the existence and unicity of such a partition of $E$ by induction. Let $x=\max(E)$, and take $y\in E$ such that $[ y, x ] \subset E$ and $[ y,x ] $ is of maximal length. Then, there are two possibilities, either $y\not = -x$ or $y =-x$.
        \begin{itemize}
            \item  If $y\not = -x$, then by symmetry $[-x, -y]\subset E$ and the segments $[-x, -y]$ and $[ y, x ]$ are not linked. Indeed, if they were linked, then $[ -x, -y ] \cup [ y, x ] =[ -x, -x ] \subset E$, which is impossible by hypothesis. Set $E' = E\backslash ([ -x, -y ]\cup [ y, x ])$, we have $E'=-E'$. Suppose by induction that $E'$ admits a unique decomposition $E'=I_1\cup \cdots \cup I_d$ satisfying all the properties of the lemma, and denote by $S'$ this partition. Then $S = S'\cup \{[ -x, -y ], [ y, x ]\} $ is a partition of $E$, it follows directly that the first property is preserved for $S$. The second property is given by the fact that all the symmetric segments of the partition $S$ are in $S'$. It remains to show that $[ y, x ]$ is not linked to a $I_i$, for any $i\in\{1, \dots, d\}$ (it will imply also that $[ -x,-y]$ is not linked to any other segment in the decomposition). Assume that there exists $i\in \{1, \dots d\}$ such that $I_i$ is linked to $[ y,x]$, then $y-1\in I_i$ because $x=\max(E)$. But then $[ y-1, x] \subset E$ and its length is superior to the length of $[ y, x]$, which is not possible. Hence property 3 is also satisfied, because the only linked segments involved in the decomposition of $E$ are the ones inherited from $E'$. 
            
            Let us now prove the unicity. Suppose we have a partition $S_{\tilde{\psi}}$ of $E$ satisfying conditions 1, 2 and 3. To prove the unicity of $S_{\tilde{\psi}}$, it suffices to show, by induction again, that $ [x,y] \in S_{\tilde{\psi}}$. Let $I\in S_{\tilde{\psi}}$ containing $x$ and of maximal length. We have $\max I= x$ by assumption, and we can assume that $y = \min I-1$. As $[y,x] \subset E$, there exists $J\in S_{\Tilde{\psi}}$, such that $y\in J$, meaning that $I$ and $J$ are linked. By property 3, $J=-I$ so $-x \in J$ and the union of $I$ and $J$ is $[ -x, x ]$, which is not possible by assumption on $y$. That proves that $I=[ y,x]$, therefore the decomposition is unique.

            \item Assume now that $y=-x$, and $\Tilde{\psi}$ is conjugate-symplectic, i.e. $\mathrm{sgn}(\tilde{\psi})=-1$ which implies that $0$ has even multiplicity in $E$. We denote $m>0$ the largest integer such that $E$ contains all the elements of $[ -x, x ]$ with multiplicity at least $m$, and define $m[ -x,x] = [ -x,x] \sqcup \cdots \sqcup [ -x,x] $. 
            
            We have $\mathrm{sgn}[ -x,x]=1$, unless $x\in \frac{1}{2}\mathbb{Z}\backslash \mathbb{Z}$. Suppose $x\in \mathbb{Z}$, then $0\in[ -x,x]$. If the multiplicity $m$ of $[ -x,x]$ is odd, then there exists $z < x$ such that 
            $$[ -z,z]\subset E\backslash(m[ -x,x]),$$
            because $0$ has even multiplicity in $E$, and we take $z$ maximal for this property. Now set 
            $$E'=E\backslash ((m-1)[ -x,x]\sqcup[ -x,z]\sqcup [ -z,x]).$$
            We have $E'=-E'$ and the multiplicity of $0$ is still even. Suppose by induction, that $E'$ admits a unique decomposition $E'=I_1\sqcup \cdots \sqcup I_d$ satisfying all the properties of the lemma. Then the partition of $E$
            $$S_{\Tilde{\psi}}=\{I_1, \dots, I_d,(m-1)[ -x,x],[ -x,z],[ -z,x]\}$$
            automatically verifies properties 1 and 2. Note that if there exists $i\in \{1, \dots, d\}$ such that $I_i$ is linked to $[ -z, x]$ then either there exists an element $t\in I_i$ such that $t>x$ which is impossible, or there exists $t\in I_i$ such that $t<-z$ but in that case $-t>z$ which means that the segment $[ -t, t] \subset E\backslash(m[-x,x])$ is longer than $[ -z, z]$ and it is also impossible. If there exists $k\in \{1, \dots, d\}$ such that $I_k=[ -z,x]$ then $-I_k=[ -x,z]$ also appears but then $(m+1)[ -x,x]\subset E$ which is not possible by definition of $m$. This implies in particular that the third condition is still verified. Hence, the partition $S_{\Tilde{\psi}}$ satisfies all the conditions. 
            
            If $\mathrm{sgn}([ -x,x])=-1$, i.e. $x\in \frac{1}{2}\mathbb{Z}\backslash \mathbb{Z}$, then $0\not \in [ -x,x]$. We set 
            $$E'= E\backslash[ -x, x ]$$
            and decompose it as 
            $$E'=I_1\cup \cdots \cup I_d$$
            satisfying all properties by induction. As $[-x,x]$ contains all the $I_i$ by assumption, it follows directly that $E = I_1\cup \cdots \cup I_d \cup [-x,x]$ is a partition satisfying our conditions. 

            It remains to show that the partition $S_{\tilde{\psi}}$ given by the above algorithm is the unique one satisfying conditions 1, 2 and 3. Suppose that $E = I_1\cup \cdots \cup I_d$ is another decomposition such that properties 1, 2 and 3 are satisfied. Let $I_i = [ y, x]$ be a segment containing $x$. We show that $I_i\in S_{\tilde{\psi}}$, i.e. that either  $y=-x$ and the multiplicity $m_i$ of $I_i$ is even, or $y=-z$, where $[-z,z]$ is maximal in $E\backslash (m[-x,x])$, with $m_i=1$. Recall that here, we have $[-x,x]\subset E$. In that case, either $I_i =-I_i$, or $I_i$ and $-I_i$ are linked because there exists a segment $I_k$ containing $y-1$, hence linked to $I_i$ and property 3 says that $I_k=-I_i$. Thanks to the second property, if $I_i=-I_i$, i.e. $y=-x$, then $m_i$ is even. If $I_i$ is linked to $-I_i$ then by the definition of $z$ there exists a segment $I_k$ containing $z$, and $-z\leq y \leq z<x$, forcing $I_k$ to be linked with $I_i$. Using the third property we have $I_k=-I_i$, hence $-z=y$, and $m_i=1$. Therefore, $I_i\in S_{\tilde{\psi}}$, the partition $S_{\Tilde{\psi}}$ is unique.

            \item Assume that $y=-x$, and $\tilde{\psi}$ is conjugate-orthogonal, i.e. $\mathrm{sgn}(\tilde{\psi})=1$. 
            
            If $\mathrm{sgn}[-x,x]=-1$, then $x\in \frac{1}{2}\mathbb{Z}\backslash\mathbb{Z}$. Suppose $[-x,x]$ appears with odd multiplicity $m$. If there exists $z \in \frac{1}{2}\mathbb{Z}\backslash \mathbb{Z}, \  z<x$ such that 
            $$[ -z, z] \subset E\backslash m[-x, x],$$ 
            we take $z$ maximal for this property. Then we set 
            $$E' = E\backslash ((m-1)[ -x,x] \cup [ -x,z]\cup [ -z,x])$$
            and suppose by induction that $E'$ admits a partition satisfying properties 1, 2 and 3, $E' = I_1 \cup\cdots \cup I_d $. We have 
            $$E= I_1 \cup\cdots \cup I_d \cup (m-1)[ -x,x] \cup [ -x,z]\cup [ -z,x]$$
            satisfying automatically properties 1 and 2. If there exists $i\in \{1, \dots, d\}$ such that $I_i$ is linked to $[ -z, x]$ then either there exists an element $t\in I_i$ such that $t>x$ which is impossible, or there exists $t\in I_i$ such that $t<-z$ but in that case $-t>z$ which means that the segment $[ -t, t] \subset E$ is longer than $[ -z, z]$ and it is also impossible. If, there exists $k\in \{1, \dots, d\}$ such that $I_k=[ -z,x]$ then $-I_k=[ -x,z]$ also appears but then $(m+1)[ -x,x]\subset E$ which is not possible by definition of $m$. This implies in particular that the third condition is still verified. Hence, the partition $$S_{\Tilde{\psi}}=\{I_1,\dots, I_d, (m-1)[ -x,x],[ -x,z], [ -z,x]\}$$
            satisfies all the conditions. If such an element $z$ does not exists in $E$, then we set 
            $$E'=E\backslash ((m-1)[ -x,x]\cup[ -x,-\frac{1}{2}]\cup [ \frac{1}{2},x]).$$ 
            The induction works in the same way.

             If $[ -x,x]$ is of odd length, then we set $E'= E\backslash[ -x, x ]$ and decompose it as $E'=I_1\cup \cdots \cup I_d$ satisfying properties 1, 2, and 3, by induction. As $[-x,x]$ contains all the $I_i$ by assumption, it follows directly that $E = I_1\cup \cdots \cup I_d \cup [-x,x]$ is a partition satisfying our conditions. 

             The proof of the unicity of the partition is the same as in the case of $\Tilde{\psi}$ conjugate-symplectic.

        \end{itemize}
    
        \end{proof}
    Once we have these partitions $S_{\Tilde{\psi}}$ for all $\Tilde{\psi}\in \mathrm{Irr}_u(W_E)$ appearing in the decomposition of $\psi$, denote $\mathcal{C}$ the set of $\Tilde{\psi}$ that are conjugate-dual, and $\overline{\mathcal{C}}$ the set of those that are not conjugate-dual. We set 

    $$\varphi = \bigoplus_{\Tilde{\psi}\in \mathcal{C}} \bigoplus_{I \in S_{\Tilde{\psi}}} \Tilde{\psi} \otimes I\bigoplus_{\Tilde{\psi} \in \overline{\mathcal{C}}}\bigoplus_{J\in S_{\Tilde{\psi}}} \Tilde{\psi}\otimes J.$$
    As $\psi$ is conjugate-orthogonal, for all $\Tilde{\psi} \in \overline{\mathcal{C}}$, we also have $\Tilde{\psi}^* \in \overline{\mathcal{C}}$ and $S_{\tilde{\psi}^*}=-S_{\Tilde{\psi}}$, this implies that 
    $$\bigoplus_{\Tilde{\psi} \in \overline{\mathcal{C}}}\bigoplus_{J\in S_{\Tilde{\psi}}} \Tilde{\psi}\otimes J$$
    is conjugate-orthogonal. Moreover, for all $\Tilde{\psi}\in \mathcal{C}$, conditions 1 and 2 of the above lemma ensure that
    $$ \bigoplus_{I \in S_{\Tilde{\psi}}} \Tilde{\psi} \otimes I$$
    is conjugate-orthogonal. Therefore $\varphi$ is indeed conjugate-orthogonal. The third condition of lemma is equivalent to the fact that for all $\Tilde{\psi} \in \mathcal{C}$ 
    $$L(s, \bigoplus_{I \in S_{\Tilde{\psi}}} \Tilde{\psi} \otimes I, \mathrm{As}^-),  \ \mathrm{and} \ L(s, \Tilde{\psi}\otimes I \otimes \Tilde{\psi}^\vee \otimes J )$$
    with $I,J\in S_{\Tilde{\psi}}$, have no pole at $s=1$. We have already seen that for all $\Tilde{\psi}\in \overline{\mathcal{C}}$, the only L-factor that can have a pole is
    $$L(1, \Tilde{\psi}\otimes I \otimes \Tilde{\psi}^\vee \otimes -J )$$
    with $I,J\in S_{\Tilde{\psi}}$. This does not happen because of the partition in maximal segments. Finally, this Langlands' parameter $\varphi$ is unique thanks to the uniqueness of $S_{\Tilde{\psi}}$ for all $\Tilde{\psi}$.
    
\end{proof}

\begin{prop}\label{closedim}
    The map $CBC: \Omega_{\mathrm{gen}}(U(n)) \to  \Omega(G_n(E))$ is a closed immersion.
\end{prop}

\begin{proof}
    First of all, note that by the previous proposition, $CBC$ is injective, hence the images of two different connected components in $\Omega_{\mathrm{gen}}(U(n))$ don't intersect. Therefore, it suffices to show that the map from a connected component $\Omega\subset \Omega_{\mathrm{gen}}(U(n))$ to the corresponding component $\Omega' \subset \Omega(G_n(E))$ is a closed immersion.
    Let $[M, \sigma]$ be an element of $\Omega_{\mathrm{gen}}(U(n))$, with

    $$M=G_{n_1}\times \cdots \times G_{n_k} \times U(m),$$
    and we decompose 

    $$\sigma = \sigma_1 \boxtimes \cdots \boxtimes \sigma_k \boxtimes \sigma_u.$$
    We choose a pair $(L,\rho)$ as a representative of $\mathbf{Cs}(BC(\sigma_u))$, with $L$ a Levi subgroup of $G_m$, and $\rho$ an irreducible and supercuspidal representation of $L$, such that
    $$\mathcal{M}=G_{n_1}^2\times \cdots \times G_{n_k}^2 \times L, \ \mathrm{and} \ \tau = \sigma_1 \boxtimes \sigma_1^* \cdots \boxtimes \sigma_k \boxtimes \sigma_k^* \boxtimes \rho,$$
    is a representative of the image of $[M, \sigma]$ through $CBC$. If we set
    $$W(M, \sigma)=\{w \in W_G(M) : \ \exists \chi \in \mathcal{X}(M), \ \sigma^w = \sigma \otimes \chi \}$$
    the connected component $\Omega_\sigma$ can then be described as a quotient of a complex torus by a finite group
    $$\Omega_\sigma = \mathcal{X}(M)\otimes \sigma/W(M, \sigma).$$
    By adapting those definition to $G_n(E)$, we also get 
    $$\Omega_\tau =  \mathcal{X}(\mathcal{M})\otimes \tau/W(\mathcal{M}, \tau).$$
    Consider the following commuting diagram 

    $$\xymatrix{   \Omega_\sigma \ar[r]  & \Omega_\tau \\
    \mathcal{X}(M)\otimes \sigma \ar[u] \ar[r] & \mathcal{X}(\mathcal{M})\otimes \tau \ar[u]
    }$$
    where $\mathcal{X}(M)\otimes \sigma$ (resp. $\mathcal{X}(\mathcal{M})\otimes \tau$) is the orbit of $\sigma$ (resp. $\tau$) under the action of $\mathcal{X}(M)$ (resp. $\mathcal{X}(\mathcal{M})$), and the left and right maps are the canonical projections. The evaluation at a uniformizer of $E$ allows us to identify $\mathcal{X}(M)$ with $(\mathbb{C}^\times)^{k}$ and $\mathcal{X(M)}$ with $(\mathbb{C}^\times)^{2k+r}$. Therefore, the base-change map gives rise to the bottom map of the diagram above, defined by

    $$\begin{array}{cccc}
         & \mathcal{X}(M) \otimes \sigma & \longrightarrow & \mathcal{X}(\mathcal{M})\otimes \tau  \\
         & (\chi_1, \dots, \chi_k)\otimes \sigma & \longmapsto & (\chi_1, \chi_1^{-1} \dots, \chi_k,\chi_k^{-1}, 1, \dots ,1)\otimes \tau
    \end{array}$$
    
    Let us describe the corresponding map between $\Omega_\sigma$ and $\Omega_\tau$. Denote by $\Sigma$ the set of isomorphism classes of the $\sigma_i$'s and for all $\nu \in \Sigma$

    $$m_\sigma(\nu):= \vert \{1\leq i \leq k, \nu \cong \sigma_i \}\vert.$$
    Up to twisting $\sigma$ by an unramified character, we can assume that for all $i\in\{1, \dots, k\}$ either $\sigma_i\cong \sigma_i^*$, or $\sigma_i^* \not \in \mathcal{X}(G_{n_i})\otimes \sigma_i$, and that for all $i,j \in \{1, \dots, k\}$ either $\sigma_i \cong \sigma_j$ or $\sigma_i, \sigma_i^*\not \in \mathcal{X}(G_{n_j})\otimes \sigma_j$. Then we have 

    $$\Omega_\sigma \cong \prod_{\underset{\nu \cong \nu^*}{\nu \in \Sigma}} (\mathbb{C}^\times)^{m_\sigma(\nu)}/ (S_{m_\sigma(\nu)}\ltimes (\mathbb{Z}/2\mathbb{Z})^{m_\sigma(\nu)}) \times \prod_{\underset{\nu \not\cong \nu^*}{\nu \in \Sigma}} (\mathbb{C}^\times)^{m_\sigma(\nu)}/ S_{m_\sigma(\nu)},$$
    where for all integers $l$, $S_{l}$ is the permutation group that acts on $(\mathbb{C^\times})^l$ by 
    $$s \cdot (x_1, \dots, x_l) = (x_{s(1)}, \dots, x_{s(l)}),$$
    and the action of $(\mathbb{Z}/2\mathbb{Z})^{l}$ on $(\mathbb{C^\times})^l$ is given by the action of $\mathbb{Z}/2\mathbb{Z}$ on $\mathbb{C}^\times$, namely
    $$\overline{1}\cdot x = x^{-1}, \ x\in \mathbb{C}^\times$$
    $\overline{1}$ being the non-trivial element in $\mathbb{Z}/2\mathbb{Z}$.
    
    We can do the same in $G_n(E)$. If $\tau= \tau_1 \boxtimes \cdots \boxtimes \tau_{k'}$, with $k'=2k+r$, and $\Sigma'$ the set of isomorphism classes of the $\tau_i$'s. Again, up to twisting $\tau$ by an unramified character, we can assume that for all $i,j\in\{1, \dots, k'\}$ either $\tau_i\cong \tau_j$, or $\tau_i \not \in \mathcal{X}(G_{n_j})\otimes \sigma_j$. We get 

    $$\Omega_\tau \cong \prod_{\nu' \in \Sigma'} (\mathbb{C}^\times)^{m_\tau(\nu')}/ S_{m_\tau(\nu')}.$$
    In our case, we have $\tau = CBC(\sigma)$, hence $\Sigma \subseteq \Sigma'$, and we want to show that the corresponding map 
    $$\prod_{\underset{\nu \cong \nu^*}{\nu \in \Sigma}} (\mathbb{C}^\times)^{m_\sigma(\nu)}/ (S_{m_\sigma(\nu)}\ltimes (\mathbb{Z}/2\mathbb{Z})^{m_\sigma(\nu)}) \times \prod_{\underset{\nu \not\cong \nu^*}{\nu \in \Sigma}} (\mathbb{C}^\times)^{m_\sigma(\nu)}/ S_{m_\sigma(\nu)} \to  \prod_{\nu' \in \Sigma'} (\mathbb{C}^\times)^{m_\tau(\nu')}/ S_{m_\tau(\nu')}$$
    is a closed immersion. On one hand, a block of the form 
    
    $$(\mathbb{C}^\times)^{m_\sigma(\nu)}/ S_{m_\sigma(\nu)}$$
    is sent to a block of the form 
    
    $$(\mathbb{C}^\times)^{m_\sigma(\nu)}/ S_{m_\sigma(\nu)} \times (\mathbb{C}^\times)^{m_\sigma(\nu)}/ S_{m_\sigma(\nu)},$$
    through the closed immersion $(x_1, \dots,x_{m_\sigma(\nu)}) \mapsto (x_1, \dots,x_{m_\sigma(\nu)},x_1^{-1}, \dots,x_{m_\sigma(\nu)}^{-1})$.
    On the other hand, a block of the form 
    
    $$(\mathbb{C}^\times)^{m_\sigma(\nu)}/ (S_{m_\sigma(\nu)}\ltimes (\mathbb{Z}/2\mathbb{Z})^{m_\sigma(\nu)})$$
    is sent to a block of the form 
    $$(\mathbb{C}^\times)^{2m_\sigma(\nu)+\alpha_\nu}/ S_{2m_\sigma(\nu)+\alpha_\nu}$$
    through the map $(x_1, \dots,x_{m_\sigma(\nu)}) \mapsto (x_1,x_1^{-1} \dots,x_{m_\sigma(\nu)}, x_{m_\sigma(\nu)}^{-1}, 1, \dots, 1)$. We can decompose this map as 
    
    $$
     (\mathbb{C}^\times)^{m_\sigma(\nu)}/ (S_{m_\sigma(\nu)}\ltimes (\mathbb{Z}/2\mathbb{Z})^{m_\sigma(\nu)}) \to  (\mathbb{C}^\times)^{2m_\sigma(\nu)}/ S_{2m_\sigma(\nu)} \to  (\mathbb{C}^\times)^{2m_\sigma(\nu)+\alpha_\nu}/ S_{2m_\sigma(\nu)+\alpha_\nu}
     $$

     $$
     (x_1, \dots,x_{m_\sigma(\nu)}) \mapsto (x_1,x_1^{-1}, \dots,x_{m_\sigma(\nu)}, x_{m_\sigma(\nu)}^{-1}) \mapsto (x_1,x_1^{-1}, \dots,x_{m_\sigma(\nu)}, x_{m_\sigma(\nu)}^{-1}, 1, \dots, 1)
     $$

    Let us denote $m=m_\sigma(\nu)$, it remains to show that the map
    $$\begin{array}{cccc}
         & (\mathbb{C}^\times)^m/(S_m\ltimes (\mathbb{Z}/2\mathbb{Z})^m) & \longrightarrow & (\mathbb{C}^\times)^{2m}/S_{2m}  \\
         & (x_1, \dots,x_{m}) & \longmapsto & (x_1,x_1^{-1} \dots,x_{m}, x_{m}^{-1})
    \end{array}$$
    is a closed immersion. We identify $(\mathbb{C}^\times)^m/(S_m\ltimes (\mathbb{Z}/2\mathbb{Z})^m)$ with the set of unitary polynomials of degree $m$, namely
    $$\left\{Q(T) = \prod_{i=1}^m(T-(x_i+x_i^{-1})), \ x_i\in \mathbb{C}^\times\right\}.$$
    Similarly, we identify $(\mathbb{C}^\times)^{2m}/S_{2m}$ with unitary polynomials $P$ of degree $2m$ and such that $P(0) \not = 0$. The image of the map $(x_1, \dots,x_{m}) \mapsto (x_1,x_1^{-1} \dots,x_{m}, x_{m}^{-1})$ through this identification is given by
    $$\left\{P(T) = \prod_{i=1}^m(T-x_i)(T-x_i^{-1}), \ x_i\in \mathbb{C}^\times\right\}.$$
    Remark that 
    $$P(T)= T^{2m}P(T^{-1}),$$
    ensures that this image is indeed closed. Through these identifications, the map is given by
    $$\begin{array}{cccc}
         
         & Q & \longmapsto &  T^mQ(T+T^{-1}) = P(T) 
    \end{array}$$
    It is indeed a closed immersion because the inverse map is given by
    $$P \mapsto R(T) = T^{-m}P(T), $$
    and we have 
    $$R(T^{-1})=R(T).$$
    Hence $R$ is invariant under the inversion $T \mapsto T^{-1}$, in other words
    $$R \in \mathbb{C}[T, T^{-1}]^{\mathbb{Z}/2\mathbb{Z}} \cong \mathbb{C}[T+T^{-1}]\cong \mathbb{C}[T].$$

\end{proof}



\section{Algebraic family of representations}\label{FamRep}

In this section, we set $G=G_n(E)$. 

\subsection{$(G,X)$-modules}
We refer to the section VI.8 of \cite{Ren}. 

\begin{definition}[$(G, X)$-modules]
Let $X$ be a reduced commutative noetherian $\mathbb{C}$-algebra. A $(G,X)$-module is a $\mathbb{C}$-vector space $\mathcal{V}$ with a smooth representation $(\pi, \mathcal{V})$ of $G$ together with a structure of flat $X$-module where both actions commute.
   
\end{definition}

If $X=\mathbb{C}[\mathcal{X}(G)]$, then for all representation $(\pi, V)$ of $G$, the evaluation morphism  
$$\begin{array}{cccc}
     \mathrm{ev} : & G  & \longrightarrow & X^\times  \\
     & g & \longmapsto & \mathrm{ev}_g 
\end{array}
$$
allows us to define a $(G, X)$-module, $V_X = V \otimes_{\mathbb{C}} X$, such that for all $v \in V$, all $x \in X$ and all $g \in G$,
$$\pi_X(g)(v\otimes x)= \pi(g) v \otimes \mathrm{ev}_g \cdot x.$$
For $Q=MU$, and $(\sigma, W)$ a $(M, \mathbb{C}[\mathcal{X}(M)])$-module, then $i_Q^G\sigma$ is a $(G, \mathbb{C}[\mathcal{X}(M)]))$-module in a natural way.

Let $K_0$ be a maximal compact subgroup of $G$, $Q=MU$ a parabolic subgroup of $G$ and $\rho\in \mathrm{Irr}(M)$. We denote  
$$
V := i^{K_0}_{K_0\cap Q}(\rho) \cong i_Q^G(\rho \otimes \chi), \quad \chi \in \mathcal{X}(M).
$$

\begin{definition}
A function $F : \mathcal{X}(M) \to i_Q^G(\rho \otimes \chi)$ is regular if its image in $V$ is finite-dimensional and if it is regular in each variable.
\end{definition}

Let $X=\mathbb{C}[\mathcal{X}(M)]$ be the algebra of polynomial functions on the torus of unramified characters of $M$.

\begin{lem}\label{SectionReg}
The following statements are equivalent.
\begin{enumerate}
\item The function
$$
\begin{array}{cccc}
F : & \mathcal{X}(M) & \to & i_Q^G(\rho \otimes \chi) \\
& \chi & \mapsto & v_\chi
\end{array}
$$
is regular.

\item There exists a compact open subgroup $K \subset G$ such that $v_\chi$ is $K$-invariant for all $\chi \in \mathcal{X}(M)$, and for all $g \in G$ and all $e^\vee \in \rho^\vee$,

$$
\chi \mapsto \langle v_\chi(g), e^\vee \rangle
$$
is regular.

\item There exists $v \in i_Q^G(\rho_X)$ such that for all $\chi \in \mathcal{X}(M)$, $v_\chi = \mathrm{sp}_\chi(v)$, where $\mathrm{sp}_\chi$ denotes the canonical morphism from $i_Q^G(\rho_X)$ to $i_Q^G(\rho \otimes \chi)$.
\end{enumerate}
\end{lem}

\begin{proof}
Let us show that (1) implies (2). By definition of the induced representation, for every character $\chi \in \mathcal{X}(M)$ there exists a compact open subgroup $K_\chi \subset G$ such that $v_\chi$ is $K_\chi$-invariant. Since $K_0$ is maximal, we have $K_\chi \subset K_0$. By assumption, $\mathrm{Im}(F) \subset V$ is of dimension $d$, so we fix a basis $\{v_1, \dots, v_d\}$ and denote by $K_i$ a compact subgroup fixing $v_i$. Let $K'=\cap_{i=1}^d K_i$. Thus, $v_\chi$ is $K'$-invariant for all $\chi \in \mathcal{X}(M)$. Moreover, since $F$ is regular in each variable, the condition
$$
\chi \mapsto \langle v_\chi(g), e^\vee \rangle, \quad g \in G, \ e^\vee \in \rho^\vee
$$
is simply a reformulation.

Let us show that (2) implies (1). Let $K'$ be a compact open subgroup of $G$ such that $v_\chi$ is $K'$-invariant for all $\chi \in \mathcal{X}(M)$. Then for all $\chi \in \mathcal{X}(M)$,
$$
v_\chi \in i_Q^G(\rho \otimes \chi)^{K'} \cong V^{K'}.
$$
Hence, the image of $F$ is finite-dimensional. For all $g \in G$ and all $e^\vee \in \rho^\vee$, define
$$
\begin{array}{cccc}
\langle \mathbf{ev}_g(\cdot), e^\vee \rangle : & V & \longrightarrow & \mathbb{C} \\
& v_\chi & \longmapsto & \langle v_\chi(g), e^\vee \rangle
\end{array}
$$
By assumption, the composition
$$
\langle \mathbf{ev}_g(\cdot), e^\vee \rangle \circ F : \chi \mapsto \langle v_\chi(g), e^\vee \rangle
$$
is regular for all $g \in G$ and all $e^\vee \in \rho^\vee$. Since the family
$$
\{\langle \mathbf{ev}_g(\cdot), e^\vee \rangle, \ g \in G, \ e^\vee \in \rho^\vee\}
$$
generates $V^*$, the function $F$ is regular.

To prove the equivalence between (3) and (1), it suffices to show that (1) implies (3). Let $\chi \in \mathcal{X}(M)$ and denote by $\phi_\chi$ the isomorphism between $i_Q^G(\rho \otimes \chi)$ and $i^{K_0}_{K_0\cap Q}(\rho)$. Let $\{v_1, \dots, v_d\}$ be a basis of the image of $F$ in $V = i^{K_0}_{K_0\cap Q}(\rho)$, and $\{v_{\chi,1}, \dots, v_{\chi,d}\}$ the corresponding basis in $i_Q^G(\rho \otimes \chi)$. We know that
$$
\phi_\chi(v_\chi)= \sum_{i=1}^d \lambda_{i,\chi} v_i,
$$
with $\lambda_i : \chi \mapsto \lambda_{i,\chi}$ regular since $F$ is regular, by assumption. Set $v_{X,i} = v_i \otimes \lambda_i$. We have an isomorphism of $(G,X)$-modules
$$
\phi_X : i_Q^G(\rho_X) \to i_{K_0\cap Q}^{K_0}(\rho) \otimes X.
$$
For each $i \in \{1, \dots, d\}$, define
$$
w_i = \phi_X^{-1}(v_i \otimes \lambda_i).
$$
Then we obtain
$$
\mathrm{sp}_\chi(v_i) = \phi_\chi^{-1}(v_i) \otimes \lambda_{i,\chi} = \lambda_{i,\chi} v_{\chi,i}.
$$
This completes the proof of the lemma.
\end{proof}

\subsection{Twisted Jacquet module}

For $\pi$ a representation of $G$, let $\mathbf{j}_{\psi_n}(\pi) = \mathbf{j}_{\psi_n}^G(\pi)$ denote the \emph{twisted Jacquet module} that is, the quotient of $\pi$ by the sub-vector space
$$\langle \pi(u)v-\psi_n(u)v, \ u\in N_n, \ v \in \pi \rangle.$$
The functor $\pi \mapsto \mathbf{j}_{\psi_n}(\pi)$ is exact. Moreover, $\pi$ is generic if and only if $\mathbf{j}_{\psi_n}(\pi)$ is non-trivial in which case $\mathbf{j}_{\psi_n}(\pi)$ is one-dimensional if $\pi$ is irreducible. Recall that if $\pi=i_Q^G(\sigma)$, with $Q=MU$ and $\sigma$ is an irreducible representation of $M$, then $\pi$ is $(N_n,\psi_n)$-generic if and only if $\sigma$ is $(M\cap N_n, \psi_n)$-generic (\cite{Rodier}, Theorem 7). Thanks to Lemma 2.6 in \cite{LapidMao}, we have an isomorphism of vector spaces
$$\mathbf{j}_{\psi_n}(\pi) \cong \mathbf{j}_{\psi_{n}}^M(\sigma).$$

\subsection{Family of Jacquet's functionals}\label{famJacquet}
Let $Q=MU$ be a standard parabolic subgroup of $G$, and denote by $\overline{Q}=M\overline{U}$ the anti-standard parabolic subgroup opposite to $Q$. Let $ \sigma\in \mathrm{Irr_{gen}}(M)$, we define $X=\mathbb{C}[\mathcal{X}(M)]$ the algebra of polynomial functions on the torus of unramified characters of $M$. One can associate to $\sigma$, a $(M, X)$-module $\sigma_X = \sigma \otimes X$. As $\sigma$ is generic, there exists a non-trivial Whittaker functional $l: \sigma \to \mathbb{C}$. We then define
$$\begin{array}{cccc}
     l_X : & \sigma_X & \longrightarrow & X \\
     & v \otimes x & \longmapsto & l(v)x
\end{array}$$
It satisfies 
$$l_X(u. (v\otimes x)) = \psi_{n}(u)l(v)x, \  u\in M\cap N_n, \ v\otimes x\in V_X$$
because the evaluation map $ev_g: \chi \mapsto \chi(g)$ is trivial on $M\cap N_n$. Denote $\sigma_\chi=\sigma\otimes \chi$, for $\chi \in \mathcal{X}(M)$. There exists a constant $C>0$ such that, if 
$$\langle \Re(\chi), \alpha^\vee\rangle\leq -C, \ \forall \alpha^\vee \in \Delta_Q^\vee,$$
then for all $\varphi \in i_{\overline{Q}}^G(\sigma_\chi)$ the integral
    $$\int_U l(\varphi(u))\psi_n(u)^{-1}du$$
is absolutely convergent. We then define 
    $$\begin{array}{cccc}
        J_{\sigma_\chi,l}^{\psi_n}: & i_{\overline{Q}}^G(\sigma_\chi) & \longrightarrow & \mathbb{C} \\
         & \varphi & \longmapsto & \int_{U} l(\varphi(u))\psi_{n}(u)^{-1} du
    \end{array}$$
and extend it regularly to all unramified characters $\chi$ (see Section 2 \cite{CassShal}).
It allows us to define a \emph{family of Jacquet functionals} from $l_X$, 
$$\begin{array}{cccc}
        J_{\sigma_X,l_X}^{\psi_n}: & \mathbf{j}_{\psi_n}(i_{\overline{Q}}^G(\sigma_X)) & \longrightarrow & X \\
        
    \end{array}$$   
The Frobenius reciprocity gives a morphism
    $$\begin{array}{cccc}
         \Tilde{J}_{\sigma_\chi}^{\psi_n}: & i_{\overline{Q}}^G(\sigma_\chi) & \longrightarrow & C^\infty(N_n \backslash G, \psi_n) \\
         & v &\longmapsto & J_{\sigma_\chi}^{\psi_n}(i_{\overline{Q}}^G(\sigma_\chi)(\cdot)v)
    \end{array}
     $$
and we define $\mathcal{W}(i_Q^G(\sigma_ \chi), \psi_n)$ as the image of $\Tilde{J}_{\sigma_\chi}^{\psi_n}$ in $C^\infty(N_n \backslash G, \psi_n)$ (even when the representation $i_Q^G(\sigma_ \chi)$ is reducible).
\begin{lem}\label{Whittreg}
    Let $Q=MU$ be a standard parabolic subgroup of $G$ and $\sigma\in\mathrm{Irr}_{\mathrm{gen}}(M)$. Denote $\pi_\chi=i_Q^G(\sigma\otimes \chi)$. There exists a regular section $\chi\in \mathcal{X}(M)\mapsto v_\chi \in \pi_\chi$, such that $\Tilde{J}_{\sigma_\chi}^{\psi_n}(v_\chi) = W_{f, \pi_\chi}^1$, for some $f\in C_c^\infty(G)$.
\end{lem}

\begin{proof}
   
    We identify $\pi_\chi^\vee$ with $i_Q^G(\sigma^\vee\otimes \chi^{-1})$ via the pairing
    $$\langle v, v^\vee\rangle_{\pi_\chi} = \int_{Q\backslash G}\langle v(g), v^\vee(g) \rangle dg.$$
    If we denote $\mathcal{B}(\pi_\chi)$ a basis of $\pi_\chi^{K}$ (we will denote $v^\vee$ the corresponding elements in the dual basis of $\mathcal{B}(\pi_\chi)$), where $K$ is a compact-open subgroup of $G$ for which  $f$ is right-invariant, then we have 
    $$W_{f,\pi_\chi}^1(g)=\sum_{v \in \mathcal{B}(\pi_\chi)} \int^*_N \langle \pi_\chi(ug)\pi_\chi(f^\vee)v, v^\vee \rangle_{\pi_\chi}\psi_n(u)^{-1}du.$$
    By Proposition 2.8 in \cite{LapidMao}, 
    $$W_{f,\pi_\chi}^1(g)=\sum_{v \in \mathcal{B}(\pi_\chi)}J_{\sigma \otimes \chi}^{\psi_n}(\pi_\chi(g)\pi_\chi(f^\vee)v). J_{\sigma^\vee\otimes\chi^{-1}}^{\psi_n^{-1}}(v^\vee)$$
    but
    $$J_{\sigma_\chi}^{\psi_n}(\pi_\chi(g)\pi_\chi(f^\vee)v) = \Tilde{J}_{\sigma_\chi}^{\psi_n} (\pi_\chi(f^\vee)v)(g),$$
    therefore
    $$W_{f,\pi_\chi}^1(g)=\Tilde{J}_{\sigma_\chi}^{\psi_n}\left(\sum_{v \in \mathcal{B}(\pi_\chi)} J_{\sigma^\vee\otimes\chi^{-1}}^{\psi_n^{-1}}(v^\vee).\pi_\chi(f^\vee)v\right).$$
    
\end{proof}

\subsection{A family version of Gelfand-Kazhdan's theorem}\label{GK}
Let $P_n = P_n(F)$ and $N_n= N_n(F)$. The following result is a variant of Theorem 6 in \cite{GelfKaz}.
\begin{prop}[Family version of Gelfand-Kazhdan]\label{GKfam}
    Let A be a complex algebra of finite type, and denote by
    $$C^\infty(N_n\backslash P_n, A, \psi_n)$$
    the $(P_n, A)$-module of smooth functions $f:P_n\to A$ such that 
    $$f(up)= \psi_n(u)f(p), \ p\in P_n, \ u\in N_n$$
    on which $P_n$ acts by right translation and $A$ by pointwise multiplication, and both actions commute. If $V\subseteq C^\infty(N_n\backslash P_n, A, \psi_n)$ is a sub-$(P_n, A)$-module, such that for all $p\in P_n$, 
    $$ev_p(V)=\{f(p), \ f\in V\} = A,$$ 
    then $V$ contains $C_c^\infty(N_n\backslash P_n, A, \psi_n)$.    
\end{prop}

\begin{proof}
    Let $n\in \mathbb{N}^*$, and suppose that that the theorem is true for all $k<n$. Let $U_n$ be the unipotent radical of $P_n$ of matrices of the form 
    $$\begin{pmatrix} 1 & \cdots & 0 & * \\
                  \vdots & \ddots & \vdots & \vdots \\
                  0 & \cdots & 1 & * \\
                  0 & \cdots & 0 & 1 \end{pmatrix}$$
    such that $P_n = G_{n-1}U_n$ and $U_n\cong F^{n-1}$. We have a surjective map
    $$N_n\backslash P_n = N_{n-1}\backslash G_{n-1} \longrightarrow P_{n-1}\backslash G_{n-1}.$$
    The action of $G_{n-1}$ on $F^{n-1}$ admits two orbits : $\{0\}$ and $F^{n-1}\backslash \{0\}$. The mirabolic subgroup can be seen as a stabilizer under this action, 
    $$P_{n-1}=\mathrm{Stab}_{G_{n-1}}\left( (0, \dots, 0, 1) \right), $$
    then we have an isomorphism
    $$\begin{array}{cccc}
         & P_{n-1}\backslash G_{n-1} & \longrightarrow & F^{n-1}\backslash \{0\} \\
         & P_{n-1}g & \longmapsto & (0, \dots, 0, 1)g
    \end{array}$$
    Let $\varphi \in C_c^\infty(U_n, A)=C_c^\infty(F^{n-1}, A)$, $f\in V$ (note that $R(\varphi)f\in V$), and $g\in G_{n-1}$. We compute 
    \begin{align*}
        (R(\varphi)f)\begin{pmatrix}
            g & 0 \\ 0 & 1
        \end{pmatrix} & = \int_{F^{n-1}} \varphi(x_1, \dots, x_{n-1})f\left(\begin{pmatrix}
            g & \\ & 1
        \end{pmatrix} \begin{pmatrix}
            1 & 0 & \cdots & 0 & x_1 \\
             & \ddots & \ddots &\vdots & \vdots \\
             & & \ddots &0 & x_{n-1} \\
             & & & & 1
        \end{pmatrix}\right)d(x_1, \dots, x_{n-1}), \\
        & = \int_{F^{n-1}} \varphi(x_1, \dots, x_{n-1}) \psi\left(e_{n-1}g.\begin{pmatrix}
            x_1 \\ \vdots \\ x_{n-1}
        \end{pmatrix}\right)d(x_1, \dots, x_{n-1}) f\left(\begin{pmatrix}
            g & \\ & 1
        \end{pmatrix}\right), \\
        & = \hat{\varphi}(e_{n-1}g)f\left(\begin{pmatrix}
            g & \\ & 1
        \end{pmatrix}\right),
    \end{align*}
    where $\hat{\varphi}$ is the Fourier transform of $\varphi$ with respect to $\psi$. Recall that the Fourier transform 
    $$\varphi \mapsto \hat{\varphi}$$
    is an isomorphism of $C_c^\infty(F^{n-1}, A)$ to itself. By the identification $P_{n-1}\backslash G_{n-1}\cong F^{n-1}\backslash \{0\}$, we get 
    $$C_c^\infty(P_{n-1}\backslash G_{n-1}, A) \subset C_c^\infty(F^{n-1}, A).$$
    Therefore, this shows that the representation $V$ is fixed under multiplication by elements of $C_c^\infty(P_{n-1}\backslash G_{n-1}, A)$. 

    Now, suppose that $f\in C_c^\infty(N_n\backslash P_n,A, \psi_n)\cong C_c^\infty(N_{n-1}\backslash G_{n-1}, A, \psi_{n-1})$. The image of the support of $f$ is compact in $N_{n-1} \backslash G_{n-1}$, hence, considering the projection $$\mathbf{p}: N_{n-1}\backslash G_{n-1} \to P_{n-1}\backslash G_{n-1},$$ 
    its image is also compact in $P_{n-1}\backslash G_{n-1}$. An element of $\mathbf{p}(\mathrm{Supp}(f))$ is of the form $P_{n-1}g$ with $g\in G_{n-1}$, and 
    $$\mathbf{p}^{-1}(P_{n-1}g)=N_{n-1}\backslash P_{n-1} g \cong N_{n-1}\backslash P_{n-1}.$$
    We consider the $(P_{n-1}, A)$-submodule of $C^\infty(N_{n-1}\backslash P_{n-1}, A, \psi_{n-1})$ 
    $$V_g=\{p\in P_{n-1}\mapsto h(pg), \ h \in V\}.$$
    For all $p\in P_{n-1}$, 
    $$ev_p(V_g)=ev_{pg}(V)=A.$$    
    By induction, this space $V_g$ contains $C_c^\infty(N_{n-1}\backslash P_{n-1}, A, \psi_{n-1})$. Finally, 
    $$f\vert_{P_{n-1}g}\in \{p\in P_{n-1}\mapsto h(pg), \ h \in V\}.$$
    In conclusion, we have shown that for all $g\in G_{n-1}$, there exists a function $h_g\in V$ such that 
    $$h_g\vert_{P_{n-1}g}=f\vert_{P_{n-1}g}.$$
    By smoothness of these functions, we deduce that for all $g\in G_{n-1}$, there exists $U_g$ is a neighborhood of $g$ in $P_{n-1}\backslash G_{n-1}$, such that for all $g'\in U_g$
    $$h_g\vert_{P_{n-1}g'}=f\vert_{P_{n-1}g'}.$$
    As $\mathbf{p}(\mathrm{Supp}(f))$ is compact in $P_{n-1}\backslash G_{n-1}$, there exists $g_1, \dots, g_N\in \mathbf{p}(\mathrm{Supp}(f))$ such that 
    $$\mathbf{p}(\mathrm{Supp}(f))\subseteq \bigcup_{i=1}^N U_{g_i}.$$
    We can also choose disjoint compact open neighborhoods $V_i\subseteq U_{g_i}$ such that 
    $$\mathbf{p}(\mathrm{Supp}(f))\subseteq \bigsqcup_{i=1}^N V_{i}.$$ 
    And finally, for all $i$, there exists $h_i=h_{g_i}\in V$ such that 
    $$h_i\vert_{\mathbf{p}^{-1}(V_i)}=f\vert_{\mathbf{p}^{-1}(V_i)}$$
    and 
    $$\mathrm{Supp}(f)\subseteq \bigsqcup_{i=1}^N\mathbf{p}^{-1}(V_i).$$
    Then $f=\sum_{i=1}^N\mathbf{1}_{V_i}h_i \in V$, since $\mathbf{1}_{V_i}\in C_c^\infty(P_{n-1} \backslash G_{n-1}, A)$ and we have already seen that $V$ is invariant by multiplication by $C_c^\infty(P_{n-1} \backslash G_{n-1}, A)$.
\end{proof}

\begin{coro}\label{CoroGKfamKirrilov}
    Let $Q=LU$ be a parabolic subgroup of $G$, let $\rho\in \Irr{L}$ be generic, and let $X=\mathbb{C}[\mathcal{X}(L)]$. Then, the $(P_n, X)$-submodule $\mathcal{W}(i_Q^G(\rho_X), \psi_n)\subset C^\infty(N_n\backslash P_n, X, \psi_n)$, contains $C_c^\infty(N_n\backslash P_n, X, \psi_n)$.
\end{coro}

\begin{proof}
    By definition, the application of the previous result is almost straightforward. It suffices to show that for all $p\in P_n$, the ideal $ev_p(\mathcal{W}(i_Q^G(\rho_X), \psi_n))$ is equal to $X$. For all $\chi \in \mathcal{X}(L)$, denote by $m_\chi$ the corresponding maximal ideal in $X$, the specialization map 
    $$\mathrm{sp}_\chi : X \to X/m_\chi$$
    sends $\mathcal{W}(i_Q^G(\rho_X), \psi_n)$ to $\mathcal{W}(i_Q^G(\rho\otimes \chi), \psi_n)$ which is non-zero because $\rho$ is generic. Therefore, $ev_p(\mathcal{W}(i_Q^G(\rho_X), \psi_n))$ is not included in any of the maximal ideals of $X$, that is 
    $$ev_p(\mathcal{W}(i_Q^G(\rho_X), \psi_n)) = X.$$
    Finally, by Proposition \ref{GKfam}, we obtain 
    $$C_c^\infty(N_n\backslash P_n, X, \psi_n)\subset\mathcal{W}(i_Q^G(\rho_X), \psi_n)\subset C^\infty(N_n\backslash P_n, X, \psi_n).$$
\end{proof}

Thanks to Lemma \ref{Whittreg}, for all parabolic subgroup $Q=LU$ of $G$ and all $\rho\in \Irr{L}$ supercuspidal, the function 
$$\chi \in \mathcal{X}(L) \mapsto W_{f, i_Q^G(\rho\otimes \chi)}^1\vert_{P_n}$$
is regular. In particular, there exists a compact-open subgroup $K$ of $G$, such that for all $\chi \in \mathcal{X}(L)$
$$ W_{f, i_Q^G(\rho\otimes \chi)}^1\vert_{P_n} \in C^\infty(N_n\backslash P_n, \psi_n)^{P_n\cap K}.$$
Hence, thanks to Theorem VI.10.6 \cite{Ren}, we have

$$\pi \in \mathrm{Irr_{gen}}({G}) \mapsto W_{f, \pi}^1\vert_{P_n}$$
supported in a finite number of connected components of $\Omega(G)$. 

Let $\pi$ be an irreducible representation of $G$ of cuspidal support $[L, \rho]\in \Omega(G)$. Denote by 
$$\Omega_\rho \cong \mathcal{X}(L)/W(L,\rho),$$
the connected component of $[L, \rho]$, with 
$$W(L,\rho)=\{w\in W_G(L): \exists \chi \in \mathcal{X}(L), \ \rho^w = \rho \otimes \chi\}.$$

\begin{coro}\label{CoroGKfam}
    Let $X= \mathbb{C}[\mathcal{X}(L)]$ and $X'= X^{W(L, \rho)}$. There exists a function $f\in C_c^\infty(G)$, such that 
    $$W_{f, \pi_{X'}}^1 \in C_c^\infty(N_n\backslash P_n, X', \psi_n).$$   
\end{coro}

\begin{proof}
    Let 
    $$V:= \left\{W_{f, \pi_{X'}}^1 \vert_{P_n}, \ f \in C_c^\infty(G)\right\}\subset C^\infty(N_n\backslash P_n, X', \psi_n).$$
    To use the previous proposition, we show that $V$ is a $(P_n, X')$-module, such that for all $p\in P_n$, $ev_p(V)=X'$. Let $f\in C_c^\infty(G)$, for all $p\in P_n$, we have 
    $$R(p)W_{f, \pi_{X'}}^1 = W_{R(p)f, \pi_{X'}}^1,$$
    and $C_c^\infty(G)$ is stable under right translation by $P_n$. Recall that the action of $\mathfrak{Z}(G)$ on $C_c^\infty(G)$, seen as a $G\times G$-representation, identifies $\mathfrak{Z}(G)$ with the algebra of endomorphisms of $C_c^\infty(G)$ (\cite{BDKV}, 1; \cite{Ren}, I.1.7). In other words, for all $z\in \mathfrak{Z}(G)$ (that can be seen as a regular function on $\Omega(G)$), we have 
    $$\pi_{X'}(z.f) = z(\mathbf{Cs}(\pi))\pi_{X'}(f), \ \ \mathrm{where} \ z(\mathbf{Cs}(\pi))\in \mathbb{C}.$$
    Therefore, 
    $$X' = X^{W(L, \rho)} = \mathbb{C}[\mathcal{X}(L)/W(L, \rho)]= \mathbb{C}[\Omega_\rho],$$
    acts on $V$ by 
    $$z. W_{f, \pi_{X'}}^1 = W_{z.f, \pi_{X'}}^1,$$
    and this action commutes with the action of $P_n$. It remains to show that for all $p\in P_n$, 
    $$ev_p(V) = \{W_{f, \pi_{X'}}^1(p), \ f \in C_c^\infty(G)\}=X'.$$
    Thanks to Lemma \ref{Whittreg}, for all parabolic subgroup $Q=LU$ of $G$ and all $\rho\in \Irr{L}$ supercuspidal, the function 
    $$\chi \in \mathcal{X}(L) \mapsto W_{f, i_Q^G(\rho\otimes \chi)}^1\vert_{P_n}$$
    is regular. In particular, the following conditions are satisfied.
    \begin{enumerate}
        \item[-]  For all $p \in P_n$, the function 
        $$\chi \in \mathcal{X}(L) \mapsto W_{f, i_Q^G(\rho\otimes \chi)}^1(p)$$
        is regular.
        \item[-] Thanks to Proposition \ref{Wpi}, if a pair $(L', \rho' \otimes \chi')$ is conjugated to $(L, \rho \otimes \chi)$, then 
        $$W_{f, i_Q^G(\rho\otimes \chi)}^1 = W_{f, i_{Q'}^G(\rho'\otimes \chi')}^1.$$
    \end{enumerate}
    In other words, for all $p\in P_n$, $ev_p(V)$ is an ideal of $X'$, and for all $\chi \in \mathcal{X}(L)/W(L,\rho)$, 
    $$\mathrm{sp}_\chi(W_{f, \pi_{X'}}^1) = W_{f, \pi\otimes \chi}^1 \not = 0 $$
    for some $f\in C_c^\infty(G)$, meaning that $V$ is included in none of the maximal ideals of $X'$. Hence, for all $p\in P_n$, 
    $$ev_p(V) = \{W_{f, \pi_{X'}}^1(p), \ f \in C_c^\infty(G)\}=X'.$$
    Applying now Lemma \ref{GKfam}, we obtain 
    $$C_c^\infty(N_n\backslash P_n, X', \psi_n) \subset \left\{W_{f, \pi_{X'}}^1 \vert_{P_n}, \ f \in C_c^\infty(G)\right\}$$
    and hence, there exists a function $f\in C_c^\infty(G)$, such that 
    $$W_{f, \pi_{X'}}^1 \in C_c^\infty(N_n\backslash P_n, X', \psi_n).$$
\end{proof}

\subsection{A family version of the functional $\beta$}\label{PartBetaFam}

For any $p$-adic reductive group $G$, denote by $\mathcal{X}_{\mathrm{unit}}(G)$ the group of unitary and unramified characters of $G$. Set $G=G_n(E)$ and $H=G_n(F)$, and denote by $P_n$ the mirabolic subgroup of $G_n$. Let $Q=LU$ be a parabolic subgroup of $G$, and $\rho$ be a tempered and irreducible representation of $L$. For all $\mu \in \mathcal{X}(L)$, define
$$\begin{array}{cccccc}
    \beta^\rho_\mu : &  i_Q^G(\rho\otimes \mu) \otimes C_c^\infty (F^n) & \to &\mathcal{W}(i_Q^G(\rho\otimes \mu), \psi_n) \otimes C_c^\infty (F^n) & \to & \mathbb{C} \\
     & e \otimes \phi & \longmapsto & W_e\otimes \phi & \longmapsto & Z(1, W_e, \phi)
\end{array} $$
whenever 

$$Z(1, W_e, \phi)=\int_{N_n(F)\backslash G_n(F)} W_e(h) \phi(e_nh)\vert\det h \vert_Fdh$$
converges. Suppose $L=G_{n_1}(E)\times \cdots \times G_{n_k}(E)$, we identify $\mathfrak{a}^*_{L, \mathbb{C}}$ with $\mathbb{C}^k$, namely
$$\mathfrak{a}^*_{L, \mathbb{C}} = X^*(Z_L)\otimes_\mathbb{Z}\mathbb{C} \cong \mathbb{C}^k.$$
Combining Proposition 2.6.1 of \cite{RBPAsai}, applied to $\Tilde{J}_{\rho\otimes \mu}^{\psi_n}$, and Lemma 3.3.1 of loc. cit., we obtain that for all $\mu=(\mu_1, \dots, \mu_k)\in \mathfrak{a}^*_{L, \mathbb{C}}$, if for all $i=1, \dots, k$, $\Re(\mu_i)>-\frac{1}{2}$, the integral defining $\beta_\mu^\rho(e \otimes \phi)$ is absolutely convergent for all $e \otimes \phi \in i_Q^G(\rho\otimes \mu) \otimes C_c^\infty (F^n)$. The space $\mathcal{W}(i_Q^G(\rho\otimes \mu), \psi_n)$ is obtained via the Jacquet functional, namely it is the space $\Tilde{J}_{\rho\otimes \mu}^{\psi_n}(i_{Q}^G(\rho\otimes \mu))$. Recall that, in general
$$i_Q^G(\rho\otimes \mu)  \to \mathcal{W}(i_Q^G(\rho\otimes \mu), \psi_n)$$
is a surjection, but when is $i_Q^G(\rho\otimes \mu)$ irreducible, it is in fact an isomorphism.

\begin{definition}
    A regular section
    $$\mu \in \mathcal{X}(L) \mapsto W_\mu \in \mathcal{W}(i_Q^G(\rho\otimes \mu), \psi_n) ,$$
    is the composition of a regular section
    $$\mu \in \mathcal{X}(L) \mapsto e_\mu \in i_Q^G(\rho\otimes \mu),$$
    and the Jacquet functional $\Tilde{J}_{\rho\otimes \mu}^{\psi_n}$.

\end{definition}

\begin{prop} \label{BetaFam}
    Let $\phi \in C_c^\infty(F^n)$. Suppose $\rho$ is supercuspidal or induced by an irreducible and supercuspidal representation. Then, for every regular section 
    $$\mu \in \mathcal{X}(L) \mapsto W_\mu \in \mathcal{W}(i_Q^G(\rho\otimes \mu), \psi_n) ,$$
    the function $\mu \in \mathcal{X}_{\mathrm{unit}}(L) \mapsto \beta^\rho_\mu(W_\mu\otimes \phi)$ extends to a rational function on $\mathcal{X}(L)$, such that 
    $$\mu \in \mathcal{X}(L) \mapsto \frac{\beta^\rho_\mu(W_\mu\otimes \phi)}{L(1, i_Q^G(\rho\otimes \mu), \mathrm{As}^+)}$$
    is a regular function on $\mathcal{X}(L)$.
\end{prop}

To prove this proposition, we use Bernstein's theorem (Section 1, \cite{Banks}).

\begin{theo}[Bernstein]
    Let $\mathcal{D}$ be an algebraic variety over $\mathbb{C}$. Let $V$ be a complex vector space of countable dimension, $\Tilde{V}$ its linear dual, and 
    $$x_r \in V \otimes \mathbb{C}[\mathcal{D}], \ \ c_r\in\mathbb{C}[\mathcal{D}], \ \ r\in R$$
    with $R$ a set of indexes. Let $\mathcal{U}\subseteq \mathcal{D}$ be open (for the analytic topology) and non-empty. Suppose that for all $d\in \mathcal{U}$, the system 
    $$\lambda (x_r(d)) = c_r(d)$$
    admits a unique solution $\lambda_d\in \Tilde{V}$. Then for all $v\in V$, there exists a non-empty, Zariski-open subset $\mathcal{V}\subseteq \mathcal{D}$ such that the function 
    $$d\in \mathcal{V}\cap \mathcal{D} \mapsto \lambda_d(v)$$
    extends to a regular function on $\mathcal{V}$. In particular, it defines a rational function on $\mathcal{D}$.
\end{theo}

\begin{proof}[Proof of Proposition \ref{BetaFam}]
    First of all, if $\rho=i_P^L(\sigma)$ with $P=M_PU_P$ a parabolic subgroup of $L$ and $\sigma\in \mathrm{Irr}(M_P)$ supercuspidal, then 
    $$i_Q^G(\rho \otimes \mu) \cong i_{P}^G(\sigma \otimes \mu\vert_{M_P}).$$
    This implies that, up to replacing $\beta_\mu^\rho$ by $\beta_\mu^\sigma$, we can assume that $\rho$ is supercuspidal.

    Using the notation of Bernstein's theorem, set $\mathcal{D}= \mathcal{X}(L)$. For all $\mu \in \mathcal{X}(L)$, denote $\pi_\mu = i_Q^G(\rho \otimes \mu)$ and $\omega_{\pi_\mu}=\omega_\rho \otimes \mu\vert_{Z_G}$ its central character. Let $K_0$ be a maximal compact subgroup of $G$, recall that for all $\mu\in \mathcal{X}(L)$, $\pi_\mu$ is isomorphic to $ i_{Q\cap K_0}^{K_0}(\rho\vert_{L\cap K_0})$ as a vector space, set 
    $$V= i_{Q\cap K_0}^{K_0}(\rho\vert_{L\cap K_0}) \otimes C_c^\infty(F^n).$$
    We first remark that, for all $\mu \in \mathcal{X}_{\mathrm{unit}}(L)$
    $$\beta_\mu^\rho \in \mathrm{Hom}_{G_n(F)}(\pi_\mu \otimes C_c^\infty(F^n), \vert \det \vert_F^{-1}).$$
    From now on, we follow and adapt the proof of Theorem 1 in \cite{Kab}, to show that the space $\mathrm{Hom}_{G_n(F)}(\pi_\mu \otimes C_c^\infty(F^n), \vert \det \vert_F^{-1})$ is at most one-dimensional for $\mu$ in a non-empty Zariski-open subset of $\mathcal{X}(L)$. Let $H=G_n(F)$, the short exact sequence of $H$-representation,
    $$0 \longrightarrow C_c^\infty(F^n\backslash\{0\}) \longrightarrow C_c^\infty(F^n) \longrightarrow \mathbb{C} \to 0$$
    induces an exact sequence
    $$0 \to \mathrm{Hom}_H(\pi_\mu, \vert \det \vert_F^{-1}) \to \mathrm{Hom}_H(\pi_\mu \otimes C_c^\infty(F^n), \vert \det \vert_F^{-1}) \to \mathrm{Hom}_H(\pi_\mu \otimes C_c^\infty(F^n\backslash\{0\}), \vert \det \vert_F^{-1}).$$
    For all $\mu \in \mathcal{X}(L)$, such that 
    $$\mu\vert_{Z_H} \not = \vert \det \vert_F^{-1}\otimes \omega_\rho\vert_{F^\times}^{-1}$$
    we have
    $$\mathrm{Hom}_H(\pi_\mu, \vert \det \vert_F^{-1})  = 0.$$
    The set
    $$\mathcal{U}_0=\{\mu \in \mathcal{X}(L), \ \mu\vert_{Z_H} \not = \vert \det \vert_F^{-1}\otimes \omega_\rho\vert_{F^\times}^{-1}\}$$
    is Zariski-open. Moreover, 
    \begin{align*}
        \mathrm{Hom}_H(\pi_\mu \otimes C_c^\infty(F^n\backslash\{0\}), \vert \det \vert_F^{-1}) & \cong \mathrm{Hom}_H(\pi_\mu , i_{P_n(F)}^H(\mathbf{1})), \\
        & \cong \mathrm{Hom}_{P_n(F)}(\pi_\mu, \mathbf{1}).
    \end{align*}
    where $\mathbf{1}$ denotes the trivial character of $P_n(F)$. Let $\mu \in \mathcal{X}_{\mathrm{unit}}(L)$. Denote by
    $$\mathbf{\Psi}^- : \mathrm{Rep}(P_n(E)) \to \mathrm{Rep}(G_{n-1}(E)) \ , \ \mathbf{\Psi}^+:\mathrm{Rep}(G_{n-1}(E)) \to \mathrm{Rep}(P_{n}(E)) $$ 
    and 
    $$\mathbf{\Phi}^- : \mathrm{Rep}(P_n(E)) \to \mathrm{Rep}(P_{n-1}(E)) \ , \ \mathbf{\Phi}^+:\mathrm{Rep}(P_{n-1}(E)) \to \mathrm{Rep}(P_{n}(E))$$
    the derivative functors defined by \cite{BZ1} in Section 3. Using the results in Section 3.5 of \cite{BZ1}, $\pi_\mu$ as a representation of $P_n(E)$, $\pi_\mu$ admits a filtration 
    $$0=\pi_{n+1} \subset \pi_n \subset \cdots \subset \pi_2 \subset \pi_1 = \pi_\mu,$$
    such that, for all $j \in \{1, \dots , n\}$
    $$\pi_j/ \pi_{j+1} \cong \mathbf{\Phi}^{j-1} \circ \mathbf{\Psi}^+(\pi^{(j)}),$$
    where $\pi^{(j)}$ denotes the j-th derivative. Let us show that the natural map
    $$\mathrm{Hom}_{P_n(F)}(\pi_\mu, \mathbf{1}) \to \mathrm{Hom}_{P_n(F)}(\pi_n, \mathbf{1})$$
    is injective for $\mu$ in general position. Indeed, let $T$ be an element in the kernel of the above map. If $T$ is non-zero, then there exists $j \in \{1, \dots, n-1\}$ such that
    $$T \vert_{\pi_j} \not = 0, \ \mathrm{but} \ T\vert_{\pi_{j+1}} = 0.$$
    Hence, $T$ induces a non trivial element in 
    $$\mathrm{Hom}_{P_n(F)}(\pi_j/\pi_{j+1}, \mathbf{1}).$$
    Applying repeatedly Proposition 1 in \cite{Kab}, we obtain 
    $$\mathrm{Hom}_{P_n(F)}(\pi_j/\pi_{j+1}, \mathbf{1}) \cong \mathrm{Hom}_{P_{n-j+1}(F)}(\mathbf{\Psi}^+(\pi^{(j)}), \mathbf{1}).$$
    By definition of $\mathbf{\Psi}^+$ (Section 3.1, \cite{BZ1}), it is also isomorphic to 
    $$\mathrm{Hom}_{G_{n-j}(F)}(\pi^{(j)}, \mathbf{1}).$$

    We define 

    $$U_j := \begin{pmatrix}
        I_{n-j} & * \\ & N_j
    \end{pmatrix},$$
    and we see $\psi_j$ as a character of $U_j$ by composing with the projection from $U_j$ to $N_j$. The twisted Jacquet module 
    $$\pi_{U_j, \psi_j}= \mathbf{j}_{\psi_j}(\pi_\mu) =  \pi_\mu/\pi_\mu(U_j, \psi_j)$$
    is isomorphic to $\pi^{(j)}$. Set 
    $$V_j = \begin{pmatrix}
        I_{n-j} & * \\
         & I_j
    \end{pmatrix}.$$
    We also have 
    $$\pi^{(j)}\cong \mathbf{j}_{\psi_{j}}(\pi_{V_j}) = (\pi_{V_j})_{N_{j}, \psi_{j}}.$$
    Thanks to the geometric lemma \cite{Ren}, VI.5.1, the representation $\pi_{V_j}$ admits a filtration such that each subquotient admits a central character for the Levi $G_{n-j}(E)\times G_j(E)$. We denote by $\mathcal{W}_j$ a set of representative of the double cosets $Q\backslash G /Q'$, where $Q'$ is the parabolic subgroup with Levi $L_j=G_{n-j}(E)\times G_j(E)$ and unipotent radical $V_j$. Suppose $L=G_{n_1}(E)\times \cdots \times G_{n_k}(E)$. We identify $\mathcal{W}_j$ with the set of partitions $I_1\sqcup I_2 =\{1, \dots, k \}$ such that
    $$\sum_{i\in I_1}n_i=n-j, \ \mathrm{and } \ \sum_{i\in I_2}n_i=j,$$
    and for all $w\in \mathcal{W}_j$ we have
    $$w^{-1}Lw = \prod_{i\in I_1}G_{n_i}(E)\times \prod_{i\in I_2}G_{n_i}(E).$$
    Since $\rho$ is supercuspidal, the irreducible subquotients of $\pi_{V_j}$ arising from the geometric lemma are of the form 
    $$i_{L_j\cap w^{-1}Lw}^{L_j}\left(\rho\otimes \mu \right)^w, \  w \in \mathcal{W}_j$$
    We denote by 
    $$\omega_\rho \otimes \mu\vert_{Z_L} = (\omega_1 \mu_1, \dots, \omega_k \mu_k): (E^\times)^k \longrightarrow \mathbb{C}^\times$$
    the central character of $\rho\otimes\mu$. Therefore, the central characters of the irreducible subquotients of $\pi_{V_j}$ are characters of $E^\times \times E^\times$ of the form
    $$(\omega_\rho \otimes \mu\vert_{Z_L})^w = \left(\prod_{i\in I_1}\omega_{i} \mu_{i},\prod_{i\in I_2}\omega_{i} \mu_{i}\right), \ w \in \mathcal{W}_j.$$
    The central characters of $\pi^{(j)}$ correspond to the first entry
    $$\prod_{i\in I_1}\omega_{i} \mu_{i}.$$

    \bigskip 
    
    For all $w \in \mathcal{W}_{j}$, we denote by $\mathcal{U}_w$ the Zariski-open set of $\mu \in \mathcal{X}(L)$, such that 
    $$\prod_{i\in I_1}\mu_{i}\vert_{F^\times} \not = \left(\prod_{i\in I_1}\omega_{i}\vert_{F^\times}\right)^{-1}.$$
    Then for all $\mu\in \bigcup_{w\in \mathcal{W}_{j}} \mathcal{U}_w$, we have 
    $$\mathrm{Hom}_{G_{n-j}(F)}(\pi^{(j)}, \mathbf{1})=0.$$
    which establish the injectivity of the map 
    $$\mathrm{Hom}_{P_n(F)}(\pi_\mu, \mathbf{1}) \to \mathrm{Hom}_{P_n(F)}(\pi_n, \mathbf{1}).$$
    Using again Proposition 1 in \cite{Kab} and the definition of $\mathbf{\Psi}^+$,
    $$\mathrm{Hom}_{P_n(F)}(\pi_n, \mathbf{1}) \cong \mathrm{Hom}_{G_0(F)}(\pi^{(n)}, \mathbf{1}).$$
    The dimension of this last homomorphism space is one because it is the dimension of $\pi^{(n)}$, as $G_0(F)$ is the trivial group. Consequently, the space $\mathrm{Hom}_{P_n(F)}(\pi_\mu, \mathbf{1})$ and hence, the space 
    $$\mathrm{Hom}_{G_n(F)}(\pi_\mu \otimes C_c^\infty(F^n), \vert \det \vert_F^{-1}),$$
    are at most one-dimensional.
    
    Returning to the proof of the lemma, denote by $\mathcal{B}$ a basis of the vector space $V$. For all $e\otimes \phi \in \mathcal{B}$ there exists a compact-open subgroup $K_e\subset K_0$ of $G$, such that 
    $e \in \pi_\mu^{K_e}$, and $ H/(H\cap K_e)$ is a countable set, denote 
    $$R = \{(e \otimes \phi, h): \ e\otimes \phi \in \mathcal{B}, \ h \in H/(H\cap K_e)\}.$$ 
    Set $x_{(e \otimes \phi, h)}(\mu) = \pi_\mu(h)e \otimes R(h)\phi-\vert \det h\vert_F^{-1}.e \otimes \phi$. We have $x_{(e \otimes \phi, h)} \in V \otimes \mathbb{C}[\mathcal{D}]$. For all $\mu \in \mathcal{X}_{\mathrm{unit}}(L)$, the system
    $$L_\mu(x_{(e \otimes \phi, h)}(\mu) ) =0, \ (e \otimes \phi, h) \in R, \ L_\mu \in \Tilde{V}$$
    admits $\alpha\beta_\mu^\rho$ as solutions, with $\alpha \in \mathbb{C}$. In order to apply Bernstein's theorem we add another equation to system, to obtain $\beta_\mu^\rho$ as the unique solution. 
    
    Let $f \in C_c^\infty(N_n(E)\backslash P_n(E), \mathbb{C}[\mathcal{D}], \psi_n)$, such that
    $$\int_{N_n(F)\backslash P_n(F)}f(p)dp \not = 0.$$
    Thanks to Corollary \ref{CoroGKfamKirrilov}, there exists a regular section 
    $$e': \mu \in \mathcal{D}\mapsto e'(\mu) \in i_{Q \cap K_0}^{K_0}(\rho)^J,$$
    with $J$ a compact-open subgroup of $G$, such that
    $$W_{e'(\mu)}\vert_{P_n(E)} = \mathrm{sp}_\mu(f).$$
    Choosing $\phi'\in C_c^\infty(F^n)$ with support in $e_n(J \cap H)$, with $e_n = (0, \dots, 0, 1)$, and such that $\phi'(0)\not = 0$, we then have 
    $$Z(1, W_{e'(\mu)}, \phi') = \phi'(0) \int_{N_n(F)\backslash P_n(F)}W_{e'(\mu)}(p) dp \not = 0$$
    convergent for all $\mu \in \mathcal{D}$. Hence, 
    $$\mu \mapsto Z(1, W_{e'(\mu)}, \phi') \in \mathbb{C}[\mathcal{D}].$$  
    For all $\mu \in \mathcal{X}_{\mathrm{unit}}(L)$, the system 
    $$L_\mu(x_{(e \otimes \phi, h)}(\mu) ) =0, \ (e \otimes \phi, h) \in R,$$
    $$L_\mu(e'(\mu)\otimes \phi') = Z(1, W_{e'(\mu)}, \phi')$$
    admits a unique solution, $\beta_\mu^\rho$. Applying Bernstein's theorem, with $\mathcal{U}$ be the open set defined by the intersection 
    $$\mathcal{U} = \mathcal{U}_0\cap (\bigcup_{w\in \mathcal{W}_{j}}\mathcal{U}_w)\cap\left\{\mu=(\mu_1, \dots, \mu_k)\in \mathcal{X}(L), \ \Re(\mu_i)>-\frac{1}{2}\right\},$$
    we obtain that for every regular section $\mu \in \mathcal{X}(L) \mapsto W_\mu$,
    $$(\mu \in \mathcal{X}(L) \mapsto \beta_\mu^\rho(W_\mu\otimes \phi')) \in \mathbb{C}(\mathcal{X}(L)).$$
    It remains to show that 
    $$\mu \in \mathcal{X}(L) \mapsto \frac{\beta^\rho_\mu(W_\mu\otimes \phi')}{L(1, i_Q^G(\rho\otimes \mu), \mathrm{As}^+)}$$
    is a regular function on $\mathcal{X}(L)$. To do so, we use the following lemma.
    
    \begin{lem}
        Let $X$ be a complex, affine, irreducible and normal algebraic variety. We denote by $\mathbb{C}[X]$ the algebra of regular functions on $X$ and $\mathbb{C}(X)$ its field of fractions. Let $(\Delta_x)_{x\in X}$ be a family of closed and irreducible subvarieties of $X$, all of dimension one such that $x\in \Delta_x$, for all $x\in X$. If $f\in \mathbb{C}(X)$ is such that for all $x\in X$, none of the polar divisors of $f$ contains $\Delta_x$ and $f\vert_{\Delta_x}\in \mathbb{C}[\Delta_x]$, then $f\in \mathbb{C}[X]$.
    \end{lem}

    \begin{proof}
        For any prime ideal $P$ of $\mathbb{C}[X]$ denote by $V(P)$ the subvariety associated to $P$. Let $P$ be a prime ideal such that $V(P)$ is of codimension 1, the localization by $P$ is denoted by
        $$\mathbb{C}[X]_P=(\mathbb{C}[X]\backslash P)^{-1}.\mathbb{C}[X]\subset \mathbb{C}(X).$$
        Since $X$ is normal, we have 
        $$\mathbb{C}[X] = \bigcap_{P} \mathbb{C}[X]_P,$$
        where $P$ ranges over prime ideals of height one. Therefore it suffices to show that for all prime ideal $P$ of height one, 
        $$f\in \mathbb{C}[X]_P.$$
        Let $U_f=X \backslash \overline{\{x \in X, \ f(x)=0\}}$ be the fundamental open set associated to $f$. Let $P$ be a prime ideal of $\mathbb{C}[X]$, with $V(P)$ of codimension 1, such that $v_P(f)<0$. Let $x\in V(P) \cap U_f$, and $m_x\subset \mathbb{C}[X]$ the prime ideal corresponding to $x$, then we have $P\subset m_x$ and $P_x\subset m_x$ with $P_x$ the prime ideal corresponding to $\Delta_x$. Let $g=f^{-1}$, we then have $g(x)=0$, and we show that $g\vert_{\Delta_x}(x)=0$. As $x\in V(P) \cap U_f$, we have $g(x)=0$, or equivalently $g\in m_x.\mathbb{C}[X]_{m_x}$ (the unique maximal ideal of $\mathbb{C}[X]_{m_x}$). We have a projection 
        $$ \mathbb{C}[X]_{P_x} \longrightarrow \mathbb{C}[X]_{P_x}/P_x.\mathbb{C}[X]_{P_x} = \mathbb{C}(\Delta_x).$$
        And $m_x.\mathbb{C}[X]_{m_x}\subset \mathbb{C}[X]_{P_x}$. We denote $\overline{m_x}= m_x/P_x$ the maximal ideal corresponding to $x$ in $\mathbb{C}[\Delta_x]=\mathbb{C}[X]/P_x$. We have $\overline{m_x}.\mathbb{C}[\Delta_x]_{\overline{m_x}}\subset \mathbb{C}(\Delta_x)$, and this ideal is the image of $m_x.\mathbb{C}[X]_{m_x}$ by the above projection. We deduce that $g\vert_{\Delta_x}(x)=0$ and hence that $g\vert_{\Delta_x}\in \mathbb{C}[\Delta_x]_{\overline{m_x}}\backslash \mathbb{C}[\Delta_x]_{\overline{m_x}}^\times$. As $\mathbb{C}[\Delta_x]_{\overline{m_x}}$ is a local and integral ring, we have $g^{-1}=f\not \in\mathbb{C}[\Delta_x]_{\overline{m_x}}$, which means that $f$ has a pole in $\Delta_x$ and this contradicts the hypothesis $f\vert_{\Delta_x}\in \mathbb{C}[\Delta_x]$.
    \end{proof}
    We apply this lemma, with $\Delta_\mu=\{\mu_s= \mu \otimes \vert \det \vert_E^s, \ s \in \mathbb{C}\}$, for all $\mu \in \mathcal{X}(L)$. Recall that 
    $$i_Q^G(\rho\otimes \mu_s) \cong i_Q^G(\rho\otimes \mu)\otimes \vert\det\vert_E^s.$$
    Hence,
    $$\mathcal{W}(i_Q^G(\rho\otimes \mu_s), \psi_n) \cong \mathcal{W}(i_Q^G(\rho\otimes \mu), \psi_n)\otimes \vert \det \vert_E^s,$$
    implying that $W_{\mu_s}$ can be written $\vert \det \vert_E^sW_{\mu, s}$ with $W_{\mu, s}\in \mathcal{W}(i_Q^G(\rho\otimes \mu), \psi_n)$. Remark that for all $s \in \mathbb{C}$
    $$\frac{\beta^\rho_{\mu_s}(\vert\det\vert_E^sW_{\mu,s}\otimes \phi)}{L(1, i_Q^G(\rho\otimes \mu_s), \mathrm{As}^+)}= \frac{Z(1+2s, W_{\mu,s}, \phi)}{L(1+2s, i_Q^G(\rho\otimes \mu), \mathrm{As}^+)}$$
    is indeed regular. This means that to show the regularity of the function 
    $$\mu_s \in \Delta_\mu \mapsto \frac{\beta^\rho_{\mu_s}(W_{\mu_s}\otimes \phi)}{L(1, i_Q^G(\rho\otimes \mu_s), \mathrm{As}^+)}$$
    it suffices to show that 
    $$s \in \mathbb{C} \mapsto\frac{Z(s, W_{\mu,s}, \phi)}{L(s, i_Q^G(\rho\otimes \mu), \mathrm{As}^+)}$$
    is regular. Since 
    $$s\in \mathbb{C}\mapsto W_{\mu,s}\in \mathcal{W}(i_Q^G(\rho\otimes\mu), \psi_n)$$ 
    is a regular section, there exists a compact-open subgroup $K$ of $G$ such that for all $s\in \mathbb{C}$, 
    $$W_{\mu, s}\in \mathcal{W}(i_Q^G(\rho\otimes \mu), \psi_n)^K.$$
    Therefore, there exists a basis $\mathcal{B}_K$ of $ \mathcal{W}(i_Q^G(\rho\otimes \mu), \psi_n)^K$ such that $W_{\mu,s}$ can be written as a linear combination
    $$W_{\mu,s} = \sum_{W \in \mathcal{B}_K}\lambda_W(s)W,$$
    with $s\mapsto \lambda_W(s)$ a regular function. Finally, by definition of the Asai $L$-function the function 
    $$s \in \mathbb{C} \mapsto\frac{Z(s, W_{\mu,s}, \phi)}{L(s, i_Q^G(\rho\otimes \mu), \mathrm{As}^+)}$$
    is indeed regular.
\end{proof}

\begin{prop}\label{BetaProportional}
    Let $Q =LU$ be a standard parabolic subgroup of $G$. Let $\rho$ be an irreducible supercuspidal and unitary representation of $L$. Let $\lambda=(\lambda_1, \dots, \lambda_k) \in \mathcal{X}(L)$ be real and anti-dominant with respect to $Q$, i.e. $\lambda_1\leq \cdots \leq \lambda_k$. The space $$\mathrm{Hom}_H(i_Q^G(\rho\otimes \lambda), \mathbb{C})$$ is at most one-dimensional.
\end{prop}

\begin{proof}
  
  To simplify the notation, for all integers $r$, set $G_r=\mathrm{GL}_r(E)$ and $H_r=\mathrm{GL}_r(F)$.
  Write $L = G_{n_1}\times \cdots \times G_{n_k}$, with $(n_1, \dots, n_k)$ a partition of $n$, and 
  $$\rho\otimes \lambda = \rho_1\vert\det\vert_E^{\lambda_1} \boxtimes \cdots \boxtimes \rho_k\vert\det\vert_E^{\lambda_k}$$
  with $\rho_i$ irreducible unitary and supercuspidal and $\lambda_i\in \mathbb{R}$. Let $\{\mu_1, \dots, \mu_l\}$ be the set of distinct elements of $\{\lambda_1, \dots, \lambda_k\}$, hence $\mu_1<\cdots<\mu_l$. For all $j\in \{1, \dots, l\}$, we define 
  $$\Sigma_j = \bigtimes_{\underset{\lambda_i = \mu_j}{1\leq i\leq k,}} \rho_{i} \in \Irr{G_{\sum n_{i,j}}}$$
  where $n_{i,j} = n_j$ if $\lambda_j=\lambda_i$ (otherwise $n_{i,j}=0$); so that 
  $$i_Q^G(\rho\otimes \lambda) = i_{Q'}^G(\Sigma_1\vert \det\vert^{\mu_1}\boxtimes \cdots \boxtimes \Sigma_l\vert \det\vert^{\mu_l})  $$
  with $\mu_1<\cdots<\mu_l$, and $Q' = MU'$ the standard parabolic subgroup such that 
  $$M = G_{d_1}\times \cdots \times G_{d_l}, \ \mathrm{with} \ \  d_i =\sum_{j=1}^k n_{i,j}.$$ 

  We now use Proposition 5.4 of \cite{LgeoMatringe}. Following this paper, we parametrize the double cosets in $Q'\backslash G/H$ by symmetric matrices $a=(a_{i,j})_{i,j}$ of size $l$ and with non-negative integer coefficients, such that for all $i\in \{1, \dots, l\}$, $\sum_{j=1}^la_{i,j}=d_i$. Every line of the matrix $a$ defines a partition of $d_i$, and we set 
  $$M_u=M_{u,1}\times \cdots \times M_{u,l},$$
  where each $M_{u,i} = G_{a_{i,1}}\times \cdots \times G_{a_{i,l}}$ is a Levi subgroup of $G_{d_i}$. Again, according to \cite{LgeoMatringe}, we let $M_u^u$ be the subgroup

  $$M_u^u= \prod_{i=1}^lH_{a_{i,i}}\times \prod_{i<j}G_{a_{i,j}}$$
  embedded in $M_u$ through the map that sends an element $((h_i)_{i=1}^l, (g_{i,j})_{i<j})\in M_u^u$ to the element $m \in M_u$ whose coordinates verify
  $$ m_{i,j} = \left\{\begin{array}{cc}
       h_i, &  \mathrm{if} \ i=j\\
       g_{i,j}, & \mathrm{if} \ i<j \\
       g_{i,j}^c, & \mathrm{if} \ i>j
  \end{array}\right.$$
  The geometric lemma, proposition 5.4 of \cite{LgeoMatringe}, provides an injection of $\mathrm{Hom}_H(i_Q^G(\rho\otimes \lambda), \mathbb{C})$ into 
  $$\prod_{u\in Q'\backslash G/H} \mathrm{Hom}_{M_u^u}(r_{M_u}^M(\Sigma_1\vert \det\vert^{\mu_1}\boxtimes \cdots \boxtimes \Sigma_l\vert \det\vert^{\mu_l}), \mathbb{C}).$$
  We parametrize the real unramified characters of $M_u$ by 
  $$\bigoplus_{\underset{a_{i,j \not = 0}}{1\leq i,j \leq l}}\mathbb{R},$$
  where $\lambda = (\lambda_{i,j})_{1\leq i,j \leq l}$ corresponds to the character
  $$(g_{i,j}) \in M_u \mapsto \prod_{i,j}\vert \det g_{i,j}\vert^{\lambda_{i,j}}.$$
  Lemma \ref{expUnit} implies that the real part of the exponents of $r_{M_u}^M(\Sigma_1\vert \det\vert^{\mu_1}\boxtimes \cdots \boxtimes \Sigma_l\vert \det\vert^{\mu_l})$ is of the form $(\mu_1, \dots, \mu_1, \mu_2, \dots, \mu_2, \dots, \mu_l, \dots, \mu_l)$ (each $\mu_i$ appearing $l$ times). Suppose that 
  $$\mathrm{Hom}_{M_u^u}(r_{M_u}^M(\Sigma_1\vert \det\vert^{\mu_1}\boxtimes \cdots \boxtimes \Sigma_l\vert \det\vert^{\mu_l}), \mathbb{C}) \not =0.$$
  Then, in particular, the restriction to $M_u^u$ of the exponents of $r_{M_u}^M(\Sigma_1\vert \det\vert^{\mu_1}\boxtimes \cdots \boxtimes \Sigma_l\vert \det\vert^{\mu_l})$ should be trivial.  So, if there exists an index $i$ such that $a_{i,i}\not = 0$, we get $\mu_i = 0$. As all the $\mu_i$'s are different, if this $i$ exists, it is unique. Moreover, if for $i\not =j$, we have $a_{i,j}\not = 0$ then $\mu_i+\mu_j=0$. Hence, the only double coset $u \in Q' \backslash G /H$, for which the space $\mathrm{Hom}_{M_u^u}(r_{M_u}^M(\Sigma_1\vert \det\vert^{\mu_1}\boxtimes \cdots \boxtimes \Sigma_l\vert \det\vert^{\mu_l}), \mathbb{C})$ can possibly be non zero, is the one corresponding to the matrix 
  $$a = \begin{pmatrix}
        0 & \cdots & d_1 \\
          & \revdots & \\
        d_l & \cdots & 0  
    \end{pmatrix}$$
  which implies in particular $d_i = d_{l-i+1}$, for all $i$. For all $i$, denote by $G_{d_i}^{\mathrm{diag^c}}$ the subgroup consisting of elements $(g, g^c)$, with $g\in G_{d_i}$. Denote $\alpha=\lfloor l/2\rfloor+1$ if $l$ is odd, $\alpha=0$ if $l$ is even. It remains to show that the space 
  $$\mathrm{Hom}_{M^u}(\Sigma_1\vert \det\vert^{\mu_1}\boxtimes \cdots \boxtimes \Sigma_l\vert \det\vert^{\mu_l}, \mathbb{C}) = \mathrm{Hom}_{H_{d_\alpha}} (\Sigma_\alpha, \mathbb{C}) \otimes \bigotimes_{i=1}^{\lfloor l/2\rfloor} \mathrm{Hom}_{G_{d_i}^{\mathrm{diag}^c}}(\Sigma_i\vert \det\vert^{\mu_i}\boxtimes \Sigma_{l-i+1}\vert \det\vert^{-\mu_i}, \mathbb{C})$$
  is at most one-dimensional. We first remark that if $\alpha \not = 0$, then
  $$\mathrm{Hom}_{H_{d_\alpha}}(\Sigma_\alpha, \mathbb{C})$$
  is known to be at most one-dimensional. Finally, the space
  $$\mathrm{Hom}_{G_{d_i}^{\mathrm{diag^c}}}(\Sigma_i\vert \det\vert^{\mu_i}\boxtimes \Sigma_{l-i+1}\vert \det\vert^{-\mu_i}, \mathbb{C})$$
  is non zero if and only if $\Sigma_i \cong \Sigma_{l-i+1}^*$ and in that case the dimension is one.
    
\end{proof}

\section{A trace Paley-Wiener theorem}\label{PW}

In this section, $E/F$ is a quadratic extension of non-Archimedean local fields, $G=G_n(E)$ and $H=G_n(F)$.

\subsection{Relative characters for $\mathrm{GL}_n(F)\backslash \mathrm{GL}_n(E)$ }

For every irreducible representation $\pi$ of $G$, we denote by $\mathrm{End}^{\infty}(\pi)$ the space of endomorphisms on $\pi$ that are biinvariant by a compact-open subgroup of $G$. For all $f\in C_c^\infty(G)$, $\pi(f)\in \mathrm{End}^{\infty}(\pi)$. We identify $\mathrm{End}^{\infty}(\pi)$ with $\pi^\vee \otimes \pi$ through the isomorphism
$$\begin{array}{cccc}
     & \pi^\vee \otimes \pi & \longrightarrow & \mathrm{End}^\infty (\pi) \\
     & (v^\vee \otimes v) & \longmapsto & (w\mapsto \langle w, v^\vee\rangle v)
\end{array}$$
with inverse 
$$\begin{array}{cccc}
     & \mathrm{End}^\infty (\pi)  & \longrightarrow & \pi^\vee \otimes \pi \\
     & T & \longmapsto & \sum_{i,j} \langle T(e_j), e_i^\vee \rangle e_j^\vee \otimes e_i
\end{array}$$
with $(e_i)_i$ and $(e_j^\vee)_j$ basis of $\pi$ of $\pi^\vee$ respectively. In addition, if $\pi$ is generic, then $\pi \cong \mathcal{W}(\pi, \psi_n)$ and this isomorphism is unique up to a scalar. We choose the isomorphism 
$$\pi^\vee \otimes \pi \cong \mathcal{W}(\pi^\vee, \psi_n^{-1}) \otimes \mathcal{W}(\pi, \psi_n)$$
such that the natural pairing $\langle \cdot, \cdot\rangle $ is sent to the scalar product $(\cdot, \cdot)^{\mathrm{Whitt}}$ on $\mathcal{W}(\pi, \psi_n)$, defined in part 2.8 \cite{RBPPlanch} by 
$$(W, W')^{\mathrm{Whitt}} = \int_{N_n(E)\backslash P_n(E)} W(p)\overline{W'(p)}dp, \ \ W, W' \in \mathcal{W}(\pi, \psi_n),$$
and is convergent when $\pi$ is generic and unitary. This pairing $(\cdot, \cdot)^{\mathrm{Whitt}}$ provides, moreover, a bilinear form on $\overline{\mathcal{W}(\pi, \psi_n)}\times \mathcal{W}(\pi, \psi_n)$ and, as $\pi$ is unitary, identifies $\overline{\mathcal{W}(\pi, \psi_n)}$ to ${\mathcal{W}(\pi^\vee, \psi_n^{-1})}$. Hence, through the isomorphisms
$$(*): \mathrm{End}^\infty(\pi) \cong \pi^\vee \otimes \pi \cong \mathcal{W}(\pi^\vee, \psi_n^{-1}) \otimes \mathcal{W}(\pi, \psi_n)\cong \overline{\mathcal{W}(\pi, \psi_n)}\otimes \mathcal{W}(\pi, \psi_n)$$
the image of $\pi(f)$ in $\overline{\mathcal{W}(\pi, \psi_n)}\otimes \mathcal{W}(\pi, \psi_n)$ is given by

\begin{align*}
    & \sum_{W, W' \in \mathcal{B}(\pi, \psi_n)}(R(f)W, \overline{W'})^{\mathrm{Whitt}}\overline{W}\otimes W', \\
    & = \sum_{W\in \mathcal{B}(\pi, \psi_n)} \overline{W}\otimes \sum_{W'\in \mathcal{B}(\pi, \psi_n)} (R(f)W, \overline{W'})^{\mathrm{Whitt}}W', \\
    & = \sum_{W\in \mathcal{B}(\pi, \psi_{n})}\overline{W}\otimes R(f) W,
\end{align*}
where $\mathcal{B}(\pi, \psi_{n})$ is an orthonormal basis of $\mathcal{W}(\pi, \psi_n)$. Finally, recall that we have defined (see \ref{Wfpi}) 
$$W_{f, \pi}(g_1, g_2) =\int^*_N\mathrm{Trace}\left(\pi(g_1^{-1}ug_2)\pi(f^\vee)\right)\psi_n(u)^{-1}du.$$
Using the formula 
$$\mathrm{Trace}(\pi(g)\pi(f^\vee)) = \sum_{v} \langle \pi(g)\pi(f^\vee)v, v^\vee \rangle,$$
we obtain 
$$\mathrm{Trace}(\pi(g))\pi(f^\vee)) = \sum_{W\in \mathcal{B}(\pi, \psi_n)}(R(g)R(f^\vee)W, W)^{\mathrm{Whitt}}.$$
By using the inversion formula, Proposition 2.143 in \cite{RBPPlanch}, for all $W\in \mathcal{B}(\pi, \psi_n)$ we get
$$ \int_N (R(g_1^{-1}ug_2)R(f^\vee)W, W)^{\mathrm{Whitt}} \psi_n(u)^{-1}du = R(g_2)R(f^\vee)W(1)\overline{R(g_1)W(1)}. $$
Therefore, the function
$$W_{f, \pi} = \sum_{W\in \mathcal{B}(\pi, \psi_n)} \overline{W} \otimes R(f^\vee)W,$$
is the image of $\pi(f)$ in $\overline{\mathcal{W}(\pi, \psi_n)}\otimes \mathcal{W}(\pi, \psi_n)$ by the isomorphisms $(*)$. In other words, the isomorphisms
$$\mathrm{End}^\infty(\pi) \cong \pi^\vee \otimes \pi \cong \mathcal{W}(\pi^\vee, \psi_n^{-1}) \otimes \mathcal{W}(\pi, \psi_n)\cong \overline{\mathcal{W}(\pi, \psi_n)}\otimes \mathcal{W}(\pi, \psi_n).$$
allow us to define a map 
$$\begin{array}{cccc}
     & C_c^\infty(G) & \longrightarrow & \overline{\mathcal{W}(\pi, \psi_n)}\otimes \mathcal{W}(\pi, \psi_n)\\
     & f & \mapsto & W_{f, \pi}
\end{array}$$
and from this map we define a \emph{relative character}.

\begin{definition}[Relative character]
    Let $G=G_n(E)$, $H=G_n(F)$, and $\pi\in \mathrm{Temp}_H(G)$. The linear form $\beta$, defined by 
    $$\beta (W) = \int_{N_n(F)\backslash P_n(F)} W(p)dp, \ \mathrm{for} \ W\in \mathcal{W}(\pi^\vee, \psi_n^{-1}),$$
    is a non trivial element in $ \mathrm{Hom}_H(\mathcal{W}(\pi^\vee, \psi_n^{-1}), \mathbf{1})$ (Section 4 of \cite{Offen}). We then define the \emph{relative character}, $I_\pi$ of $\pi$ by the composition  
    $$\xymatrix{
     C_c^\infty(G)   \ar[rd]_{I_\pi} \ar[r]  & \mathcal{W}(\pi^\vee, \psi_n^{-1}) \otimes \mathcal{W}(\pi, \psi_n)\ar[d]^{\beta \otimes \mathbf{ev}_1} \\
      & \mathbb{C}
     }$$
    Therefore, by identifying again $\overline{\mathcal{W}(\pi, \psi_n)}$ with ${\mathcal{W}(\pi^\vee, \psi_n^{-1})}$ we have 
     $$I_\pi(f) = \sum_{W\in \mathcal{B}(\pi, \psi_n)} \beta(\overline{W}) R(f)W(1) \in \mathbb{C}$$
     where $\mathcal{B}(\pi, \psi_n)$ is a suitable orthonormal basis of $\mathcal{W}(\pi, \psi_n)$.
\end{definition}
Let $\pi\in \mathrm{Temp}_H(G)$, we denote 
$$W_{f, \pi}^1 = W_{f^\vee, \pi}(1, \cdot) \in \mathcal{W}(\pi, \psi_n).$$
By Proposition \ref{ZBeta}, for all $\phi \in C_c^\infty(G)$, we have 
$$\beta(W_{f, \pi}^1)=\hat{\phi}(0)^{-1}Z(1, W_{f, \pi}^1, \phi).$$
Since
$$W_{f, \pi} = \sum_{W\in \mathcal{B}(\pi, \psi_n)} \overline{W} \otimes R(f^\vee)W,$$
the relative character $I_\pi(f)$, for $f\in C_c^\infty(G)$, can be written as
$$I_\pi(f) = \beta(W_{f, \pi}^1).$$
Let $\phi \in C_c^\infty(F^n)$ such that $\hat{\phi}(0) \not = 0$, the relative character can be written in terms of Rankin-Selberg Zeta integrals as follows, 
$$I_\pi(f)=\hat{\phi}^{-1}(0)Z(1, W_{f, \pi}^1, \phi). $$

\subsection{The theorem}
We set
$$\mathcal{I} = \{ \pi \in \mathrm{Temp}_H(G) \mapsto I_\pi(f), \ f\in C_c^\infty(G)\}$$
Let $f\in C_c^\infty(G)$, the map $\pi \mapsto I_\pi(f)$ is a function from $\mathrm{Temp}_H(G)$ to $\mathbb{C}$, and thanks to Theorem 4.2 in \cite{Mat}, we have
$$\mathrm{Temp}_H(G) = BC_n(\mathrm{Temp_{gen}}(U(n))).$$
Therefore, the application $\sigma \mapsto I_{BC_n(\sigma)}(f)$ is now a function on $\mathrm{Temp_{gen}}(U(n))$, and 
$$ \mathcal{I} \cong \{ \sigma \in \mathrm{Temp_{gen}}(U(n)) \mapsto I_{BC_n(\sigma)}(f), \ f\in C_c^\infty(G)\}.$$
Let $\mathfrak{Z}_{\mathrm{gen}}(U(n))$ be the algebra of regular functions on the variety of generic cupsidal supports of $U(n)$, that is $\mathfrak{Z}_{\mathrm{gen}}(U(n)) =\mathbb{C}[\Omega_{\mathrm{gen}}(U(n))]$. Recall that we can attach a cuspidal support to each irreducible representation, thus this \emph{generic Bernstein center} can be seen as a set of functions from $\mathrm{Irr_{gen}}(U(n))$ to $\mathbb{C}$, and we write 
$$\mathfrak{Z}_{\mathrm{gen}}(U(n))\vert_{\mathrm{Temp_{gen}}(U(n))}$$
for the set corresponding to these functions restricted to $\mathrm{Temp_{gen}}(U(n))$. The application of restriction 
$$\mathfrak{Z}_{\mathrm{gen}}(U(n)) \longrightarrow \mathfrak{Z}_{\mathrm{gen}}(U(n))\vert_{\mathrm{Temp_{gen}}(U(n))}$$
is injective, because the image of $\mathrm{Temp_{gen}}(U(n))$ in $\Omega_{\mathrm{gen}}(U(n))$, via the cuspidal support map $\mathbf{Cs}$, is Zariski-dense.
\begin{theo}\label{mainT}
    The base-change map induces an isomorphism
    $$\mathcal{I} \cong \mathfrak{Z}_{\mathrm{gen}}(U(n)).$$
\end{theo}

\begin{proof}
Let $f \in C_c^\infty(G)$. Let $M= G_{m_1}\times \dots \times G_{m_r}\times U(l)$ be a Levi subgroup of $U(n)$, and $\sigma$ be an irreducible, unitary, and supercuspidal representation of $M$. Let $\pi\in \mathrm{Temp_{gen}}(U(n))$, and assume that $\pi$ is a subquotient of $i_P^{U(n)}(\sigma\otimes \lambda)$, where $P=MU$ is a parabolic subgroup of $U(n)$, and $\lambda\in \mathcal{X}(M)$ is real. In other words, the cuspidal support of $\pi$ is $[M, \sigma\otimes \lambda]$. Let $\mathbf{L}=\mathrm{Res}_{E/F}M_E$ be the corresponding Levi in $\mathrm{GL}_n$. From now on, we identify $\mathcal{X}(M)$ as a subgroup of $\mathcal{X}(L)$ via the injection
$$\lambda_1 \boxtimes \cdots \boxtimes \lambda_r \longmapsto \lambda_1 \boxtimes \lambda_1^{-1} \boxtimes \cdots \boxtimes \lambda_r\boxtimes \lambda_r^{-1}.$$
We denote $G=\mathrm{GL}_n(E)$, $H=\mathrm{GL}_n(F)$, and $Q=LU'$ a parabolic subgroup of $G$ where $L=\mathbf{L}(E)$. When $i_P^{U(n)}(\sigma\otimes\lambda)$ is irreducible, we have 

$$BC(i_P^{U(n)}(\sigma\otimes\lambda)) = i_{Q}^G(BC(\sigma)\otimes \lambda),$$
when $i_P^{U(n)}(\sigma\otimes\lambda)$ is reducible, we define 
$$BC(i_P^{U(n)}(\sigma\otimes\lambda)) := i_{Q}^G(BC(\sigma)\otimes \lambda).$$
To establish the inclusion 
$$\mathcal{I} \subseteq \mathfrak{Z}_{\mathrm{gen}}(U(n)),$$
we use the following lemma.

\begin{lem}\label{RestReg}
    Let $\Phi : \mathbb{C}^n \to \mathbb{C}$ be a meromorphic function and let $V, \Delta$ be vector subspaces of $\mathbb{C}^n$, with $\dim \Delta = 1$. Suppose the following conditions.
    \begin{enumerate}
    
        \item The restriction $\Phi\vert_V$ is well defined as a meromorphic function. In other words, $V$ is not included in any of the poles of $\Phi$.
        \item For all $v\in V$, the restriction $\Phi\vert_{v+\Delta}$ is well defined as a meromorphic function and is regular in $v$.
        \item The poles $D_1, \dots, D_m$ of $\Phi$ are all simple and affine hyperplanes, and for all $i\not= j $
        $$D_i \cap V \not = D_j \cap V.$$
        
    \end{enumerate}

    Then $\Phi\vert_V$ is holomorphic on $V$.
\end{lem}

\begin{proof}
    First, remark that, since for all $i\not = j$, 
    $$D_i \cap V \not = D_j \cap V,$$
    the codimension of $D_i\cap V \cap D_j$ in $V$ is either $2$ or $\infty$. Therefore, applying Hartogs' lemma, it suffices to show that for all $v\in D_i\cap V\backslash (\bigcup_{i\not = j}D_i\cap D_j\cap V)$, the function $\Phi\vert _V$ is holomorphic in a neighborhood of $v$. Let $i \in \{1, \dots, m\}$, for simplicity, denote $D=D_i$. Let $v\in D\cap V\backslash (\bigcup_{i\not = j}D_i\cap D_j\cap V)$, there exists an open neighborhood $U_v$ of $v$, such that we can write 
    $$\Phi\vert_{U_v} = \frac{g}{\lambda},$$
    where $\lambda$ denotes an equation of the hyperplane $D$ and $g$ is a holomorphic function in $U_v$. Without loss of generality, we can choose a system of coordinates such that $D$ is the set of elements $z= (z_1, \dots, z_n)\in \mathbb{C}^n$, such that $z_1=0$. Hence, it suffices to show that the function 
    $$z= (z_1, \dots, z_n) \in U_v\cap V \mapsto \Phi(z) = \frac{g(z)}{z_1},$$
    is holomorphic, or equivalently, that $g=0$ on $U_v\cap V\cap D$. By Assumption 1, we have 
    $$V \not \subseteq D = \{z=(z_1, \cdots, z_n)\in \mathbb{C}^n, \ z_1=0\}.$$
    By Assumption 2, for all $z\in V$ the function
    $$\alpha \in \mathbb{C}\mapsto\Phi(z+\alpha l), \ l \in \Delta$$
    is holomorphic at $\alpha=0$. Hence, for $z\in U_v\cap V\cap D$
    $$\alpha \mapsto\Phi(z+\alpha l) = \frac{g(z+\alpha l)}{\alpha l_1}$$
    is also holomorphic in $\alpha = 0$, meaning that $g(z) = 0$ for all $z\in U_v \cap V\cap D$ and that finishes the proof.
    
\end{proof}

We now want to apply this lemma to a meromorphic continuation of the relative characters given by a consequence of Proposition \ref{BetaFam}. We extend the definition of relative characters to $\mathrm{Temp}(G)$. Let $Q=LU$ be a parabolic subgroup of $G$, and let $\rho$ be a representation of $L$, induced from an irreducible and supercuspidal representation. We define, for all $\mu \in \mathcal{X}_{\mathrm{unit}}(L)$,
$$I_{i_Q^G(\rho\otimes \mu)}(f) := \beta^{\rho}_\mu (W_{f,i_Q^G(\rho\otimes \mu) }^1)= \sum_{e\in \mathcal{B}({i_Q^G(\rho\otimes \mu)})} \beta^{\rho}_\mu(R(f)e) J^\rho_\mu(e^\vee),$$
where $J^\rho_\mu$ denotes the the Jacquet functional on $i_Q^G(\rho^\vee\otimes -\mu)$.
\begin{coro}[of Proposition \ref{BetaFam}]
    Using the same notation, the function 
    $$\mu \in \mathcal{X}_{\mathrm{unit}}(L) \mapsto I_{i_Q^G(\rho\otimes \mu)}(f)$$
    admits a meromorphic continuation to $\mathcal{X}(L)$. 
\end{coro}
Indeed, the meromorphic continuation is a direct consequence of Proposition \ref{BetaFam}, the regularity of Jacquet's functionals and the formula
$$I_{i_Q^G(\rho\otimes \mu)}(f) = \beta^{\rho}_\mu (W_{f,i_Q^G(\rho\otimes \mu) }^1).$$
Denote by $\pi_\mu^+$ the unique irreducible and generic subquotient of $i_Q^G(\rho\otimes \mu)$. Thanks to Proposition \ref{Wpi}, for all $\mu \in \mathcal{X}(L)$, we have 
$$W_{f, i_Q^G(\rho\otimes \mu)}^1 = W_{f, \pi_\mu^+}^1 \in \mathcal{W}(\pi_\mu^+, \psi_n).$$
Since two induced representations of the form 
$$i_Q^G(\rho), \ \ \mathrm{and} \ \ i_{Q'}^G(\rho')$$
have the same Jordan-Hölder factors as long as the cuspidal pairs $(L, \rho)$, $(L', \rho')$ are conjugate, the function $W_{f, i_Q^G(\rho\otimes \mu)}^1 = W_{f, \pi_\mu^+}^1$ only depends on the generic subquotient $\pi_\mu^+$, which does not depend on the choice of a parabolic. Therefore, the above definition of $I_{i_Q^G(\rho\otimes \mu)}$ does not depend on the choice of the parabolic subgroup $Q$.

We now show that this rational extension of the relative characters is a regular function when restricted to the image of the base-change map. For $M$ a Levi subgroup of $U(n)$, and $\sigma$ an irreducible, unitary, and supercuspidal representation of $M$, assume $\rho = BC_n(\sigma)$ (is a representation of $L=\mathrm{Res}_{E/F}M_E(F)$), and $\mu \in \mathcal{X}(M)$. We identify 
$$\mathbb{C}^k \cong \mathfrak{a}_{L, \mathbb{C}}^* = X^*(L)\otimes \mathbb{C}.$$
We say that a character $\mu\in \mathfrak{a}^*_{L, \mathbb{C}}$ is \emph{anti-dominant} with respect to $Q=LU$ standard with $L=G_{n_1}\times \cdots \times G_{n_k}$, if $\mu=(\mu_1, \dots, \mu_k)$ with $\Re(\mu_1)\leq \cdots \leq \Re(\mu_k)$. Let 
$$\mu \in \mathfrak{a}_{L, \mathbb{C}}^*\mapsto e_\mu\in i_{Q}^G(\rho \otimes \mu)$$ 
be regular section and $\phi\in C_c^\infty(F^n)$. We apply Lemma \ref{RestReg} to 
$$\Phi :\mu \in \mathfrak{a}_{L, \mathbb{C}}^* \mapsto \beta_\mu^{\rho}(e_\mu \otimes \phi),$$ 
with $V=\mathfrak{a}_{M, \mathbb{C}}^*\subset \mathfrak{a}_{L, \mathbb{C}}^*$ and $\Delta = \{\vert \det \vert_E^s, \  s\in \mathbb{C}\}$, in order to show that $\Phi\vert_V$ is regular on 
$$V_Q^-:= \{\lambda \in V \  \mathrm{anti-dominant}\}.$$

By definition, 
$$\mu \in \mathcal{X}(L) \mapsto \beta_\mu^{BC_n(\sigma)}(e_\mu \otimes \phi)$$
is regular in a neighborhood of the unitary characters, in particular it is regular in a neighborhood of $0 \in \mathcal{X}(M)$, meaning that none of the poles contain $0$, therefore they do not contain $\mathcal{X}(M)$. 
In other words,
$$\Phi :\mu \in \mathfrak{a}_{L, \mathbb{C}}^* \mapsto \beta_\mu^{\rho}(e_\mu \otimes \phi),$$
is a meromorphic function verifying condition 1 of Lemma \ref{RestReg}. 

The second part of Proposition \ref{BetaFam} ensures that
$$\mu \in \mathcal{X}(L) \mapsto \frac{\beta^\rho_\mu(e_\mu \otimes \phi)}{L(1, i_Q^G(\rho\otimes \mu), \mathrm{As}^+)}$$
is a regular function on $\mathcal{X}(L)$. In particular, in the proof we show that for all $\mu \in \mathcal{X}(L)$
$$s \in \Delta \mapsto \frac{\beta^\rho_{\mu_s}(e_{\mu_s} \otimes \phi)}{L(1, i_Q^G(\rho\otimes \mu_s), \mathrm{As}^+)},$$
where $\mu_s= \mu \otimes \vert \det \vert_E^s$, is regular, and therefore this also holds for all $\mu \in \mathcal{X}(M)$. Recall that we have fixed the isomorphism between $i_Q^G(\rho\otimes \mu)$ and its Whittaker model via the Jacquet functional. Since $W_{\mu_s} \in \mathcal{W}(i_Q^G(\rho\otimes \mu_s), \psi_n)$ can be written as $\vert \det \vert_E^sW_{\mu, s}$ with $W_{\mu, s}\in \mathcal{W}(i_Q^G(\rho\otimes \mu), \psi_n)$, we showed that the regularity of the above quotient is given by the regularity of 
$$s \in \mathbb{C} \mapsto\frac{Z(1+2s, W_{\mu,s}, \phi)}{L(1+2s, i_Q^G(\rho\otimes \mu), \mathrm{As}^+)}.$$
If $\mu$ is anti-dominant, then the unique generic irreducible subquotient $\pi_\mu^+$ of $i_Q^G(\rho\otimes \mu)$ is in fact a quotient (see Proposition 2.10 in \cite{BZ2}, or \cite{JS}). In that case we have $W_{\mu,s} \in \mathcal{W}(\pi^+_\mu, \psi_n)$ and $Z(1+2s, W_{\mu,s}, \phi)$ is regular at $s=0$.

Combining Proposition 7.4 in \cite{GGPlocalroot} and Proposition \ref{LAd}, we obtain that 
$$L(1, i_Q^G(\rho\otimes \mu), \mathrm{As}^+)\not = \infty,$$
as long as $i_Q^G(\rho\otimes \mu)$ is irreducible. Recall that there exists a parabolic subgroup $Q'=L'U'$, a tempered irreducible representation $\rho'$ of $L'$ and a dominant character $\mu'\in \mathcal{X}(L')$ such that $\pi_\mu^+$ is isomorphic to its  standard module $i_{Q'}^G(\rho'\otimes \mu')$. Even if $\pi_\mu^+$ might not be in the image of the base-change map, it is a conjugate-dual representation. Indeed, $(\pi_\mu^+)^*$ is the unique generic irreducible subquotient of 
$$i_Q^G(\rho\otimes \mu)^*\cong i_Q^G((\rho\otimes \mu)^*) \cong i_{\overline{Q}}^G(\rho\otimes \mu).$$
Therefore, 
$$\pi_\mu^+\cong(\pi_\mu^+)^*$$
and we can use again Proposition 7.4 in \cite{GGPlocalroot} and Proposition \ref{LAd}, to get
$$L(1, \pi_\mu^+, \mathrm{As}^+) = L(1, i_{Q'}^G(\rho'\otimes \mu'), \mathrm{As}^+) \not = \infty.$$
This shows that condition 2 of Lemma \ref{RestReg} is satisfied, with 
$$\Delta_\mu = \{\mu \otimes \vert \det \vert^s, s \in \mathbb{C}\}, \ \forall \mu \in V_{Q}^-.$$

It remains to show that the third condition is satisfied. Set
$$M = G_{m_1}\times \dots \times G_{m_r}\times U(l), \ \ \mathrm{and} \ \ \sigma = \sigma_1 \boxtimes \cdots \boxtimes \sigma_r \boxtimes \sigma^u$$
with $\sigma_i, \sigma^u$ irreducible, supercuspidal and unitary. Then 
$$\rho = \sigma_1 \boxtimes \sigma_1^* \boxtimes \cdots \boxtimes \sigma_r \boxtimes \sigma_r^* \boxtimes BC_n(\sigma^u)$$
and we show that the poles of 
$$\mu \in \mathcal{X}(L) \mapsto L(1, i_Q^G(\rho\otimes \mu), \mathrm{As}^+)$$
are simple, defined by hyperplanes that are different when intersected with $\mathcal{X}(M)$. Recall that the Asai $L$-function decomposes as follows.
\begin{align*}
    L(1, i_Q^G(\rho\otimes \mu), \mathrm{As}^+) = & \prod_{i=1}^rL(1+2\mu_i, \sigma_i, \mathrm{As}^+)L(1-2\mu_i, \sigma_i^*, \mathrm{As}^+) \\
    & \times L(1, BC_n(\sigma^u), \mathrm{As}^+) \times \prod_{i=1}^rL(1, \sigma_i\boxtimes\sigma_i^\vee) \\
    & \times \prod_{i<j} [ L(1+\mu_i+\mu_j, \sigma_i\boxtimes\sigma_j^c)L(1-\mu_i+\mu_j, \sigma_i^*\boxtimes\sigma_j^c) \\
    &  \ \ \ \ \ \ \ \ \    L(1+\mu_i-\mu_j, \sigma_i\boxtimes\sigma_j^\vee)L(1-\mu_i-\mu_j, \sigma_i^*\boxtimes\sigma_j^\vee) ] \\
    & \times \prod_{i=1}^r L(1+\mu_i, \sigma_i \boxtimes BC_n(\sigma^u)^c)L(1-\mu_i, \sigma_i^* \boxtimes BC_n(\sigma^u)^\vee)
\end{align*}
Note that we have 
$$\varphi =\left( \bigoplus_{i=1}^r \varphi_i \right) \oplus \varphi_{\sigma^u},$$
with $\varphi_1, \dots, \varphi_r, \varphi_{\sigma^u}$ trivial on $\mathrm{SL}_2(\mathbb{C})$. Thanks to Proposition \ref{propcuspBC}, we also have 
$$\varphi_{BC_n(\sigma^u)}= \bigoplus_{i=1}^{r'}\varphi'_i$$
conjugate-orthogonal, where $\varphi'_i$ is trivial on $\mathrm{SL}_2(\mathbb{C})$ and for all $i\not = j$, $\varphi'_i\not = \varphi_j'$. We first remark that we have 
$$L(1, BC_n(\sigma^u), \mathrm{As}^+)\not = \infty$$ because the representation $BC_n(\sigma^u)$ is tempered. For all $i\in \{1, \dots, r\}$, we have 
$$L(1, \sigma_i\boxtimes \sigma_i^\vee)=\infty$$
if and only if $\sigma_i \cong \vert \cdot \vert\sigma_i$ which never happens. The following list consists of the possible equations for the poles of the function 
$$\mu \in \mathcal{X}(L) \mapsto L(1, i_Q^G(\rho\otimes \mu), \mathrm{As}^+).$$
\begin{enumerate}
    \item For all $i\in \{1, \dots, r\}$ the representation $\sigma_i$ is supercuspidal and unitary, hence $L(1+2\mu_i, \sigma_i, \mathrm{As}^+)$ has a pole if $1+2\mu_i \in i\mathbb{R}$, denote 
    $$D_1=\{\Re(\mu_i)=-1/2\}.$$
    \item Similarly, $L(1-2\mu_i, \sigma_i, \mathrm{As}^+)$ has a pole if $1-2\mu_i \in i\mathbb{R}$, denote 
    $$D_2= \{\Re(\mu_i)=1/2\}.$$
    \item For all $i<j$, $ L(1+\mu_i+\mu_j, \sigma_i\boxtimes\sigma_j^c)$ can have a pole on $1+\mu_i+\mu_j \in i \mathbb{R}$,
    $$D_3 = \{\Re(\mu_i) + \Re(\mu_j)=-1\}$$
    \item For all $i<j$, $ L(1-\mu_i+\mu_j, \sigma_i^*\boxtimes\sigma_j^c)$ can have a pole on $1-\mu_i+\mu_j \in i \mathbb{R}$,
    $$D_4 =\{\Re(\mu_i) - \Re(\mu_j)=1\} $$
    \item For all $i<j$, $ L(1+\mu_i-\mu_j, \sigma_i\boxtimes\sigma_j^\vee)$ can have a pole on $1+\mu_i-\mu_j \in i \mathbb{R}$,
    $$D_5 =\{\Re(\mu_i) - \Re(\mu_j)=-1\} $$
    \item For all $i<j$, $ L(1-\mu_i-\mu_j, \sigma_i^*\boxtimes\sigma_j^\vee)$ can have a pole on $1-\mu_i-\mu_j \in i \mathbb{R}$,
    $$D_6 =\{\Re(\mu_i) + \Re(\mu_j)=1\} $$
    \item Thanks to the decomposition of $\varphi_{BC_n(\sigma^u)}$ given above, $L(1+\mu_i, \sigma_i \boxtimes BC_n(\sigma^u)^c)$ has a pole for $1+\mu_i \in i\mathbb{R}$, denote 
    $$D_7= \{\Re(\mu_i)=-1\}.$$
    \item Similarly, $L(1-\mu_i, \sigma_i^* \boxtimes BC_n(\sigma^u)^\vee)$ has a pole for $1-\mu_i \in i\mathbb{R}$, denote 
    $$D_8= \{\Re(\mu_i)=1\}.$$
    
\end{enumerate}
Clearly, for all $i\not = j$, we have $D_i\not = D_j$, which means that the poles of $L(1, i_Q^G(\rho\otimes \mu), \mathrm{As}^+)$ for $\rho \otimes \mu = BC_n(\sigma_i)\otimes \mu$ with $\mu \in \mathcal{X}(M)$ are at most simple. It allows us to use Lemma \ref{RestReg}, and therefore the function 
$$\Phi :\mu \in \mathfrak{a}_{M, \mathbb{C}}^* \mapsto \beta_\mu^{\rho}(e_\mu \otimes \phi),$$
is regular for all $\mu$ anti-dominant. As a direct consequence, we obtain the regularity of 
$$\mu \in \mathcal{X}(M) \mapsto I_{i_Q^G(\rho\otimes \mu)}(f) = \beta^{\rho}_\mu (W_{f,\pi_\mu^+ }^1)$$
for all $\mu \in \mathcal{X}(M)$, since the definition of $I_{i_Q^G(\rho\otimes \mu)}(f)$ does not depend on the choice of $Q$.


Let $\lambda \in \mathcal{X}(M)$ be anti-dominant. Assume there exists a representation $\gamma\in \mathrm{Temp_{gen}}(U(n))$, such that $BC_n(\gamma) = \pi_\lambda^+$, and 

$$\begin{array}{cccc}
     \beta_{\pi_\lambda^+} : & \pi_\lambda^+ & \to & \mathbb{C}  \\
     & v & \longmapsto & \beta(W_v)
\end{array}$$
and 
$$J_{\pi_\lambda^+} :  (\pi_\lambda^+)^\vee  \to  \mathbb{C}, $$
such that the relative character of $\pi_\lambda^+$ can be written as
$$I_{\pi_\lambda^+}(f) = \beta_{\pi_\lambda^+} (W_{f, \pi_\lambda^+}^1) = \sum_{v\in \mathcal{B}(\pi_\lambda^+)}\beta_{\pi_\lambda^+}(R(f)v).J_{\pi_\lambda^+}(v^\vee),$$
where $\mathcal{B}(\pi_\lambda^+)$ is an orthonormal basis of $\pi_\lambda^+$. Set $\beta_\lambda=\beta_\lambda^{BC_n(\sigma)} \in \mathrm{Hom}_H(i_Q^G(BC_n(\sigma) \otimes \lambda, \mathbb{C})$, we have the following diagrams. 

$$\xymatrix{
     i_Q^G(BC_n(\sigma)\otimes \lambda)  \ar[rd]_{\beta_\lambda} \ar[r]^{p}  & {\pi_\lambda^+}\ar[d]^{\beta_{\pi_\lambda^+}} \\
      & \mathbb{C}
     }$$
where $p$ is the natural projection. And

$$\xymatrix{
      ({\pi_\lambda^+})^\vee \ar[rd]_{J_{\pi_\lambda^+}} \ar[r]^{\iota}  &  i_Q^G(BC_n(\sigma)^\vee\otimes -\lambda) \ar[d]^{J_\lambda} \\
      & \mathbb{C}
     }$$
where $\iota$ is the natural injection, and by uniqueness of Whittaker functionals, $J_{\pi_\lambda^+}$ and $J_\lambda$ are proportional. We recall Proposition \ref{Wpi}, namely
$$W_{f, i_Q^G(BC_n(\sigma)\otimes \lambda)}^1 = W_{f, {\pi_\lambda^+}}^1 \in \mathcal{W}({\pi_\lambda^+}, \psi_n),$$
for all $f\in C_c^\infty(G)$. Moreover, thanks to Corollary \ref{CoroGKfam}, we can find a function $f\in C_c^\infty(G)$ such that 
$$\beta_\lambda(W_{f, i_Q^G(BC_n(\sigma)\otimes \lambda)}^1) = \beta_\lambda (W_{f, {\pi_\lambda^+}}^1) \not = 0,$$
Thus,
$$\mathrm{Hom}_H(i_Q^G(BC_n(\sigma)\otimes \lambda), \mathbb{C})\not = 0.$$

\begin{coro}[of Proposition \ref{BetaProportional}]
    Let $\gamma\in \mathrm{Temp_{gen}}(U(n))$ be a subquotient of $i_P^{U(n)}(\sigma\otimes \lambda)$, where $P=MU$ is a parabolic subgroup of $U(n)$, $\sigma\in \Irr{M}$ supercuspidal and unitary and $\lambda\in \mathcal{X}(M)$ is real. Then, for all function $f\in C_c^\infty(G)$, we have
    $$ I_{BC_n(\gamma)}(f) = I_{BC_n(i_P^{U(n)}(\sigma\otimes \lambda))}(f).$$
\end{coro}
Recall that we have defined 
$$BC_n(i_P^{U(n)}(\sigma\otimes\lambda)) = i_{Q}^G(BC_n(\sigma)\otimes \lambda).$$
We have shown in part \ref{cuspBC}, and more precisely in Proposition \ref{propcuspBC}, that the base-change of an irreducible supercuspidal and generic representation of a unitary group is an induced representation of supercuspidal representations in $G$. Hence, thanks to Proposition \ref{BetaProportional}, when $\lambda$ is anti-dominant with respect to $Q$, the space $\mathrm{Hom}_H(i_Q^G(BC_n(\sigma)\otimes \lambda), \mathbb{C})$ is at most one dimensional, implying that $\beta_\lambda$ and $\beta_{\pi_\lambda^+} \circ p$ are proportional, where $\pi_\lambda^+$ denotes $BC_n(\gamma)$. Therefore, there exists a constant $c\in \mathbb{C}$, such that the following equality holds for all $f \in C_c^\infty(G)$
$$\beta_{\pi_\lambda^+} =c.\beta_\lambda.$$
Thanks to Corollary \ref{CoroGKfamKirrilov}, we can chose a family $v_\mu\otimes \phi \in i_Q^G(BC_n(\sigma_X))\otimes C_c^\infty(F^n)$, with $\hat{\phi}(0)\not =0$ and the image $W_\mu\in \mathcal{W}(i_Q^G(BC_n(\sigma_X)), \psi_n)$ of the family $v_\mu$ is such that 
$$W_\mu \in C_c^\infty(N_n(E)\backslash P_n(E), X, \psi_n),$$
and
$$\int_{N_n(F)\backslash P_n(F)}W_\mu(p)dp \not = 0$$
converges for all $\mu \in \mathcal{X}(L)$. Since for all $\mu \in \mathcal{X}_{\mathrm{unit}}(L)$, we have 
$$\beta_\mu(v_\mu\otimes \phi) = \int_{N_n(F)\backslash P_n(F)}W_\mu(p)dp,$$
this formula extends to $\mathcal{X}(L)$ by analytic continuation. In particular, when $\mu = \lambda$, 
$$\beta_\lambda(v_\lambda\otimes \phi) = \int_{N_n(F)\backslash P_n(F)}W_\lambda(p)dp = \beta_{\pi_\lambda^+} \circ p (v_\lambda \otimes \phi).$$
Therefore $c = 1$, that is 
$$I_{{\pi_\lambda^+}}(f)= I_{BC_n(i_P^{U(n)}(\sigma\otimes \lambda))}(f).$$
We then have shown that
$$\{ \pi \in BC_n(\mathrm{Temp_{gen}}(U(n))) \mapsto I_\pi(f), \ f\in C_c^\infty(G)\} \subseteq \mathfrak{Z}_{\mathrm{gen}}(U(n))\vert_{\mathrm{Temp_{gen}}(U(n))}.$$
Finally, the action of $\mathfrak{Z}(G)$ on $C_c^\infty(G)$ gives rise to an action of $\mathfrak{Z}(G)$ on 
$$\{ \pi \in \mathrm{Temp}_H(G) \mapsto I_\pi(f), \ f\in C_c^\infty(G)\}$$
by $I_\pi(z.f)=z_\pi I_\pi(f)$, $z\in \mathfrak{Z}(G)$ where $z_\pi$ is a scalar. By the proposition \ref{closedim}, the surjection from $\mathfrak{Z}(G)$ to $\mathfrak{Z}_{\mathrm{gen}}(U(n))$ provides an action of $\mathfrak{Z}_{\mathrm{gen}}(U(n))$ on 
$$\{ \pi \in BC_n(\mathrm{Temp_{gen}}(U(n)) \mapsto I_\pi(f), \ f\in C_c^\infty(G)\},$$
which makes it an ideal of $\mathfrak{Z}_{\mathrm{gen}}(U(n))$. The Corollary \ref{CoroGKfam} shows that 
$$\{ \pi \in BC_n(\mathrm{Temp_{gen}}(U(n)) \mapsto I_\pi(f), \ f\in C_c^\infty(G)\}$$
is included in none of the maximal ideals of $\mathfrak{Z}_{\mathrm{gen}}(U(n))$ and hence 
$$\{ \pi \in BC_n(\mathrm{Temp_{gen}}(U(n)) \mapsto I_\pi(f), \ f\in C_c^\infty(G)\}=\mathfrak{Z}_{\mathrm{gen}}(U(n))$$
\end{proof}

\section{Towards relative trace formula: application to a relative comparison}

In this section, let $E/F$ be a quadratic extension of $p$-adic fields, $G= \mathrm{GL}_n(E)$, $H=\mathrm{GL}_n(F)$ and $S=H\backslash G$. Let $c$ be the non-trivial element in $\mathrm{Gal}(E/F)$. Let $B=TN$ be the Borel subgroup of upper triangular matrices in $G$. We fix $(N, \psi)$ a Whittaker datum for $G$, such that 
$$\psi\begin{pmatrix}
        1 & u_1 &  * & \dots  & * \\
          & \ddots & \ddots& \ddots  & \vdots \\
          & & \ddots & \ddots & * \\
          &  &  &\ddots &   u_{n-1} \\
          & & & & 1
    \end{pmatrix} = \psi_0(\sum_{i=1}^{n-1}u_i),$$
where $\psi_0:E/F \to \mathbb{C}^\times$ is non-trivial, in particular we have 
$$\psi\vert_{N\cap H}\equiv 1, $$
implying $\psi^c = \psi^{-1}$. Let 
$$w_n = \begin{pmatrix} & & 1 \\
                        & \revdots & \\
                        1 & & \end{pmatrix}$$
be the longest element in both Weyl groups $W_G$ and $W_{U(n)}$, and recall that 
$$U(n)=\{g\in G \mid {}^tg^c w_n g = w_n\}.$$
Now choose $(N', \psi')$ a Whittaker datum for $U(n)$, with $N'= N\cap U(n)$ and $\psi'=\psi\vert_{N'}$, and denote by $B' = B\cap U(n)$.
\begin{rem}
    Let $x=\begin{pmatrix}
        1 & u_1 &  * & \dots  & * \\
          & \ddots & \ddots& \ddots  & \vdots \\
          & & \ddots & \ddots & * \\
          &  &  &\ddots &   u_{n-1} \\
          & & & & 1
    \end{pmatrix} \in N'$, for all $i<n-1$ we have 
    $$u_i = -u_{n-i}^c.$$
    Now, for all $i<n/2$ denote by $N'_i$ the simple root groups in $N'$.
    Recall that $\psi'$ is generic, if and only if for all $i<n/2$
    $$\psi'\vert_{N'_i}\not \equiv 1.$$
    We have 
    $$\psi'\vert_{N'_i} : u_i \mapsto \psi_0(u_i-u_i^c)$$
    with 
    $$\psi_0\vert_F =\psi_0\vert_{E^{c=1}} \equiv 1.$$
    Since $ E =F\oplus E^{c=-1}$, we have $E^{c=-1} = \{x-x^c\mid x\in E\}$.
    We obtain 
    $$\psi'\vert_{N'_i}\not \equiv 1,$$
    for all $i<n/2$, because $\psi_0$ is non-trivial. Therefore, $\psi'$ is generic.
\end{rem}

\subsection{Spectral transfer}\label{spectrtransf}
Let $(N, \psi)$ and $(N', \psi')$ be the Whittaker data for $G$ and $U(n)$, respectively, that we chose above. We define the following Bessel distribution for $U(n)$, for all $\pi \in \mathrm{Temp_{gen}}(U(n))$, and all $f'\in C_c^\infty(U(n))$
$$J_\pi(f') = \int_{N'}^*f'_\pi(u)\psi'(u)du,$$
with $f'_\pi(g) = \mathrm{Tr}(\pi(g^{-1})\pi(f'))$ (we refer to Definition \ref{RegInt} for the regularized integral), and \cite{RBP1} provides a similar trace Paley-Wiener theorem (Theorem 2.4.1), namely 
$$\{\pi \in \mathrm{Temp_{gen}}(U(n)) \mapsto J_\pi(f'), \ f'\in C_c^\infty(U(n))\}\cong \mathfrak{Z}_{\mathrm{gen}}(U(n)).$$
\begin{definition}[Spectral transfer]
    Two functions $f\in C_c^\infty(G)$, and $f'\in C_c^\infty(U(n))$ are \emph{transfers} from one to the other if for all $\pi\in \mathrm{Temp_{gen}}(U(n))$
    $$I_{BC_n(\pi)}(f) = J_\pi(f').$$
\end{definition}

Let $x\in G$, $y\in U(n)$. Let $f\in C_c^\infty(G)$, and $f'\in C_c^\infty(U(n))$, the relative orbital integrals 
$$O_\psi(x,f) = \int_{(H\times N)_x\backslash H\times N} f(h^{-1}xu)\psi(u)dhdu,$$
and 
$$O_{\psi'}(y,f')=\int_{(N'\times N')_y\backslash N'\times N'} f'(n_1^{-1}yn_2)\psi'(n_1n_2^{-1})dn_1dn_2,$$
are well defined if and only if $x$ and $y$ satisfy the following property
\begin{itemize}
    \item For all $(h,u)\in (H\times N)_x$, $\psi(u)=1$, 
    \item For all $(n_1, n_2)\in (N'\times N')_y$, $\psi'(n_1n_2^{-1})=1$,
\end{itemize}
where 
$$(H\times N)_x = \{(h,u)\in H\times N, \ h^{-1}xu= x\}$$
and 
$$(N'\times N')_y = \{(n_1,n_2) \in N'\times N', \ n_1^{-1}yn_2=y\}$$
are the stabilizers of $x$ and $y$ for the $H\times N$ and $N'\times N'$ actions by left and right translations, respectively. Let $f\in C_c^\infty(G)$, remark that 
$$O_\psi(x,f) = \int_{(N\cap x^{-1}Hx)\backslash N}\int_{H} f(h^{-1}xu)dh.\psi(u)du.$$
Thus, for all $x\in G$, the orbital integral $O_\psi(x, \cdot)$ factorizes through the surjection 
$$\begin{array}{cccc}
     & C_c^\infty(G) & \longrightarrow & C_c^\infty(H\backslash G) \\
     & f & \longmapsto & \int_Hf(h\cdot)dh
\end{array}.$$
Let $S=H\backslash G$, and consider
$$\begin{array}{cccc}
     & S & \hookrightarrow & G \\
     & Hg & \mapsto & (g^c)^{-1}g 
\end{array}$$
whose image is the set of elements $s\in G $, such that $ss^c = 1$, thanks to Hilbert's Theorem 90. Under this embedding, the right translation action of $G$ on $S$ becomes 
$$s \cdot g = (g^c)^{-1}sg, \ \ s\in S, \ g \in G.$$
Under this action, the stabilizer of $s$ in $N$ is
$$N_s=\{u\in N, \ u^c=sus^{-1}\}.$$
We define the sets of \emph{relevant} elements in $G$, $U(n)$ and $S=H\backslash G$, respectively, by 
$$G_{\mathrm{rel}} := \{g\in G, \ (\psi \boxtimes \psi^{-1})\vert_{(N\times N)_g}\equiv 1\},$$
$$U(n)_{\mathrm{rel}}= \{x\in U(n), \ (\psi' \boxtimes (\psi')^{-1})\vert_{(N'\times N')_x}\equiv 1\},$$
and 
$$S_{\mathrm{rel}} := \{s\in S, \psi\vert_{N_s}\equiv 1\}.$$

\begin{ex}
    Let $w=w_n$ be the longest element in $W_{U(n)}$. For any $n\in N'$, the element $w_nnw_n^{-1}$ belongs to the unipotent radical of the opposite Borel subgroup, $\overline{N'}$, in particular for all $t\in T'$
    $$(N'\times N')_{tw_n} = \{(1,1)\}.$$
    Therefore, $N't\overset{.}{w_n} N'$ is always relevant, in other words
    $$T'w_n \subset N' \backslash U(n)_{\mathrm{rel}}/N'.$$
    Now let $w= 1$, for all $t\in T'$ and all $n\in N'$ we have $tnt^{-1}\in N$, i.e.
    $$(N'\times N')_t = \{(tnt^{-1}, n) \vert \ n\in N'\}.$$
    The corresponding double coset is relevant if and only if 
    $$\psi'(tnt^{-1}n^{-1}) = 1, \forall n\in N',$$
    if and only if $t\in Z_{U(n)}$. Since $\psi'$ is generic, by definition, $N'tN'$ with $t\in Z_{U(n)}$ is always relevant. In other words
    $$Z_{U(n)} \subset N' \backslash U(n)_{\mathrm{rel}}/N'.$$
\end{ex}

\begin{rem}\label{Relelements}
Consider $G$ and the Borel subgroup $B=TN$ of upper triangular matrices. The Bruhat decomposition of $G$ can be written as 
$$G = NN_G(T)N, \ \mathrm{with} \ N_G(T) = T\rtimes W_G.$$
Define 
$$N_G(T)_{\mathrm{rel}}:= N_G(T) \cap G_{rel}.$$
Let $g \in N_G(T) = T\rtimes W_G$. The stabilizer of $g$ is given by 
$$(N\times N)_g = \{ (n_1,n_2) \in N\times N \mid  n_1^{-1}gn_2=g \}.$$
We now follow the description of relevant elements, given by Jacquet-Ye \cite{JacquetYe}. Let $\theta_0:N\to E$ be an algebraic additive character defined by 

$$\theta_0\begin{pmatrix}
        1 & u_1 &  * & \dots  & * \\
          & \ddots & \ddots& \ddots  & \vdots \\
          & & \ddots & \ddots & * \\
          &  &  &\ddots &   u_{n-1} \\
          & & & & 1
    \end{pmatrix} = \sum_{i=1}^{n-1}u_i,$$
and recall that $\psi = \psi_0\circ \theta_0$. An element $g'$ is \emph{Jacquet-Ye relevant} if the character $(n_1,n_2)\mapsto \theta_0 (n_1n_2)$ is trivial on the stabilizer
$$N^{g'}= \{(n_1,n_2) \in N\times N \mid {}^tn_1g'n_2=g'\}.$$
Since $\theta_0$ is an algebraic additive character, and $N^{g'}$ is connected, we have 
$$\theta_0(N^{g'}) = \{0\}, \ \mathrm{or} \ E.$$
Therefore, $\theta_0$ is trivial on $N^{g'}$, if and only if $\psi = \psi_0 \circ \theta_0$ is trivial on $N^{g'}$. Let $w_n = w_l^G$ be the longest element in $G$. For all $g\in G$, we have 
$$N^{w_ng} = (N\times N)_g,$$
because $(n_1,n_2)\in N^{w_ng}$, if and only if, $(w_n{}^tn_1^{-1}w_n, n_2)\in (N\times N)_g$. This shows that an element $g\in N_G(T)$ is relevant if, and only if, $w_ng$ is Jacquet-Ye relevant. Therefore, there exists a standard Levi subgroup $M$ of $G$, such that 
$$w_ng= a_Mw_l^M,$$
with $a_M$ an element of the center $Z_M$ of $M$, and $w_l^M$ is the longest element in $W_M=W_G\cap M$. Finally, an element $g\in N_G(T)_{\mathrm{rel}}$ is of the form 
$$g= a_Mw_l^Mw_n.$$

\end{rem}

\subsection{Parametrization of the orbits}\label{orbits}

We use the same notation as in the previous section. The goal of this section is to construct a natural bijection  
$$N'\backslash U(n)_{\mathrm{rel}}/N' \cong S_{\mathrm{rel}}/N.$$
First of all, to construct this bijection, we need these intermediate results on unipotent groups.

\begin{lem}\label{n-root}
    Let $V$ be a unipotent group (defined over a field of characteristic $0$), and let $n\geq 1$ be an integer. Any element $v\in V$ admits a unique $n$-th root.
\end{lem}

\begin{proof}
    Let us show that the following morphism is bijective
    $$\begin{array}{cccc}
         & V & \longrightarrow & V \\
         & v & \longmapsto & v^n
    \end{array}.$$
    Since the exponential map provides an isomorphism, of algebraic varieties, between $V$ and its Lie algebra $\mathfrak{V}$ (Proposition 14.32, in \cite{Milne_2017}), and 
    $$\mathrm{exp}(X)^n = \mathrm{exp}(nX),$$
    for all $X\in \mathfrak{V}$, we obtain the following commuting diagram.
    $$\xymatrix{
     V \ar[r]^{(\cdot)^n} & V    \\
     \mathfrak{V} \ar[u]^{\mathrm{exp}} \ar[r]_{n\cdot} &  \mathfrak{V}\ar[u]_{\mathrm{exp}}
     }$$
    Since $n$ is invertible, the map $\mathfrak{V}\to \mathfrak{V}$ given by the multiplication by $n$ is bijective. Therefore, 
    $$\begin{array}{cccc}
         & V & \longrightarrow & V \\
         & v & \longmapsto & v^n
    \end{array}$$
    is a bijection, i.e. for all $v\in V$ there exists a unique element $x\in V$ such that
    $$v= x^n.$$
\end{proof}

\begin{lem}\label{unipdecomp}
    Let $V$ be a unipotent group (defined over a field of characteristic $0$), and $\iota\in \mathrm{Aut}(V)$ be an involution. Denote by 
    $$V^{\iota =1} =\{v\in V \mid  \iota(v)=v\},$$
    and 
    $$V^{\iota = -1}=\{v\in V \mid  \iota(v)=v^{-1}\}.$$
    Then 
    $$V=V^{\iota = 1}V^{\iota = -1}.$$
\end{lem}

\begin{proof}
    Let $v\in V$, remark that $\iota(v)^{-1}v\in V^{\iota=-1}$, because
    $$\iota(\iota(v)^{-1}v)= v^{-1}\iota(v) = (\iota(v)^{-1}v)^{-1}.$$
    Thanks to Lemma \ref{n-root}, there exists $b\in V$, such that 
    $$b^2 = \iota(v)^{-1}v,$$
    and we have 
    $$\iota(b)^2 = (\iota(v)^{-1}v)^{-1} = b^{-2},$$
    by uniqueness of the square-root, we have $b\in V^{\iota=-1}$, and $v = \iota(v)b^2$. 
    Now, set $a = vb^{-1}$, we have 
    $$a = \iota(v) b$$
    and therefore 
    $$\iota(a) = vb^{-1} = a.$$
    For an arbitrary $v\in V$, we have obtained a factorization 
    $$v = ab, \ \mathrm{with} \ a\in V^{\iota=1}, \ b\in V^{\iota = -1}, $$
    which finishes the proof.
\end{proof}

The following Lemma is reformulation in our setting of Proposition 2.1 in \cite{Benoist1984}.
\begin{lem}\label{antiinv}
    Let $V$ be a unipotent group (defined over a field of characteristic $0$) and $\iota\in \mathrm{Aut}(V)$ be an involution, then 
    $$V^{\iota = -1} = \{\iota(v)^{-1}v, \ v\in V\}.$$
\end{lem}

\begin{rem}
    Proposition 2.1 of \cite{Benoist1984} actually shows that Lemma \ref{n-root} is also true for $V^{\iota=-1}$, although it is not a unipotent group.
\end{rem}

\begin{lem}\label{relUn}
    We have $U(n)_{\mathrm{rel}} = U(n)\cap G_{\mathrm{rel}}$.
\end{lem}

\begin{proof}
    The inclusion $ U(n)\cap G_{\mathrm{rel}} \subset U(n)_{\mathrm{rel}} $ is straightforward, using the definitions $N'=N\cap U(n)$ and $\psi'= \psi\vert_{N'}$. 

    Let $x\in  U(n)_{\mathrm{rel}}$, and let 
    $$\begin{array}{cccc}
        \theta' :  & (N\times N)_x & \to & (N\times N)_x \\
         & (n_1, n_2) & \mapsto & (n_1^*, n_2^*)
    \end{array}$$
    be an involution, where $g^* = w_n{}^t(g^{-1})^cw_n$. Let $(n_1, n_2)\in (N\times N)_x$, the involution $\theta'$ is well defined, because 
    \begin{align*}
        (n_1^*)^{-1}xn_2^* & =  (w_n{}^tn_1^cw_n)x(w_n{}^t(n_2^{-1})^cw_n), 
    \end{align*}
    and $x\in U(n)$ implies that $w_nxw_n = ({}^tx^{-1})^c$, therefore
    \begin{align*}
        (n_1^*)^{-1}xn_2^* & =  w_n({}^tn_1{}^tx^{-1}{}^tn_2^{-1})^cw_n, \\
        & = w_n({}^t(n_1^{-1}xtn_2)^{-1})^cw_n, \\
        & = w_n {}^t(x^{-1})^cw_n, \\
        & = x^*, \\
        & = x.
    \end{align*}
    Applying Lemma \ref{unipdecomp} to $V= (N\times N)_x$ and $\iota = \theta'$, we obtain 
    $$(N\times N)_x = (N\times N)_x^{\theta'=1}(N\times N)_x^{\theta'=-1},$$
    where 
    $$(N\times N)_x^{\theta'=1} = \{(u_1, u_2)\in (N\times N)_x, \ \theta' (u_1, u_2) = (u_1, u_2) \},$$
    and 
    $$(N\times N)_x^{\theta'=-1} = \{(u_1, u_2)\in (N\times N)_x, \ \theta' (u_1, u_2) = (u_1, u_2)^{-1} = (u_1^{-1}, u_2^{-1})\}.$$
    In other words,
    $$(N\times N)_x^{\theta'=-1} = \{(u_1, u_2) \in (N\times N)_x, \ u_i^* = u_i^{-1} \}.$$
    By the definition of $\psi$, we have 
    $$\psi(u^*) = \psi(u).$$
    
    Denote by $N^{*=-1}$ the set of elements of $u\in N$ such that $u^*=u^{-1}$, and remark that 
    $$\psi(u)=\psi(u^*)=\psi(u)^{-1}, \ \mathrm{for \ all} \ u\in N^{*=-1},$$
    or equivalently 
    $$\psi^2\vert_{N^{*=-1}}\equiv 1, \ \mathrm{i.e.} \ \psi\vert_{N^{*=-1}}\equiv \pm 1.$$
    Let $u\in N^{*=-1}$, thanks to Lemma \ref{antiinv}, there exists $v\in N$, such that 
    $$u= (v^*)^{-1}v.$$
    Then 
    $$\psi(u) =\psi(v^*)^{-1}\psi(v) = 1, $$
    and therefore, 
    $$\psi\vert_{N^{*=-1}}\equiv 1.$$
    
    Finally, 
    $$(N\times N)_x^{\theta'=1} = \{(u_1, u_2) \in (N\times N)_x, \ u_i = u_i^* \} \cong (N'\times N')_x, $$
    and since $x\in  U(n)_{\mathrm{rel}}$, 
    $$\psi\boxtimes \psi^{-1}\vert_{(N\times N)_x^{\theta'=1}} = \psi'\boxtimes (\psi')^{-1}\vert_{(N'\times N')_x}\equiv 1.$$
    This shows that for $x\in U(n)_{\mathrm{rel}}$, we have 
    $$\psi\boxtimes \psi^{-1}\vert_{(N\times N)_x}\equiv 1,$$
    or equivalently $x\in U(n) \cap G_{\mathrm{rel}}$.
\end{proof} 

Now, the Bruhat decomposition for $G$ can be written as 
$$G= N N_G(T) N $$
and therefore 
$$G_{\mathrm{rel}}=NN_G(T)_{\mathrm{rel}}N,$$
where $N_G(T)_{\mathrm{rel}} := N_G(T)\cap G_{\mathrm{rel}}$. We identify the Weyl group $W_G$ of $G$ with the system of representatives given by permutation matrices, so that we have 
$$N_G(T) = T \rtimes W_G.$$
We obtain 
$$N_{U(n)}(T') = N_G(T)\cap U(n).$$
Thanks to the Bruhat decomposition of $U(n)$, a system of representatives for the $N'$-double relevant classes in $U(n)$ is given by $N_G(T)\cap U(n)_{\mathrm{rel}}$. Applying the above lemma, we finally obtain 
$$N_G(T)\cap U(n)_{\mathrm{rel}} = N_G(T)_{\mathrm{rel}}\cap U(n).$$
In other words, 
$$N'\backslash U(n)_{\mathrm{rel}}/N' \cong N_G(T)_{\mathrm{rel}}\cap U(n).$$


Now, we look for an analogous decomposition for $S_{\mathrm{rel}}$. Recall that we consider the following embedding 
$$\begin{array}{cccc}
     & S & \hookrightarrow & G \\
     & Hg & \mapsto & (g^c)^{-1}g 
\end{array}$$
whose image is the set of elements $s\in G $, such that $ss^c = 1$. Under this embedding, the right translation action of $G$ on $S$ becomes 
$$s \cdot g = (g^c)^{-1}sg, \ \ s\in S, \ g \in G.$$
Under this action, we have 
$$S_{\mathrm{rel}} = \{s\in S, \psi\vert_{N_s}\equiv 1\},$$
where 
$$N_s=\{u\in N, \ u^c=sus^{-1}\}.$$

\begin{lem}\label{Srel1}
    We have $S_{\mathrm{rel}} = S \cap G_{\mathrm{rel}}.$
\end{lem}

\begin{proof}
    Let $s\in S$, recall that an element of $(N \times N)_s$ is of form $(u, sus^{-1})$, with $u\in N$, therefore we consider the following embedding
    $$\begin{array}{cccc}
         & N_s & \hookrightarrow & (N\times N)_s\\
         & u & \mapsto & (u, u^c)
    \end{array}.$$
    Let $\theta$ be the involution defined by 
    $$\begin{array}{cccc}
        \theta :  & (N\times N)_s & \to &  (N\times N)_s \\
         & (u_1, u_2) & \mapsto & (u_2^c, u_1^c)
    \end{array}.$$
    Indeed, $\theta$ preserves $(N\times N)_s$, as for $(u_1,u_2)\in (N\times N)_s$, we have 
    $$u_1^{-1}su_2 = s$$
    taking the conjugate and the inverse of this expression, we obtain 
    $$(u_2^c)^{-1}su_1^c = s,$$
    that is $(u_2^c, u_1^c)\in (N\times N)_s$. We apply Lemma \ref{unipdecomp} to $V = (N\times N)_s$ and $\iota = \theta$, we obtain
    $$(N\times N)_s = (N\times N)_s^{\theta=1}(N\times N)_s^{\theta=-1},$$
    where 
    $$(N\times N)_s^{\theta=1} = \{(u_1, u_2)\in (N\times N)_s, \ \theta (u_1, u_2) = (u_1, u_2) \},$$
    and 
    $$(N\times N)_s^{\theta=-1} = \{(u_1, u_2)\in (N\times N)_s, \ \theta (u_1, u_2) = (u_1, u_2)^{-1} = (u_1^{-1}, u_2^{-1})\}.$$
    In other words,
    $$(N\times N)_s^{\theta=-1} = \{(u_1, (u_1^c)^{-1}) \in N\times N \mid \ (u_1^c)^{-1} = su_1s^{-1}\}.$$
    For $(u_1,(u_1^c)^{-1})\in (N\times N)_s^{\theta=-1}$, we have 
    $$\psi(u_1)\psi((u_1^c)^{-1})^{-1} = \psi(u_1)\psi(u_1^c) = \psi(u_1)\psi(u_1)^{-1}=1.$$
    It implies that 
    $$\psi\boxtimes \psi^{-1}\vert_{(N\times N)_s} = \psi\boxtimes \psi^{-1}\vert_{(N\times N)_s^{\theta=1}}.$$
    Finally, 
    \begin{align*}
        (N\times N)_s^{\theta=1} & = \{(u, sus^{-1}) \in (N\times N)_s, \ u^c = sus^{-1}\} \cong N_s
        .
    \end{align*}
    Through this isomorphism, we obtain 
    $$\psi\boxtimes \psi^{-1}\vert_{(N\times N)_s^{\theta = 1}} = \psi\vert_{N_s}^2.$$
    If $s\in S_{\mathrm{rel}}$, we have 
    $$\psi\vert_{N_s} \equiv 1,$$
    and therefore
    $$\psi\boxtimes \psi^{-1}\vert_{(N\times N)_s^{\theta = 1}} \equiv 1.$$
    If $s\in S\cap G_{\mathrm{rel}}$, then
    $$\psi\boxtimes \psi^{-1}\vert_{(N\times N)_s^{\theta = 1}} = \psi\vert_{N_s}^2 \equiv 1.$$
    Since $N_s$ is a unipotent group, we apply Lemma \ref{n-root}, and we get that for all $u\in N_s$, there exists $v\in N_s$ such that $u=v^2$, that is 
    $$\psi(u) = \psi(v)^2 = 1.$$
    Therefore, 
    $$\psi\vert_{N_s} \equiv 1,$$
    meaning that $s\in S_{\mathrm{rel}}$.
    To conclude, we have shown that for all $s\in S$
    $$\psi\boxtimes \psi^{-1}\vert_{(N\times N)_s^{\theta = 1}} \equiv 1$$
    is equivalent to 
    $$\psi\vert_{N_s} \equiv 1,$$
    that is $S_{\mathrm{rel}} = S\cap G_{\mathrm{rel}}$.
\end{proof}

\begin{lem}\label{relS2}
    We have $S\cap G_{\mathrm{rel}} = (N_G(T)_{\mathrm{rel}}\cap S)\cdot N$.     
\end{lem}

\begin{proof}
    Let us first show that $N_G(T)\cap S$ is a system of representatives for the $N$-orbits in $S$. 

    First of all, we know that $N_G(T)$ is a system of representatives for the $N\times N$-orbits in $G$, thanks to the Bruhat decomposition. Under the embedding 
    $$\begin{array}{cccc}
     & S & \hookrightarrow & G \\
     & Hg & \mapsto & (g^c)^{-1}g 
    \end{array},$$
    for an element $s\in S$, we have 
    $$s\cdot N \subset NsN.$$
    Now, for all $g \in N_G(T)$, if $(NgN)\cap S\not = \emptyset$ then $g\in N_G(T)\cap S$. Indeed, let $g \in N_G(T)$, such that $(NgN)\cap S\not = \emptyset$. There exists $(n_1,n_2)\in N\times N$, such that 
    $$n_1^cg^cn_2^c = n_2^{-1}g^{-1}n_1^{-1}.$$
    In other words, 
    $$Ng^cN \cap Ng^{-1}N\not = \emptyset.$$
    Since $N_G(T)$ is a system of representatives for the $N\times N$-orbits in $G$, we obtain 
    $$Ng^cN = Ng^{-1}N,$$
    and therefore $g^c=g^{-1}$.

    It remains to show that for all $g\in N_G(T)$, $(NgN)\cap S$ contains at most one $N$-orbit in $S$. According to the previous observation, it suffices to show that for all $s\in S$, 
    $$NsN\cap S = s\cdot N.$$
    We use again the involution
    $$\begin{array}{cccc}
        \theta :  & N\times N & \to &  N\times N \\
         & (u_1, u_2) & \to & (u_2^c, u_1^c)
    \end{array}.$$
    We have 
    $$(N\times N)^{\theta= 1} = \{(u, u^c), u \in N\}.$$
    In other words, 
    $$NsN\cap S = s\cdot N$$
    is equivalent to 
    $$s\cdot (N\times N) \cap S = s\cdot (N\times N)^{\theta= 1}.$$
    Let $(u_1, u_2)\in N\times N$ such that $s\cdot (u_1,u_2)\in S$, we show that 
    $$(u_1,u_2)\in (N\times N)_s(N\times N)^{\theta =1}.$$
    We have 
    $$(s\cdot (u_1,u_2))^c = (s\cdot (u_1,u_2))^{-1}, $$
    hence 
    $$s^c\cdot (u_1^c,u_2^c) = s^{-1}\cdot (u_2,u_1),$$
    and then, since $s^c = s^{-1}$, 
    $$s\cdot (u_2^c,u_1^c) = s\cdot (u_1,u_2).$$
    This implies that 
    $$s^c\cdot ((u_1,u_2)\theta(u_1,u_2)^{-1}) = s,$$
    and therefore
    $$(u_1,u_2)\theta(u_1,u_2)^{-1} \in (N\times N)_s.$$
    We apply Lemma \ref{antiinv} to $V=(N\times N)_s$ and $\iota =\theta\vert_{(N\times N)_s}$. We have already shown that ${(N\times N)_s}$ is $\theta$-stable. The element $(u_1,u_2)\theta(u_1,u_2)^{-1} \in (N\times N)_s$, satifies
    $$\theta( (u_1,u_2)\theta(u_1,u_2)^{-1}) = \theta (u_1,u_2)(u_1,u_2)^{-1} = ( (u_1,u_2)\theta(u_1,u_2)^{-1})^{-1},$$
    that is $(u_1,u_2)\theta(u_1,u_2)^{-1} \in (N\times N)_s^{\theta =-1}$.
    Hence, there exists an element $v\in(N\times N)_s$, such that 
    $$(u_1,u_2)\theta(u_1,u_2)^{-1} = \theta(v)^{-1}v,$$
    or equivalently
    $$\theta(u_1,u_2) = v^{-1}\theta(v)(u_1,u_2).$$
    Set $x = \theta(v)(u_1,u_2)$, we have 
    $$\theta(x) = v\theta(u_1,u_2) = v(v^{-1}\theta(v)(u_1,u_2)) = \theta(v)(u_1,u_2) = x,$$
    therefore $x\in (N\times N)^{\theta =1}$, and 
    $$(u_1, u_2) = \theta(v)^{-1} x \in (N\times N)_s(N\times N)^{\theta =1}.$$
    So far we have shown that 
    $$s\cdot (N\times N) \cap S \subset s\cdot (N\times N)_s (N\times N)^{\theta= 1},$$
    but $s\cdot (N\times N)_s = s$, thus
    $$s\cdot (N\times N) \cap S \subset s\cdot (N\times N)^{\theta= 1}.$$
    Since the inverse inclusion is direct, we obtain
    $$s\cdot (N\times N) \cap S = s\cdot (N\times N)^{\theta= 1}.$$
    Combining the fact that 
    $$S = (N_G(T)\cap S) \cdot N$$
    with Lemma \ref{Srel1}, we conclude that 
    $$S\cap G_{\mathrm{rel}} = (N_G(T)_{\mathrm{rel}}\cap S)\cdot N. $$
    
\end{proof}

We conclude this section with the following lemma 

\begin{lem}
    We have 
    $$ N_G(T)_{\mathrm{rel}}\cap U(n) = N_G(T)_{\mathrm{rel}}\cap S.$$
    In particular,
    $$N'\backslash U(n)_{\mathrm{rel}}/N' \cong S_{\mathrm{rel}}/N.$$
    
\end{lem}

\begin{proof}
    Recall that 
    $$N_G(T)_{\mathrm{rel}}\cap U(n) = \{n\in N_G(T)_{\mathrm{rel}}, \ w_n{}^t(n^{-1})w_n = n^c \},$$
    and 
    $$N_G(T)_{\mathrm{rel}}\cap S = \{n\in N_G(T)_{\mathrm{rel}}, \ n^{-1} = n^c \}.$$
    According to Remark \ref{Relelements}, an element $n\in N_G(T)_{\mathrm{rel}}$ is of the form 
    $$n= a_Mw_l^Mw_n$$
    where $M$ is a Levi subgroup of $G$, $a_M$ is an element of the center $Z_M$ of $M$ and $w_l^M$ is the longest element in the Weyl group $W_M$ of $M$. Remark that 
    $$w_n{}^tnw_n = w_n(w_nw_l^Ma_M)w_n=a_Mw_l^Mw_n= n.$$
    If, moreover, $n\in N_G(T)_{\mathrm{rel}}\cap U(n)$, then 
    $$n = w_n{}^tnw_n = (n^{c})^{-1} \in  N_G(T)_{\mathrm{rel}}\cap S.$$ 
    Finally, if $n\in N_G(T)_{\mathrm{rel}}\cap S$, then 
    $$w_n{}^tn^{-1}w_n = n^{-1} = n^c,$$
    that is $n\in N_G(T)_{\mathrm{rel}}\cap U(n)$.
    Therefore, 
    $$ N_G(T)_{\mathrm{rel}}\cap U(n) = N_G(T)_{\mathrm{rel}}\cap S.$$
    Thanks to Lemma \ref{relUn}, we have 
    $$N'\backslash U(n)_{\mathrm{rel}}/N' \cong N_G(T)_{\mathrm{rel}}\cap U(n),$$
    that is, 
    $$N'\backslash U(n)_{\mathrm{rel}}/N' \cong N_G(T)_{\mathrm{rel}}\cap S.$$
    Lemma \ref{Srel1} and Lemma \ref{relS2}, provide
    $$S_{\mathrm{rel}}/N = S\cap G_{\mathrm{rel}}/N \cong N_G(T)_{\mathrm{rel}}\cap S,$$
    and therefore
    $$N'\backslash U(n)_{\mathrm{rel}}/N' \cong S_{\mathrm{rel}}/N.$$
\end{proof}

\subsection{Geometric transfer}\label{geotransf}
In this section, let $E/F$ be a quadratic extension of $p$-adic fields, $G= \mathrm{GL}_n(E)$, $H=\mathrm{GL}_n(F)$ and we fix, as in previous sections, $(N, \psi)$ and $(N', \psi')$ Whittaker data for $G$ and $U(n)$, respectively.

\begin{conj}
    Two functions $f\in C_c^\infty(G)$, and $f'\in C_c^\infty(U(n))$ are \emph{transfers} from one to the other, if and only if 
     $$ O_{\psi}(x, f)=O_{\psi'}(y,f'),$$
     for $ y = x \in N_G(T)_{\mathrm{rel}}\cap U(n) = N_G(T)_{\mathrm{rel}}\cap S$.
\end{conj}

\begin{prop}
     If two functions $f\in C_c^\infty(G)$, and $f'\in C_c^\infty(U(n))$ are {transfers} from one to the other, then
     $$O_{\psi}(1,f)=O_{\psi'}(1,f').$$
\end{prop}

\begin{proof}
    Let $f\in C^\infty_c(G)$, recall that the orbital integral on the trivial class is given by
    $$O_\psi(1,f) = \int_{(H\times N)_1\backslash H\times N} f(h^{-1}u)\psi(u)dhdu.$$
    The stabilizer
    $$(H\times N)_1=\left\{(h,n)\in H\times N \vert \ h^{-1}n=1\right\}$$
    can be described as
    $$(H\times N)_1=\left\{(n,n)\in H\times N \vert \ n\in N\cap H\right\}\cong N\cap H.$$
    Therefore, 
    $$O_\psi(1,f) = \int_{(N\cap H)\backslash H} \int_{N}f(h^{-1}u)\psi(u)dudh.$$
    We have defined in Section \ref{WhittSection}, the following function
    $$W_{f, \pi}(g_1, g_2):= W_{f_\pi}(g_1, g_2)= \int_N^* f_\pi(g_1^{-1}ug_2)\psi_n(u)^{-1} du, \ (g_1,g_2)\in G\times G.$$
    Thanks to the Plancherel-Whittaker formula (see Proposition 2.142 in \cite{RBPPlanch}), we obtain
    $$O_\psi(1,f) = \int_{(N\cap H)\backslash H} \int_{\mathrm{Temp}(G)}W_{f, \pi}(h,1)d\mu_G(\pi)dh,$$
    where we use the same measure normalization as in \cite{RBPPlanch}.
    We then use Corollary 3.51 in \cite{RBPPlanch}, to write
    $$O_\psi(1,f) = \int_{\mathrm{Temp_{gen}}(U(n))}\beta(W_{f, BC_n(\pi')}^1)d\mu_{U(n)}(\pi'),$$
    and thus, 
    $$O_\psi(1,f) = \int_{\mathrm{Temp_{gen}}(U(n))}I_{BC_n(\pi')}(f)d\mu_{U(n)}(\pi').$$
    On the unitary side, we have $(N'\times N')_1 \cong N'$, therefore
    \begin{align*}
        O_{\psi'}(1,f') & =\int_{(N'\times N')_1\backslash N'\times N'} f'(n_1^{-1}n_2)\psi'(n_1n_2^{-1})dn_1dn_2, \\
        & = \int_{ N'} f'(n)\psi'(n)dn.
    \end{align*}
    Using again the Plancherel-Whittaker formula, Proposition 2.142 in \cite{RBPPlanch}, we get
    $$O_{\psi'}(1,f') = \int_{\mathrm{Temp_{gen}}(U(n))}W_{f', \pi}(1,1)d\mu_{U(n)}(\pi),$$
    and therefore
    $$O_{\psi'}(1,f') = \int_{\mathrm{Temp_{gen}}(U(n))} J_{\pi}(f')d\mu_{U(n)}(\pi).$$

    Finally, if $f\in C_c^\infty(G)$ is the transfer of $f'\in C^\infty_c(U(n))$, then we have
    $$I_{BC_n(\pi)}(f) = J_\pi(f'),$$
    which, according to the previous calculations, implies 
    $$O_{\psi}(1,f) = O_{\psi'}(1,f').$$

\end{proof}

\pagebreak
\nocite{*}
\bibliographystyle{plain}
\bibliography{GLn_corrige}

\end{document}